\documentclass[reqno, 11pt, reqno]{amsart}

\usepackage{amsfonts, amsthm, amsmath, amssymb}
\usepackage{hyperref}
\hypersetup{colorlinks=false}

\usepackage[margin=3cm]{geometry}

\RequirePackage{mathrsfs}
\let\mathcal\mathscr

\numberwithin{equation}{section}

\newtheorem{theorem}{Theorem}[section]

\newtheorem{lemma}[theorem]{Lemma}

\newtheorem{corollary}[theorem]{Corollary}
\newtheorem{proposition}[theorem]{Proposition}

\theoremstyle{definition}

\newtheorem*{ack}{Acknowledgements}

\newcommand{\e}{\operatorname{e}}
\renewcommand{\tilde}{\widetilde}

\newcommand{\BF}{{\mathbb {F}}}

\newcommand{\BN}{{\mathbb {N}}}

\newcommand{\BP}{{\mathbb {P}}}

\newcommand{\BR}{{\mathbb {R}}}

\newcommand{\BZ}{{\mathbb {Z}}}

\newcommand{\CD}{{\mathcal {D}}}

\newcommand{\CL}{{\mathcal {L}}}
\newcommand{\CM}{{\mathcal {M}}}

\newcommand{\CO}{{\mathcal {O}}}

\newcommand{\CS}{{\mathcal {S}}}

\newcommand{\BZprim}{\BZ_{\operatorname{prim}}}
\newcommand{\deltabad}{\varDelta_{\operatorname{bad}}}

\newcommand{\xx}{\mathbf{x}}

\newcommand{\yy}{\mathbf{y}}

\newcommand{\ww}{{\mathbf{w}}}

\newcommand{\bb}{{\mathbf{b}}}
\newcommand{\cc}{{\mathbf{c}}}
\newcommand{\dd}{{\mathbf{d}}}

\newcommand{\bfx}{\mathbf{x}}

\newcommand{\bfy}{\mathbf{y}}

\newcommand{\bfv}{\mathbf{v}}

\DeclareMathOperator{\Res}{\operatorname{Res}}
\DeclareMathOperator{\Pic}{Pic}
\DeclareMathOperator{\vol}{vol}
\DeclareMathOperator{\Vol}{vol}

\DeclareMathOperator{\rank}{rank}
\DeclareMathOperator{\Eff}{Eff}

\renewcommand{\phi}{\varphi}
\renewcommand{\rho}{\varrho}
\renewcommand{\epsilon}{\varepsilon}

\renewcommand{\leq}{\leqslant}
\renewcommand{\le}{\leqslant}
\renewcommand{\geq}{\geqslant}
\renewcommand{\ge}{\geqslant}
\renewcommand{\bar}{\overline}

\newcommand{\PP}{\mathbb{P}}

\newcommand{\ZZ}{\mathbb{Z}}
\newcommand{\NN}{\mathbb{N}}
\newcommand{\QQ}{\mathbb{Q}}
\newcommand{\RR}{\mathbb{R}}

\newcommand{\ZZp}{\mathbb{Z}_{\text{prim}}}

\newcommand{\cB}{\mathcal{B}}

\newcommand{\cM}{\mathcal{M}}

\newcommand{\x}{\mathbf{x}}
\newcommand{\y}{\mathbf{y}}
\newcommand{\uu}{\mathbf{u}}
\newcommand{\vv}{\mathbf{v}}

\newcommand{\w}{\mathbf{w}}
\newcommand{\DD}{\mathbf{D}}

\newcommand{\ve}{\varepsilon}

\begin{document}
\title[Rational points on a family of quadric bundle threefolds]{Density of rational points on some \\ quadric bundle threefolds}

\author{Dante Bonolis}
\author{Tim Browning}
\author{Zhizhong Huang}
\address{IST Austria\\
Am Campus 1\\
3400 Klosterneuburg\\
Austria}
\email{dante.bonolis@ist.ac.at, tdb@ist.ac.at, zhizhong.huang@ist.ac.at}


\begin{abstract}
	We prove the Manin--Peyre conjecture for the number of rational points of bounded height outside of a thin subset on a family of  Fano threefolds of bidegree $(1,2)$.
\end{abstract}

\date{\today}

\maketitle

\thispagestyle{empty}
\setcounter{tocdepth}{1}
\tableofcontents

\section{Introduction}\label{s:intro}

The primary purpose of this paper is
 to resolve the Manin--Peyre conjecture 
for a family of smooth Fano threefolds in $\BP^1\times \BP^3.$
Let $L_1,\dots,L_4\in \BZ[x_1,x_2]$ be binary linear forms which are pairwise non-proportional. 
Let $V\subset\BP^1\times\BP^3$ be given by 
\begin{equation}\label{eq:quadbundle}
	L_1(x_1,x_2)y_1^2+L_2(x_1,x_2)y_2^2+L_3(x_1,x_2)y_3^2+L_4(x_1,x_2)y_4^2=0,
\end{equation}
defining  a smooth Fano threefold of bidegree $(1,2)$. The  Picard group is $\Pic(V)\cong \ZZ^2$ and $V$ is equipped with two morphisms, corresponding to the 
projections
  $\pi_1:\BP^1\times\BP^3\to\BP^1$ and $\pi_2:\BP^1\times\BP^3\to\BP^3$.
Let $|\cdot|$ denote the sup-norm on $\RR^d$, for any $d\in \NN$.
We can associate an anticanonical height function on $V(\QQ)$ via
$
H(x,y)=|\x||\y|^2,
$
if $(x,y)\in V(\QQ)$ is represented by a vector $(\x,\y)\in \ZZp^2\times \ZZp^4$.
The Manin conjecture \cite{fmt} predicts that there should exist a thin subset $\Omega\subset V(\QQ)$,
the sense of 
Serre  \cite[\S 3.1]{Ser08}, 
 such that 
\begin{equation}\label{eq:manin}
\#\{(x,y)\in V(\QQ)\setminus \Omega: H(x,y)\leq B\} \sim c_V B\log B,
\end{equation}
where $c_V$ is a constant whose value has been conjectured by Peyre \cite{Peyre}.  
Addressing a question of Peyre,  raised in his lecture at the 2009 conference {\em Arithmetic and algebraic geometry of higher-dimensional varieties} at the University of Bristol, numerical evidence towards this conjecture has been supplied by Elsenhans \cite{Elsenhans}.
We can write  \eqref{eq:quadbundle} as
\begin{equation}\label{eq:xQ}
x_1Q_1(y_1,\dots,y_4)=x_2Q_2(y_1,\dots,y_4),
\end{equation}
for suitable diagonal quadratic forms $Q_1,Q_2\in \ZZ[y_1,\dots,y_4]$. Consider the subvariety $Z\subset \PP^3$ given by 
$
Q_1=Q_2=0.
$
Our assumptions on $L_1,\dots,L_4$ ensure
that $Z$ is a smooth genus one curve and so the closed subvariety 
$\pi_2^{-1}(Z)\cong \BP^1\times Z \subset V$ defines 
an elliptic cylinder.  If $Z(\QQ)\neq\emptyset$ then there are $\gg B^2$
rational points of anticanonical height $\leq B$ on 
$\pi_2^{-1}(Z)$. Thus, in general,   we should certainly demand that 
$\Omega$ contains 
$\pi_2^{-1}(Z)(\QQ)$. 

The restriction of $\pi_1$ to $V$ gives a fibration into quadrics $\pi_1:V\to \PP^1$.  
If $x\in \PP^1(\QQ)$ is represented by $\x=(x_1,x_2)\in\BZprim^2$, then $\pi_1^{-1}(x)$ is split if and only if 
\begin{equation}\label{eq:square}
\prod_{i=1}^{4}L_i(\x)=\square.
\end{equation}
For such a point $x$ the fibre will contribute $\sim c_x B\log B$ points, as $B\to \infty$, for an appropriate constant $c_x>0$ that depends on $x$. Based on numerical investigation, 
Elsenhans  suggests that the  conjectured asymptotic  \eqref{eq:manin}
 holds when $\Omega$ is taken to be union of 
$\pi_2^{-1}(Z)(\QQ)$ and the set of $(x,y)\in V(\QQ)$ for which 
\eqref{eq:square} holds. In particular, 
 $\Omega\subset V(\QQ)$ is  a thin set of rational points, since 
 $\pi_2^{-1}(Z)(\QQ)$  lies on a proper subvariety of $V$ and the set of rational points satisfying 
\eqref{eq:square} correspond to rational points on a double covering.
Note that  \eqref{eq:square} defines an elliptic curve  $E\subset \PP(2,1,1)$ and so   $\Omega$ is Zariski dense in $V$ if $E$ has positive rank.  In fact, as discussed by Skorobogatov \cite[\S~3.3]{skoro}, 
$E$ is the Jacobian of the genus 1 curve $Z$.
An explicit example is given by taking $V$ to be 
\begin{equation}\label{eq:example}
x_1y_1^2+x_2y_2^2+(x_1+2x_2)y_3^2+(x_1+6x_2)y_4^2=0,
\end{equation}
for which  $Z(\RR)=\emptyset$ and    $E(\QQ)\cong (\ZZ/2\ZZ)^2 \times \ZZ$ (with 
Cremona label 192a2).
Our main result settles the thin set version of the  Manin--Peyre conjecture under mild assumptions on $V$.

\begin{theorem}\label{t:main}
Assume that  $L_1,\dots,L_4\in \BZ[x_1,x_2]$ are
pairwise non-proportional
 linear forms, each with 
 coprime 
 coefficients,   such that 
 $Z(\RR)=\emptyset$.
Then, if   $\Omega$ is the set of 
 $(x,y)\in V(\QQ)$ for which 
\eqref{eq:square} holds, we  have 
$$
\#\{(x,y)\in V(\QQ)\setminus \Omega: H(x,y)\leq B\} \sim c_V B\log B,
$$
where $c_V$ is the constant predicted by Peyre \cite{Peyre}.
\end{theorem}

The example \eqref{eq:example} satisfies the assumptions of the theorem. 
With further work, it is possible to remove the 
 hypotheses that each $L_i$ has coprime coordinates and 
 $Z(\RR)=\emptyset$. The latter is a particularly convenient assumption that allows 
 us to assume that the regions we work in are not too lopsided. 
Since these hypotheses  simplify an already lengthy argument, we have chosen not to attempt their removal  here. 

As discussed in \cite{brian2}, 
a geometric framework for
identifying  problematic  thin sets in the Manin--Peyre conjecture has been developed by 
 Lehmann, Sengupta and Tanimoto.
In private communication with the authors, Professor Tanimoto has  indicated that similar arguments to those in \cite[\S~12]{brian2} show  that the thin set $\Omega$ in Theorem \ref{t:main} agrees with their prediction.

Manin \cite{manin} used height machinery to establish a lower bound supporting  linear growth for all smooth Fano threefolds, possibly after an extension of the ground field. More recently, 
Tanimoto \cite{sho} has produced a range of upper bounds for various classes of Fano threefolds, but his work does not cover 
\eqref{eq:quadbundle}. 
The classification of 
Fano threefolds with  Picard number $2$ goes back to Mori and Mukai \cite{MM}, but it is convenient to appeal to the summary 
 of 
Iskovskikh and Prokhorov  \cite[Table~12.3]{isk}. 
Over $\overline{\mathbb{Q}}$, there are 
36 isomorphism types of  Fano threefold of Picard number $2$.
The expectation is that the arithmetic of these varieties becomes harder the higher up the table they appear.
As explained in Remark (i) before  \cite[Table~12.3]{isk}, varieties numbered 33--36 are toric. Thus, 
for each of these, the Manin--Peyre conjecture follows from work of Batyrev and Tschinkel \cite{BT}.
Equivariant
compactifications of the additive group $\mathbb{G}_a^3$ are also known to satisfy the Manin--Peyre conjecture, thanks to work of Chambert-Loir and Tschinkel \cite{ct}, and 
the  smooth Fano threefolds that arise 
as equivariant
compactifications of $\mathbb{G}_a^3$ have been identified by 
 Huang  and Montero
\cite{HM}. These cover varieties numbered 28, 30, 31  and 33--36.
Variety number 32 is a bilinear  hypersurface in $\PP^2\times\PP^2$, which is 
a flag variety and so covered by \cite{fmt}. 
Finally, 
in a recent tour de force \cite{blomer,blomer2},  Blomer, Br\"udern, Derenthal and Gagliardi 
have shown that the Manin--Peyre conjecture holds
for many  spherical Fano threefolds 
of semisimple rank one. Among those of Picard number $2$, this covers a quadric in $\PP^4$ blown-up along a conic, which corresponds to variety number 29
in \cite[Table~12.3]{isk}. 
Our variety $V\subset \PP^1\times\PP^3$ 
has Picard number $2$ and 
can be viewed as the blow-up 
of  $\BP^3$ along the genus one curve $Z$, corresponding to  variety number 25.
In particular it is neither spherical, nor  toric, nor an 
equivariant
compactification of $\mathbb{G}_a^3$.

The proof of Theorem \ref{t:main} is inspired  by recent  work by Browning and Heath-Brown \cite{duke}, which resolves the Manin--Peyre conjecture for the Fano fivefold
 \begin{equation}\label{eq:duke}
 x_1y_1^2+ x_2y_2^2+ x_3y_3^2+ x_4y_4^2=0,
\end{equation}
of bidegree $(1,2)$ in $\PP^3\times \PP^3$. 
In this setting, the anticanonical height is $|\x|^3|\y|^2$, if $(x,y)$ is represented by 
$(\x,\y)\in \ZZp^4\times \ZZp^4$.
 The basic line of attack in  \cite{duke} involves counting points  as a union of
planes when $|\y|\leq B^{\frac{1}{4}}$, and as a union of quadrics when
$|\x|\leq B^{\frac{1}{6}-\eta}$, for any fixed $\eta>0$. 
 In the first case, geometry of numbers arguments are used to count the relevant vectors $\x$, whereas the circle method underpins the second case.
In terms of 
the inequality $|\x|^3|\y|^2\leq B$, this leaves 
a small range 
uncovered, for which it is necessary to have an upper bound of the
correct order of magnitude.

Our work follows a similar strategy, but with  substantial extra difficulties. 
This is reflected in the geometry of the 
effective cone  of divisors $\Eff_V$. 
Let  $H_1=\pi_1^{*}\CO_{\PP^1}(1)$ and 
$H_2=\pi_2^{*}\CO_{\PP^3}(1)$. As discussed by Ottem \cite[Thm.~1.1]{Ottem}, the effective cone is
$\Eff_V= \BR_{\geqslant 0}H_1+\BR_{\geqslant 0}C$,
where  $C=-H_1+2H_2$ is the class of the exceptional divisor $\pi_2^{-1}(Z)$.
 This is   larger than  the nef cone
$\operatorname{Nef}_V=\BR_{\geqslant 0} H_1+\BR_{\geqslant 0} H_2$, 
meaning that 
  $\Eff_V^{\vee}$ is smaller than 
$\operatorname{Nef}_V^\vee$, which strongly 
influences the asymptotic behaviour in 
Theorem \ref{t:main}. This is in stark contrast to the situation in 
\eqref{eq:duke}, where the effective cone is equal to the nef cone.
When $|\x|\leq B^{\frac{1}{4}}$, we 
view $V$ as a family of quadrics via \eqref{eq:quadbundle}. We will then reapply the circle method arguments worked out in \cite{duke}, 
but 
we shall face  extra challenges in dealing with 
a family over $\PP^1$ rather than over $\PP^3$.
When $|\x|\geq B^{\frac{1}{4}+\eta}$, we  view $V$ as being given by 
\eqref{eq:xQ}, which we can use to eliminate $x_1$ and $x_2$, on extracting common 
divisors. This ultimately leads to a
counting problem of the form
$$
\sum_{d\leq D} \#\left\{\y\in \ZZp^4: |\y|\leq Y, ~Q_1(\y)\equiv Q_2(\y) \equiv 0 \bmod{d}\right\},
$$
for $D,Y\geq 1$. On interpreting the inner sum as a disjoint union of lattice conditions, the main work  is to show that the successive minima have the expected order of magnitude,  as one averages over $d$ and over the different lattices. This counting problem is  rather different to the one appearing \cite{duke}, 
and it  seems likely the methods developed could be useful in the study of other quadric bundles, including  del Pezzo surfaces with a conic bundle structure.  It is in this part of the argument that the 
difference between 
$\Eff_V$ and $\operatorname{Nef}_V$ manifests itself.  If
$|\x|\geq B^{\frac{1}{4}+\eta}$ and $|\x||\y|^2\leq B$, then we are only interested in the range $|\y|\leq B^{\frac{3}{8}-\frac{\eta}{2}}$. However, it turns out that the contribution from $|\y|\leq B^{\frac{1}{4}}$ is negligible, which thereby 
reduces the size  of the leading constant.  In \eqref{eq:duke}, by comparison, 
the contribution from $|\y|\leq B^\delta$ makes a positive proportional contribution for any $\delta>0$.

Finally, we  are left with a small range to cover via an auxiliary upper bound of the correct order.
The  completely diagonal structure of \eqref{eq:duke} renders it easier  to obtain  the  necessary upper bound via a modification of Hua's inequality. Lacking this diagonal structure, our approach involves an 
array of inputs, from 
character sum estimates and point counting on $Z$ modulo prime powers, to 
various circle method applications \cite{BrowningM, duke} and a general upper bound
for the number of rational points of bounded height on diagonal quadric surfaces \cite{DA}.


\begin{ack}
The authors are grateful to Florian Wilsch for useful comments. 
While working on this paper the  authors were
supported by FWF grant P~32428-N35.  
\end{ack}

\section{Roadmap of the proof}\label{s:road}

We proceed to summarise  some of the  steps in the proof of 
Theorem~\ref{t:main}. 
We shall frequently switch between the  representation of $V$ given by \eqref{eq:quadbundle}, involving the pairwise  non-proportional linear forms
$L_1,\dots,L_4\in \BZ[x_1,x_2]$ with coprime coefficients, and the representation
 \eqref{eq:xQ}, involving diagonal quadratic forms $Q_1,Q_2\in \ZZ[y_1,\dots,y_4]$ such that 
  $Z(\RR)=\emptyset$. 
All of the estimates in our work are allowed to depend implicitly on the coefficients of the polynomials $L_1,\dots,L_4$. Any other dependence will be indicated by an appropriate subscript.  

For each $1\leq i\leq 4$, 
we may  assume that 
$
L_i(x_1,x_2)=a_ix_1+b_ix_2, 
$
for $a_i,b_i\in \ZZ$ such that $\gcd(a_i,b_i)=1$.
Define the set of primes
$
\mathcal{P}=\{p: p\mid (a_{i}b_{j}-a_{j}b_{i})\text{ for some } i\neq j\}\cup\{2\}.
$
Then  $\mathcal{P}$ is a finite set, since $L_i$ and $L_j$ are assumed to be non-proportional, for distinct $1\leq i,j\leq 4$. It will be convenient to set
\begin{equation}\label{eq:DELTA}
\Delta=\prod_{p\in \mathcal{P}}p.
\end{equation}
Our  diagonal quadratic forms take the shape
\begin{equation}\label{eq:Q1Q2}
Q_{1}(\y)=\sum_{i=1}^{4}a_{i}y_{i}^{2},\quad Q_{2}(\y)=-\sum_{i=1}^{4}b_{i}y_{i}^{2}.
\end{equation}
Moreover, $Q_1=Q_2=0$ defines a smooth genus $1$ curve  $Z\subset \PP^3$ such that 
$Z(\RR)=\emptyset$.

Let $N_V(\Omega,B)$ denote the counting function whose asymptotics we are trying to determine.
We shall avoid the set $\Omega$ by stipulating that 
\begin{equation}\label{eq:notsquare}
\prod_{i=1}^{4}L_i(\x)\neq \square,
\end{equation}
for any rational point $(x,y)\in V(\QQ)$ that is to be counted.
On taking into account the action of the units in $\PP^1(\QQ)\times\PP^3(\QQ)$, we have 
\begin{equation}\label{eq:step1}
N_V(\Omega,B)=\frac{1}{4}\#\left\{
(\xx,\yy)\in\BZprim^2\times\BZprim^4:
\begin{array}{l}
\text{\eqref{eq:quadbundle} and \eqref{eq:notsquare} hold}\\ 
|\xx||\yy|^2\leqslant B
\end{array}
\right\}.
\end{equation}
Following the line  of attack in \cite{duke},
we will use different techniques  to estimate the size of $N_V(\Omega,B)$, according to the relative sizes of $|\x|$ and $|\y|$. 
When $|\x|$ is small,  we will  fix $\x$ and use the circle method to estimate the number of $\y$. In fact the relevant application of the circle method carried out in \cite[\S~4]{duke} is directly in a form that can be applied to our own setting. 
On the other hand,
when $|\x|$ is large, we will use \eqref{eq:xQ} to eliminate $\x$ and reduce to a problem about counting integer vectors which reduce to $Z$ modulo $d$, for varying moduli $d$.
There   remains an annoying middle range  which requires a sufficiently sharp upper bound.

Let
\begin{equation}\label{eq:MX}
	\cM(X,Y)=\left\{(\xx,\yy)\in\BZprim^2\times\BZprim^4:
	\begin{array}{l}
	\text{\eqref{eq:quadbundle} and \eqref{eq:notsquare} hold}\\ 
	X\leq |\xx|< 2X,~|\yy|\leq  Y\end{array}\right\},
\end{equation}
for  $X,Y\geqslant 1$. Since $Z(\RR)=\emptyset$ it is easy to deduce from the alternative representation \eqref{eq:xQ} of \eqref{eq:quadbundle} that 
$\cM(X,Y)$ is empty unless $X\ll Y^2$ for a suitable implied constant depending only on $V$.
In Section \ref{s:upper} we shall prove the following general upper bound for 
the cardinality of $\cM(X,Y)$.

\begin{theorem}\label{t:upper}
Let  $X,Y\gg 1$.  Then
$$
\#\cM(X,Y)\ll  XY^2
+ \min\left\{X^2Y^{4/3},   X^{\frac{1}{4}}Y^{\frac{5}{2}}
\frac{\log Y}{\log\log Y}\right\}.
$$
\end{theorem}

On breaking the ranges for $|\x|$ and $|\y|$ into dyadic intervals, Theorem \ref{t:upper}  easily implies the optimal upper bound $
N_V(\Omega,B)=O( B\log B).
$
In fact, not only does it help cover an awkward  range for the relative sizes of $|\x|$ and $|\y|$, but 
certain steps in the proof of Theorem \ref{t:upper} also  play a vital role in the proof of the asymptotic formula, where it used to show that certain  lattices are rarely  lopsided.

Let us now summarise the proof of Theorem \ref{t:main} in a little more detail. Let
\begin{equation}\label{eq:L}
\CL(B)=
\left\{
(\xx,\yy)\in\BZprim^2\times\BZprim^4:
\begin{array}{l}
\text{\eqref{eq:quadbundle} and \eqref{eq:notsquare} hold}\\ 
|\xx||\yy|^2\leqslant B
\end{array}
\right\},
\end{equation}
so that $N_V(\Omega,B)=\frac{1}{4}\#\CL(B)$ in \eqref{eq:step1}. 
We will decompose $\CL(B)$ into three sets
\begin{equation}\label{eq:cl1-3}
\CL(B)=\CL_1(B)\sqcup \CL_2(B)\sqcup \CL_3(B),
\end{equation}
where
\begin{align*}
\CL_1(B)&=\left\{(\xx,\yy)\in \CL(B): B^{\frac{1}{4}+\eta}\leq |\x|\leq 
B^{\frac{1}{2}-\eta} 
\right\},\\
\CL_2(B)&=\left\{(\xx,\yy)\in \CL(B):|\x|\leq B^\frac{1}{4}\right\},\\
\CL_3(B)&=\left\{(\xx,\yy)\in \CL(B):B^\frac{1}{4}<|\x|<B^{\frac{1}{4}+\eta} \text{ or }
 |\x|> 
B^{\frac{1}{2}-\eta} 
 \right\},
\end{align*}
for any $\eta>0$.  The parameter $\eta$ will ultimately be taken to be arbitrarily small, but it is fixed 
at each appearance. 
We now reveal our estimates for the cardinality of these sets.

The treatment of $\#\CL_1(B)$ rests on 
rewriting \eqref{eq:quadbundle} as \eqref{eq:xQ} and then appealing to the geometry of numbers. 
In order to record our result, we  first need to define some auxiliary  quantities.
For any prime $p$ and $a\in \NN$, we  can define an equivalence relation on 
$\left(\mathbb{Z}/p^a\mathbb{Z}\right)^{4}$, by saying $\mathbf{u}$ is equivalent to $\mathbf{v}$
 if and only if there exists $\lambda\in\left(\mathbb{Z}/p^a\mathbb{Z}\right)^{\times}$ such that $\lambda\uu\equiv \vv \bmod{p^a}$. An important role will be played in our work by the 
set  of equivalence classes
$$
V_{p^a}^{\times}=\{\uu\in\left(\mathbb{Z}/p^a\mathbb{Z}\right)^{4}: p\nmid \uu,~Q_{1}(\uu)\equiv Q_{2}(\uu)\equiv 0 \bmod{p^a}\}/\left(\mathbb{Z}/p^a\mathbb{Z}\right)^{\times}.
$$
We may now define
\begin{equation}\label{eq:series}
\mathfrak{S}_1=\prod_{p}\left(1-\frac{1}{p}\right)
\left(1-\frac{1}{p^4}+\left(1-\frac{1}{p}\right)^2\sum_{a=1}^{\infty}\frac{\#V_{p^{a}}^{\times}}{p^{2a}}
\right).
\end{equation}
The absolute convergence of  $\mathfrak{S}_1$ is ensured 
by Corollary \ref{cor:stone}, which implies that 
$$
1-\frac{1}{p^4}+\left(1-\frac{1}{p}\right)^2\sum_{a=1}^{\infty}\frac{\#V_{p^{a}}^{\times}}{p^{2a}}
=1+\frac{1}{p}+O\left(\frac{1}{p^{\frac{3}{2}}}\right),
$$
if  $p\nmid \Delta$.
We can now state our first asymptotic formula, which will be the object of Section \ref{sec:geometry}

\begin{proposition}\label{pro:L1}
Let $\eta>0$. There exists absolute constants $c_2>c_1>0$ such that 
$$
\#\CL_1(B)=2
\mathfrak{S}_1 B 
\int_{\substack{\y\in \RR^4\\ 
c_1B^{\frac{1}{4}+\frac{\eta}{2}}
\leq |\y|<
c_2B^{\frac{3}{8}-\frac{\eta}{2}}}}
 \frac{\mathrm{d}\y}{|\y|^2\max(|Q_1(\y)|,|Q_2(\y)|)}
+O_\eta\left(B\sqrt{\log B}\right).
$$
\end{proposition}

Next, we show that  $\#\CL_2(B)$ can be estimated asymptotically, as $B\to \infty$.
This will be achieved using the Hardy--Littlewood circle method.
Let 
\begin{equation}\label{eq:tau-inf}
\tau_\infty= \int_{-\infty}^\infty \int_{[-1,1]^6} \operatorname{e}\left(\theta\left(L_1(\x)y_1^2+\cdots + L_4(\x)y_4^2\right)\right) \mathrm{d}\x \mathrm{d}\y \mathrm{d}\theta
\end{equation}
and 
\begin{equation}\label{eq:defAq}
A_q=\sum_{\substack{a\bmod{q}\\ \gcd(a,q)=1}}
 \sum_{\mathbf{b}\in (\ZZ/q\ZZ)^4}\sum_{\substack{\cc\in (\BZ/q\BZ)^2\\ \gcd(q,\cc)=1}} \operatorname{e}_q\left(a\sum_{i=1}^{4}L_i(\cc)b_i^2\right),
\end{equation}
for any $q\in \NN$, where $\e_q=\e(\frac{\cdot}{q})$. We may then define
\begin{equation}\label{eq:def-SS2}
\mathfrak{S}_2=\sum_{q=1}^\infty \frac{A_q}{q^6}\prod_{p\mid q}\left(1-\frac{1}{p^{2}}\right)^{-1}.
\end{equation}
The convergence of $\tau_\infty$ and $\mathfrak{S}_2$ are  established in
\eqref{eq:tau-inf-upper} and \eqref{eq:SS2}, respectively. 
We may now record the following result, which will be proved  in Section \ref{sec:circle-a}.

\begin{proposition}\label{pro:L2}
Let $\eta>0$. Then 
$$
\#\CL_2(B)=\frac{\tau_\infty\mathfrak{S}_2}{4\zeta(2)^2} B\log B +O(\eta^{\frac{1}{2}}\log B) +O_\eta(1).
$$
\end{proposition}

Finally, for $\#\CL_3(B)$ we shall produce the following upper bound.

\begin{proposition}\label{pro:L3}
Let $\eta>0$. Then 
$$
\#\CL_3(B)=
O\left( \eta B\log B +\frac{B\log B}{\log\log B}\right).
$$
\end{proposition}

\begin{proof}
Recalling \eqref{eq:MX}, have already remarked that  $\#\CM(X,Y)=0$ unless $X\ll Y^2$, which implies that $X\ll B^{\frac{1}{2}}$.
Let $\mathcal{X}_1$ denote the set of $(X,Y)\in \NN^2$,  where $X,Y$
 run over non-negative powers of two such that $XY^{2}\ll  B$ and
 $
 B^{\frac{1}{4}}\ll X\ll  B^{\frac{1}{4}+\eta}. 
$
Similarly, let $\mathcal{X}_2$ denote the corresponding set in which the final inequalities are replaced by 
$
B^{\frac{1}{2}-\eta}\ll X\ll B^{\frac{1}{2}}.
$
Then it follows that 
\[
\#\CL_3(B)\leq 
\sum_{\substack{(X,Y)\in \mathcal{X}_1}}
\#\mathcal{M}(X,Y)+ 
\sum_{\substack{(X,Y)\in \mathcal{X}_2}}
\#\mathcal{M}(X,Y).
\]
But Theorem \ref{t:upper} implies that 
\begin{align*}
\sum_{\substack{(X,Y)\in \mathcal{X}_1}}
\#\mathcal{M}(X,Y)
&\ll
\sum_{\substack{(X,Y)\in \mathcal{X}_1}}
\left(XY^2
+ X^{\frac{1}{4}}Y^{\frac{5}{2}}
\frac{\log Y}{\log\log Y}\right)\\
&\ll B\sum_{X}1 +
\frac{B^{\frac{5}{4}}\log B}{\log\log B}
\sum_X 
\frac{1}{X} \\
&\ll \eta B\log B +\frac{B\log B}{\log\log B},
\end{align*}
on summing over  $X$ and $Y$.
Similarly, we obtain
\begin{align*}
\sum_{\substack{(X,Y)\in \mathcal{X}_2}}
\#\mathcal{M}(X,Y)
&\ll \eta B\log B+B.
\end{align*}
The statement of the lemma follows.
\end{proof}

We shall combine these estimates in Section \ref{sec:finale}, which is where the proof of Theorem~\ref{t:main} will be drawn to a close.

\subsection*{Notation}

For any $D>0$, we shall write $d\sim D$ to  mean $D\leq d<2D$, and we shall write $d\asymp D$ 
to mean that there exist constants $c_2>c_1>0$ (depending only on the linear forms $L_1,\dots,L_4$ in \eqref{eq:quadbundle}) such that 
 $c_{1}D\leq d \leq c_{2}D$.
Moreover,  we shall often adopt the notation 
$$
X_1 \swarrow S\nearrow X_2
$$ 
within a sum, in order to denote that a dyadic parameter $S$ runs over an interval $X_1\ll S\ll X_2$, 
with  implied constants depending only on the setting. We shall write 
$S\nearrow X$ to mean that the dyadic parameter $S$ runs over an interval $S\ll X$.

\section{Preliminary technical results}\label{s:pre}

\subsection{Character sum estimates}

The following result is a straightforward consequence of combining partial summation 
with  Burgess bound \cite{burgess2}.
 
 \begin{lemma}[Burgess]\label{le:Burgess}
 Let $\chi$ be a non-principal Dirichlet character modulo $q$, let  $\theta>\frac{3}{16}$ and let 
$\sigma\geqslant 1$. Then 
$$
		\sum_{n>N}\frac{\chi(n)}{n^\sigma}\ll_\theta N^{\frac{1}{2}-\sigma}q^\theta.
$$
\end{lemma}


We will also require a generalisation of  the P\'olya--Vinogradov bound that involves products of linear polynomials.

\begin{lemma}\label{lem:PV-prep}
Let
$J_1,\dots,J_{k}\in \ZZ[x_1,x_2]$ be pairwise non-proportional  linear forms. Given $n\in \ZZ$ we write  
$$
J_{i,n}(x)=J_i(n,x) \in \ZZ[x],
$$ 
for $1\leq i\leq k$.
Let $\dd=(d_{1},\dots,d_{k})\in\BN^{k}$
and put $D=d_1\cdots d_{k}$.
Suppose that there exists $A\in \NN$ such that 
$\gcd(d_i,d_j)\mid A$, for all  $i\neq j$. 
Let $r\in\BN$ be square-free, let  $q\in \NN$, let $a\in \ZZ/q\ZZ$ and let 
$I\subset \RR$ be an interval. 
Assume that  $\gcd(r,qD)=1$. 
Then for any $n\in \ZZ$ we have 
$$
\sum_{\substack{x\in I\cap \ZZ\\
x\equiv a\bmod{q}\\
d_i\mid J_{i,n}(x)}}
\left(\frac{J_{1,n}(x)\cdots J_{k,n}(x)}{r}\right)\ll_\ve
 \left(
		\frac{\vol(I)}{r^{\frac{1}{2}}[q, D]} 
+				r^{\frac{1}{2}} \log(qD)	\right)
				r^{\ve}\gcd(rD,n),
$$
for any $\ve>0$, 
where the implied constant is only allowed to depend on $A, J_1,\dots,J_{k}$ and the choice of $\ve$.
\end{lemma}

\begin{proof}
Let $\Sigma_n(I)$ denote the sum that we are trying to estimate.
Suppose that 
$
J_i(x_1,x_2)=a_ix_1+b_ix_2$ for $1\leq i\leq k$, 
for $a_i,b_i\in \ZZ$.
Put 
$$
\Delta_0= \prod_{1\leq i<j\leq k}|a_ib_j-a_jb_i|.
$$
Then $\Delta_0\neq 0$ since the linear forms are assumed to be non-proportional. 
Put
$$
G(x)=J_{1,n}(x)\cdots J_{k,n}(x)
$$
and let 
  $D=d_{1}\cdots d_{k}$. Recalling that $\gcd(r,qD)=1$, we may complete the sum by breaking into residue classes modulo $r[q,D]$. This yields
\begin{equation}\label{eq : pv3}
	\begin{split}
		\Sigma_n(I)&=\sum_{\substack{y\bmod r[q,D]\\ 
		y\equiv a\bmod{q}\\
		d_i\mid J_{i,n}(y)}}
		\left(\frac{G(y)}{r}\right)
		\sum_{\substack{x\in I\cap \ZZ}}
		\frac{1}{r[q,D]}\sum_{\alpha\bmod r[q,D]}
		\operatorname{e}_{r [q,D]}(\alpha(y-x))\\
		&= \frac{1}{r[q, D]}\sum_{\alpha\bmod r[q,D]}S(\alpha)
		\sum_{\substack{x\in I\cap \ZZ}}\operatorname{e}_{r[q,D]}(-\alpha x),
	\end{split}
\end{equation}
where 
$$
 	S(\alpha)=
	\sum_{\substack{y\bmod r[q,D]\\y\equiv a\bmod{q}\\
		d_i\mid J_{i,n}(z)}}\left(\frac{G(y)}{r}\right)\operatorname{e}_{r[q,D]}(\alpha y).
$$
Any  $\alpha\in\BZ/r[q,D]\BZ$
has a representative satisfying $|\alpha|\leq \frac{1}{2}r[q,D]$. 
Thus we  clearly have 
\begin{equation}\label{eq:geom}
\left|\sum_{\substack{x\in I\cap \ZZ}}\operatorname{e}_{r[q,D]}(-\alpha x)\right|\ll\begin{cases}
	\vol(I)& \text{ if } \alpha=0,\\ 
	\frac{r[q,D]}{|\alpha|} &\text{ if } \alpha\neq 0.
\end{cases}
\end{equation}

We now proceed with a detailed study of  $S(\alpha)$.
Since $\gcd(r,[q,D])=1$,  any $y\bmod r[q,D]$ can be decomposed as
$
	y= y_{1}[q,D]\overline{[q,D]} +y_{2}r\bar{r},
$
for  $y_{1}\bmod r$ and $y_{2}\bmod [q,D]$.
(Here, $\overline{[q,D]}\in \ZZ$ is the multiplicative inverse of  $[q,D]$ modulo $r$, and  $\overline{r}\in \ZZ$ is the inverse of  $r$ modulo $ [q,D]$.) 
Under this change of variables we obtain
\begin{align*}
	S(\alpha)&=
T(\alpha;r)	
		\sum_{\substack{y_2\bmod [q,D]\\
		y_2\equiv a \bmod{q}\\
		d_i\mid J_{i,n}(y_2)}}
\operatorname{e}_{D}(\alpha y_{2}\overline{r}),
\end{align*}
where
$$
T(\alpha;r)=	\sum_{y_{1}\bmod r}\left(\frac{G(y_1)}{r}\right)
	\operatorname{e}_{r }(\alpha y_{1}\overline{[q,D]} ).
$$
But then it follows that 
\begin{equation}\label{eq:white-flag}
|S(\alpha)|\leq |T(\alpha;r)| N(q;\mathbf{d}),
\end{equation}
where
$$
N(q;\mathbf{d})=
\#\left\{
y_2\bmod [q,D]:
		y_2\equiv a \bmod{q}, ~
		d_i\mid J_{i,n}(y_2) \text{ for $1\leq i\leq k$}\right\}.
$$
It remains to estimate $T(\alpha;r)$ and $N(q,\mathbf{d})$.

We begin by estimating $T(\alpha;r)$ for a square-free integer $r\in \NN$.
We shall prove that 
\begin{equation}\label{eq:white-T}
T(\alpha;r)\ll_\ve r^{\frac{1}{2}+\ve}\gcd(r,n)^{\frac{1}{2}},
\end{equation}
for any $\ve>0$, uniformly in $\alpha\in \ZZ/r\ZZ$.
By multiplicativity, it will  suffice to prove that 
\[
\sum_{y\bmod p}\left(\frac{G(y)}{p}\right)\operatorname{e}_{p}(\alpha y)\ll 
\begin{cases}
\sqrt{p} & \text{ if $p\nmid n$,}\\
p & \text{ if $p\mid n$,}
\end{cases}
\]
for any prime $p\mid r$, 
where the implied constant only depends on  $\Delta_0$. 
The result is trivial if $p\mid \Delta_0$ and so we can assume that $p\nmid \Delta_0$. But then it follows that the two linear polynomials  $a_in+b_iT$ and $a_jn+b_jT$ are non-proportional modulo $p$ if and only if $p\nmid n$. Thus $G$ is separable modulo $p$ if and only if $p\nmid n$. If $p\mid n$ we take the trivial bound for the exponential sum. If $p\nmid n$, on the other hand, the desired bound follows from  
Weil's resolution of the Riemann hypothesis for curves \cite{weil}.
 This completes the proof of \eqref{eq:white-T}.

Turning to $N(q;\mathbf{d})$ we shall prove that 
\begin{equation}\label{eq:white-U}
N(q;\mathbf{d})=O(\gcd(D,n)),
\end{equation}
for an implied constant that is only allowed to depend on $A$ and $J_1,\dots,J_{k}$.
Before doing so, let us see how it suffices to complete the proof of the lemma. 
Combining it with \eqref{eq:white-T} in \eqref{eq:white-flag}, we deduce that 
$$
S(\alpha)\ll_\ve r^{\frac{1}{2}+\ve}\gcd(r,n)^{\frac{1}{2}} \gcd(D,n)\leq 
r^{\frac{1}{2}+\ve}\gcd(rD,n),
$$
since $r$ and $D$ are coprime. 
Once coupled with \eqref{eq:geom} in \eqref{eq : pv3}, we are finally led to the bound 
\begin{align*}
		\Sigma_n(I)
		&\ll_\ve 
		\frac{1}{r[q, D]} 
		\left(\vol(I)+ 
				\sum_{0<|\alpha|\leq \frac{1}{2}r[q,D]} \frac{r[q,D]}{|\alpha|}\right)
				r^{\frac{1}{2}+\frac{\ve}{2}}\gcd(rD,n)\\
		&\ll_\ve \left(
		\frac{\vol(I)}{r^{\frac{1}{2}}[q, D]} 
+				r^{\frac{1}{2}} \log(qD)	\right)
				r^{\ve}\gcd(rD,n),
\end{align*}
as claimed in the lemma.

Returning to \eqref{eq:white-U}. It suffices to examine the case
$
N_p=N(p^\alpha;p^{\beta_1},\dots,p^{\beta_{k}})$, for any prime $p$, by the Chinese remainder theorem.
We may suppose without loss of generality that $\alpha\geq 0$ and $0\leq\beta_1\leq \cdots\leq \beta_{k}$. It then follows from the hypotheses of the lemma that $p^{\beta_{k-1}}\mid A$. 
We now have 
$$
N_p
\leq
\#\left\{
y\bmod{p^{\max(\alpha,\beta_1+\cdots+\beta_{k})}}:
		y\equiv a \bmod{p^\alpha}, ~
		b_{k}y\equiv -a_{k}n\bmod{p^{\beta_{k}}}
\right\}.
$$
If $\alpha\geq \beta_1+\cdots+\beta_{k}$,  then we  trivially have  $N_p\leq 1$.
If $\alpha< \beta_1+\cdots+\beta_{k}$, then
$$
N_p
\leq p^{\beta_1+\cdots+\beta_{k-1}}
\#\left\{
y\bmod{p^{\beta_{k}}}:
		b_{k}y\equiv -a_{k}n\bmod{p^{\beta_{k}}}
\right\}.
$$
Our remark above shows that 
$p^{\beta_1+\cdots+\beta_{k-1}}\leq p^{(k-1)v_p(A)}$.
Moreover, 
$$
\#\left\{
y\bmod{p^{\beta_{k}}}:
		b_{k}y\equiv -a_{k}n\bmod{p^{\beta_{k}}}
\right\}\leq \gcd(p^{\beta_{k}}, a_{k}n, b_{k}).
$$
If $b_{k}\neq 0$ this is at most $p^{v_p(b_{k})}$. On the other hand, if $b_{k}=0$, then this is at most $p^{v_p(a_k)}\gcd(p^{\beta_{k}},n)\leq p^{v_p(a_k)}\gcd(p^{\beta_1+\cdots+\beta_{k}},n)$.
Taking a product over all primes, the bound in \eqref{eq:white-U} easily follows. This completes the proof of the lemma. 
\end{proof}

\subsection{Pairs of quadrics modulo prime powers}

Our first result concerns the roots  of  the pair of diagonal quadratic forms $Q_1,Q_2$ in 
\eqref{eq:Q1Q2}. 
Recalling the definition \eqref{eq:DELTA} of $\Delta$, we have  the following result. 

\begin{lemma}
Let $p$ be a prime  and let $a\in \mathbb{N}$.
Let 
$$
\nu_p=\begin{cases}
1 &\text{ if $p=2$,}\\
0 &\text{ if $p>2$.}
\end{cases}
$$
Let $\uu\in \left(\mathbb{Z}/p^{a}\mathbb{Z}\right)^{4}$ such that $p\nmid \uu$ and $Q_{1}(\uu)\equiv Q_{2}(\uu)\equiv 0 \bmod{p^{a}} $. Then,
for any  integer $b>a$, we have 
\[
\#\left\{\vv\in \left(\mathbb{Z}/p^{b}\mathbb{Z}\right)^{4}:\begin{array}{l}Q_{1}(\vv)\equiv Q_{2}(\vv)\equiv 0 \bmod{p^{b}} \\ \vv\equiv\uu \bmod{p^{a}} \end{array}\right\}\leq
\begin{cases}
p^{2(b-a)}						&\text{if $p\nmid \Delta$,}\\ 
p^{2\nu_p+3(b-a)}		&\text{if $p\mid \Delta$.}
\end{cases}
\]
\label{lem : loc}
\end{lemma}
\begin{proof}
Let us write $S_a(p^b)$ for the set whose cardinality is  to be estimated. 
We begin by treating the case  $p>2$, in which case $\nu_p=0$.
We claim that 
$$
\#S_a(p^{b})\leq \#S_a(p^{b-1}) \times 
\begin{cases}
p^{2} &\text{ if $p\nmid \Delta$,}\\
p^{3} &\text{ if $p\mid \Delta$,}
\end{cases}
$$
for any $b> a$. 
In particular, this implies that $b\geq 2$, since $a\geq 1$.
Noting that $\#S_a(p^a)=1$, 
an inductive argument completes the proof when $p>2$.
To check the claim we note that any 
 $\vv'\in S_a(p^b)$ can be written  $\vv'=\vv+p^{b-1}\w$ for 
$\vv\in S_a(p^{b-1})$ and $\w\in (\ZZ/p\ZZ)^4$.
In particular $\vv'\equiv \uu\bmod{p^a}$.
Moreover, the condition $p^b\mid Q_{i}(\vv')$ for $i=1,2$ implies that
\[
p^{-b+1}Q_{i}(\vv)+\nabla Q_{i}(\vv)\cdot\ww\equiv 0\bmod{p},
\]
for $i=1,2$. Note that 
$
\nabla Q_1(\vv)=2(a_1v_1,\dots, a_4v_4)$ and
$\nabla Q_2(\vv)=-2(b_1v_1,\dots, b_4v_4)$.
Moreover, we have $\vv\equiv \uu\bmod{p}$, since $\vv\in S_a(p^{b-1})$.
It follows that we are interested in counting 
$\ww\in (\ZZ/p\ZZ)^4$ for which 
$p^{-b+1}Q_{i}(\vv)+\cc_i\cdot\ww\equiv 0\bmod{p}$,
where $\cc_i=\nabla Q_{i}(\uu)$. 
If  $p\nmid \Delta$ then $\cc_1$ and $\cc_2$ are non-proportional modulo $p$ and there are  $p^{2}$ choices for $\ww$. If $p\mid \Delta$ but 
$p\nmid \cc_1$ then  we get at most $p^{3}$ choices for $\ww$. Similarly, if $p\mid \Delta$ but
$p\nmid \cc_2$. Finally, we note that $p\mid (\cc_1,\cc_2)$ is impossible, since it would then follow that 
$a_iu_i\equiv b_iu_i\equiv 0 \bmod{p}$ for $1\leq i\leq 4$, which contradicts the assumption  $p\nmid \uu$,  since $\gcd(a_i,b_i)=1$ for $1\leq i\leq 4$. 

It remains to deal with the case $p=2$.  This time we note that any $\vv\in S_a(2^b)$ can be written 
$\vv=\uu+2^a \ww$, where $\w\in (\ZZ/2^{b-a}\ZZ)^4$. If $2^b\mid Q_1(\vv)$ then it easily follows 
$2^{a+1}\mid Q_1(\uu)$, whence
\[
A+\sum_{i=1}^4 a_iu_iw_i +2^{a-1}Q_1(\ww) \equiv 0\bmod{2^{b-a-1}},
\]
where $A=Q_1(\uu)/2^{a+1}\in \ZZ$. Similarly, 
\[
B+\sum_{i=1}^4 b_iu_iw_i +2^{a-1}Q_2(\ww) \equiv 0\bmod{2^{b-a-1}},
\]
where $B=Q_2(\uu)/2^{a+1}\in \ZZ$. Since $2\nmid \uu$, we may assume without loss of generality that $2\nmid u_1$. Moreover, since $\gcd(a_1,b_1)=1$, we may further assume that $2\nmid a_1$. 
Hence, for $2^{3(b-a)}$ choices of $w_2,w_3,w_4$, we are left with counting the number of $w_1 \in
\ZZ/2^{b-a}\ZZ$ such that $f(w_1)\equiv 0 \bmod{2^{b-a-1}}$,
where
$
f(x)=2^{a-1}a_1x^2+a_1u_1x+C,
$
for an appropriate integer $C$. Since $f'(x)=2^a a_1x+a_1u_1$ is always odd, so it follows from Hensel's lemma that the congruence
$f(x)\equiv 0 \bmod{2^{b-a-1}}$ has at most $2$ roots modulo $2^{b-a-1}$, for fixed $w_2,w_3,w_4$. 
This therefore implies that 
$
\#S_a(2^b)\leq 2^{2+3(b-a)},
$
which completes the proof of the lemma.
\end{proof}

Of special interest in our work will be the functions
\begin{equation}\label{eq:rho-q}	\varrho(q)=\#\{\yy \in (\ZZ/q\ZZ)^4:Q_1(\yy)\equiv Q_2(\yy)\equiv 0\bmod q\}
\end{equation}
and 
\begin{equation}\label{eq:rho-q*}	\varrho^*(q)=\#\{\yy \in (\ZZ/q\ZZ)^4:
\gcd(\yy,q)=1,~
Q_1(\yy)\equiv Q_2(\yy)\equiv 0\bmod q\},
\end{equation}
for any $q\in \NN$.
These counting functions have already featured in work of Browning and Munshi
 \cite{BM13}, leading to the following result.

\begin{lemma}
Let $p$ be a prime and let $r\in \NN$. Then we have 
\begin{itemize}
\item[(i)] 
$\rho^*(p^r)=
p^{2r}(1+O(p^{-\frac{1}{2}}))$ if $p\nmid\Delta$.
\item[(ii)] 
$\rho(p^r)=O(rp^{2r})$.
\end{itemize}
\label{lem : loc2'}
\end{lemma}

\begin{proof}
If $p\nmid\Delta$, then the curve 
$Q_{1}=Q_{2}=0$ defines a smooth curve over $\mathbb{F}_{p}$ and (i) follows from combining Lemma 
\ref{lem : loc} with the Weil bound. Alternatively, for any $p$,  (ii) follows from the proof of  \cite[Lemma~2]{BM13}. 
\end{proof}

\subsection{Geometry of numbers and a special lattice} 

We shall also care deeply about the shape of a certain lattice that features in our work.
For any  $d\in \NN$, 
we can define an equivalence relation on 
$\left(\mathbb{Z}/d\mathbb{Z}\right)^{4}$, by saying 
$\uu$ is equivalent to $\vv$ if and only if there exists $\lambda\in\left(\mathbb{Z}/d\mathbb{Z}\right)^{\times}$ such that $\lambda\uu\equiv \vv \bmod{d}$.
We shall be interested in the set  of equivalence classes
\begin{equation}\label{eq:def-Vd}
V_{d}^{\times}=\{\uu\in\left(\mathbb{Z}/d\mathbb{Z}\right)^{4}: \gcd(\uu,d)=1,~Q_{1}(\uu)\equiv Q_{2}(\uu)\equiv 0 \bmod{d}\}/\left(\mathbb{Z}/d\mathbb{Z}\right)^{\times}.
\end{equation}
For any $\uu\in\left(\mathbb{Z}/d\mathbb{Z}\right)^{4}$ such that $\gcd(\uu,d)=1$ and $Q_{1}(\uu)\equiv Q_{2}(\uu)\equiv 0 \bmod{d}$, we will denote by $[\uu]$ its class in $V_{d}^{\times}$. For any $[\uu]\in V_{d}^{\times}$, and any $k\mid d$, the main goal of this section is to discuss various properties of the lattice
\begin{equation}\label{eq:lattice-def}
\Lambda_{[\uu],k}=\left\{\bfy\in\mathbb{Z}^{4}:\text{ }\exists\lambda\in\mathbb{Z}\text{ such that }
\bfy \equiv \lambda\uu\bmod{k}\right\}.
\end{equation}
This definition is clearly  independent of the particular choice of  representative $\uu\in[\uu]$.
The lattice $\Lambda_{[\uu],k}$ has rank $4$ and determinant $k^{3}$. We denote by 
$
1\leq s_{1,[\uu],k}\leq \cdots \leq s_{4,[\uu],k}$
the 
associated successive minima. It follows from Minkowski's second theorem that
\begin{equation}\label{eq:mink}
k^3\ll s_{1,[\uu],k} \cdots s_{4,[\uu],k}\ll k^3.
\end{equation}

We will need good estimates for the size of $V_d^\times$ and also for the number of classes in 
$V_d^\times$ which reduce modulo $k$ to a given class in $V_k^\times$, for any $k\mid d$. 
First we introduce some notation. 
 For any $n\in \NN$, we write
\begin{equation}\label{eq:delta-notation}
n_{\Delta}=\prod_{p\mid \Delta}p^{v_{p}(n)}.
\end{equation}
With this in mind, we shall prove the following result.

\begin{lemma}\label{lem : loc''}
Let $d,k\in \NN$ such that $k\mid d$.  Let 
$[\uu]\in V_k^\times$. Then 
$$
\#\{[\vv]\in V_{d}^{\times}:[\vv\bmod{k}]=[\uu]\}\ll \left(\frac{d}{k}\right)_{\Delta}\cdot \frac{d}{k}.
$$
Moreover, we have 
$$
\#V_{d}^{\times}\ll_{\varepsilon}d \cdot d_{\Delta}^{\varepsilon} \cdot \prod_{\substack{p\mid d\\p\nmid\Delta}}(1+O(p^{-\frac{1}{2}})).
$$
\end{lemma}

\begin{proof}
By the Chinese remainder theorem it will suffice to treat  the case that 
$k=p^a$ and $d=p^b$ for a prime $p$ and integers $0\leq a<b$. For any 
$[\uu]\in V_{p^a}^\times$, we observe that 
\begin{equation}\label{eq:chloe}
\begin{split}
\#&\left\{\vv\in \left(\mathbb{Z}/p^{b}\mathbb{Z}\right)^{4}:\begin{array}{l}Q_{1}(\vv)\equiv Q_{2}(\vv)\equiv 0\bmod{p^b}\\ \vv\bmod{p^{a}}\in[\uu] \end{array}\right\}\\
&\qquad\qquad= \varphi(p^{b})\#\left\{[\vv]\in V_{p^{b}}^{\times}: [\vv\bmod{p^{a}}]=[\uu]\right\}.
\end{split}\end{equation}
Taking $a=0$, we deduce that 
\begin{equation}\label{eq:stone}
\#V_{p^b}^\times =\frac{\rho^*(p^b)}{p^{b}}\left(1-\frac{1}{p}\right)^{-1},
\end{equation}
for any $b\in \NN$, where $\rho^*(p^b)$ is defined in  \eqref{eq:rho-q*}. 
The  second part of the lemma is now  a  consequence of Lemma~\ref{lem : loc2'}.

To handle the first part of the lemma, we may clearly assume that $a\geq 1$.
We observe that the left hand side of \eqref{eq:chloe} is
$$
\sum_{\uu'\in[\uu]}\#\left\{\vv\in \left(\mathbb{Z}/p^{b}\mathbb{Z}\right)^{4}:\begin{array}{l}Q_{1}(\vv)\equiv Q_{2}(\vv)\equiv 0\bmod{p^{b}}\\ \vv\equiv\uu ' \bmod{p^{a}} 
\end{array}\right\}.
$$
The number of $\uu'$ in the outer sum is $\phi(p^a)$. Moreover, we can use Lemma \ref{lem : loc} to estimate the remaining cardinality, which easily completes the proof of the lemma, since 
$\phi(p^b)=\phi(p^a)p^{b-a}$ if $a\geq 1$.
\end{proof}

We take this opportunity to record 
the following facts about $\#V_{p^b}^\times$ for generic primes. 

\begin{corollary}\label{cor:stone}
For any prime $p\nmid \Delta$ and $b\in \NN$, we have 
$\#V_{p^b}^\times=p^b(1+O(p^{-\frac{1}{2}}))$.
\end{corollary}

\begin{proof}
This 
follows on combining 
 \eqref{eq:stone}
with  Lemma \ref{lem : loc2'} in the case
 $p\nmid \Delta$ .
\end{proof}

\section{Counting points on quadric surfaces}\label{s:quadric}
For given non-zero $A_1,\dots,A_4\in \ZZ$, let  
$
Q(\yy)=A_1y_1^2+A_2y_2^2+A_3y_3^2+A_4y_4^2
$ 
be a fixed diagonal  quadratic form.
In this section we record some estimates for the  counting function 
\begin{equation}\label{eq:NQB}
	N(Q;B)=\#\{\yy\in\BZprim^4:Q(\yy)=0,|\yy|\leqslant B\},
\end{equation} 
which have the key feature that they depend uniformly on $A_1,\dots,A_4$.
Our first estimate is based on the geometry of numbers arguments used in \cite{DA}, and our second is based on a circle method analysis  \cite{BrowningM}.

Let 
$
\varDelta_Q=A_1A_2A_3A_4
$ 
be the discriminant of $Q$ and let 
$
\|Q\|=\max_{1\leqslant i\leqslant 4}|A_i|
$ be its height. We define    a Dirichlet character  $\chi_Q$ induced by the Kronecker symbol $(\frac{\varDelta_Q}{\cdot})$. 
Let $\varpi$ be the multiplicative arithmetic function defined by \begin{equation}\label{eq:pim}
	\varpi(m)=\prod_{p\mid m}\left(1+\frac{1}{p}\right)
\end{equation}
and set  
\begin{equation}\label{eq:yellow-bag}
\deltabad=\prod_{\substack{p^e\| \varDelta_Q\\ e\geqslant 2}}p^e.
\end{equation}
Combining the argument in \cite[p.3]{DA} with the  main result of \cite{DA}, we obtain the following bound

\begin{lemma}\label{thm:discreteanalysis}
Let $\ve>0$.
	If $\deltabad\leqslant B^\frac{1}{20}$, then 
	$$N(Q;B)\ll_{\varepsilon} \varpi(\varDelta_Q)\deltabad^{\frac{1}{4}+\varepsilon}\left(\frac{\|Q\|^4}{|\varDelta_Q|}\right)^\frac{5}{8}\left(B^\frac{4}{3}+\frac{B^2}{|\varDelta_Q|^\frac{1}{4}}\right)
	L(\sigma_B,\chi_Q),
	$$ 
	where $\sigma_B=1+\frac{1}{\log B}$.
\end{lemma}

The appearance of the factor $\deltabad^{\frac{1}{4}+\varepsilon}$ will  be problematic when $\deltabad$ is large. We now revisit the arguments in  \cite{DA} to show that the dependence on 
$\deltabad$ can be mitigated at the expense of an additional $B^\varepsilon$-factor. 
This is summarised in the following result. 

\begin{lemma}\label{prop:N(B)discreteanalyis1}
Let $\ve>0$. Then 
	$$N(Q;B)\ll_\varepsilon B+B^\varepsilon\sum_{\substack{\cc\in\BZprim^4\\ |\cc|\ll B^\frac{1}{3}\\
	Q^*(\cc)\neq 0}}\left(1+B\frac{\|Q\|}{|\varDelta_Q|^\frac{1}{2}}\frac{\gcd(\deltabad^3,Q^*(\cc)^2)^\frac{1}{6}}{|\cc|}\right),$$
	where $Q^*$ is the dual quadratic form.
\end{lemma}
\begin{proof}
We sketch the proof of Lemma \ref{thm:discreteanalysis}.
The main idea is to use Siegel's lemma to cover with plane sections the integer solutions to the equation  $Q(\yy)=0$ which lie in the box $|\yy|\leqslant B$. Thus any such point lies on at least one  plane $\cc\cdot\yy=0$,
where $\cc\in\BZprim^4$ satisfies  $|\cc|\ll B^\frac{1}{3}$, for an absolute implied constant.
This produces a union of conics  $Q_\cc$, as in \cite[Lemma 2.1]{DA}. 
We cover points 
on each conic $Q_\cc$
 using a family of ellipsoids, the number of which is effectively bounded in terms of the dual form $Q^*$ and $\cc$. This is the object of \cite[Lemma 2.2]{DA}. In this way,  the problem reduces to counting lattice points in a  conic within a fixed ellipsoid, which can be transformed to  counting points on conics in unequal boxes.

 For the purposes of the  lemma the main idea is to 
not use the  inequality displayed  above \cite[Eq.~(2.16)]{DA}, but to
 take the trivial bounds 
 $
 \log(2+|\cc|^2\|Q\|^3/|Q^*(\cc)|)=O_\ve(B^\ve)
 $
and 
 $R(Q^*(\cc))=O_\ve(B^\ve)$ 
at the close of \cite[\S~2.2]{DA}.  The statement of the lemma easily follows.
\end{proof}

Our next estimate for $N(Q;B)$ in \eqref{eq:NQB} is based on the circle method. 
Let us begin with a few remarks about the singular series 
$\mathfrak{S}(Q)$. This is 
defined to be 
\begin{equation}\label{eq:interview}
		\mathfrak{S}(Q)=\prod_p\sigma_p,
\end{equation}
where
$$
\sigma_p=\lim_{k\to \infty}\frac{\#\left\{\x\in (\ZZ/p^k\ZZ)^4: Q(\x)\equiv 0 \bmod{p^k}
\right\}}{p^{3k}}.
$$
The following result is concerned with an upper bound for $\mathfrak{S}(Q)$. 

\begin{lemma}\label{lem:interview'}
Let $\ve>0$.
Assume that there exists $A\in \NN$ such that 
 $\gcd(A_i,A_j)\mid A$, for all distinct $i,j\in \{1,\dots, 4\}$. 
Then 
$$\mathfrak{S}(Q)\ll_{\ve,A} \deltabad^\ve 
 L(1,\chi_Q),
$$
where $\deltabad$ is given by  \eqref{eq:yellow-bag} and the implied constant
depends on $\ve$ and $A$.
\end{lemma}
\begin{proof}
On revisiting the proof of \cite[Lemma 4.10]{duke}, it is shown that 
$
\prod_{p\nmid 2\varDelta}\sigma_p\ll  L(1,\chi_{Q}).
$ 
Moreover, by \cite[Lemma 4.8]{duke}, we have 
$
\sigma_p=1
$ 
if  $p\mid \varDelta_Q$ but $p\nmid 2\deltabad$.
Thus
$$
\mathfrak{S}(Q)\ll 
 L(1,\chi_Q)
\prod_{p\mid 2\deltabad} 
\sigma_p.
$$
To estimate the remaining product, we examine $\sigma_p$ for a given prime $p$.
Let  $f_{i}=v_p(A_i)$, 
for $1\leq i\leq 4$, and assume without loss of generality that 
$f_{1}\leq f_2\leqslant\cdots\leqslant f_{4}$.
Then the last part of the proof of \cite[Lemma 4.10]{duke} gives 
 $\sigma_p\ll p^{(f_1+f_2+f_3)/3}.$  It follows that 
	$\sigma_p=O_A(1)$ for an implied constant that is allowed to depend on $A$. 
We obtain the statement of the lemma on taking 
the product over all   $p\mid 2\deltabad$.
\end{proof}

The next result  has the advantage that  there is no restriction on the size of $\deltabad$,  or on the size of coefficients, but it comes  at the expense of a worse error term. 

\begin{lemma}\label{co:uniformquad}
	Let $\ve>0$ and let $m(Q)=\min_{1\leqslant i\leqslant 4} |A_i|$. 
	Assume that there exists $A\in \NN$ such that 
 $\gcd(A_i,A_j)\mid A$, for all distinct $i,j\in \{1,\dots, 4\}$. 
	Then $$N(Q;B)\ll_{\varepsilon ,A}
	\frac{
 \deltabad^\ve L(1,\chi_Q) 	
	}{(m(Q)\|Q\|)^\frac{1}{2}}B^2+\frac{\|Q\|^{11+\varepsilon}}{m(Q)^6|\varDelta_Q|^\frac{1}{2}}B^{\frac{3}{2}+\varepsilon}.$$
\end{lemma}

\begin{proof}
In fact we shall prove the same upper bound for the quantity $N'(Q;B)$, in which the stipulation that $\y$ is primitive is dropped. To do so, we shall actually apply an asymptotic formula for a smoothly weighted version of this counting function. 
Consider the non-negative smooth weight function 
$$w(x)=\begin{cases}
	\exp\left(-(1-x^2)^{-1}\right) & \text{ if } |x|<1\\ 0 & \text{ if } |x|\geqslant 1,
\end{cases}$$ and define  
$$\omega(x)=
\left(\int_{-\infty}^{\infty}w(x)\operatorname{d}x\right)^{-1}3\int_{x-\frac{2}{3}}^{x-\frac{1}{3}}w(3y)\operatorname{d}y.$$
This is a smooth function that takes values in $[0,1]$ and is supported on $(0,1)$.
Now for $\xx\in\BR^4$, we put
$$w_1(\xx)=w(x_1-2)\prod_{i=2}^4\omega\left(1-\frac{x_i}{x_1}\right).$$
Then  $w_1$ is supported on the set
$
\{\xx\in\BR^4:1\leqslant x_1\leqslant 3,~ 0\leqslant x_2,x_3,x_4\leqslant x_1\}.
$ 
We define the  weighted counting function
$$
N^\prime_{w_1}(Q;B)=
\sum_{\substack{\yy\in\BZ^4\\ Q(\yy)=0}}w_1\left(\frac{\yy}{B}\right).$$
Under the assumption that 
 $\varDelta_Q\neq \square$, it follows from \cite[Prop.~2]{BrowningM} that 
$$
N^\prime_{w_1}(Q;B)=\sigma_{\infty,w_1}(Q)\mathfrak{S}(Q)B^2+O_\varepsilon\left(\frac{\|Q\|^{11+\varepsilon}}{|A_1|^6|\varDelta_Q|^\frac{1}{2}}B^{\frac{3}{2}+\varepsilon}\right),
$$ 
for any $\ve>0$. Here,   
$\mathfrak{S}(Q)$ is given by \eqref{eq:interview} and 
 $\sigma_{\infty,w_1}(Q)\geqslant 0$ is the singular integral, which is shown to  satisfy
$\sigma_{\infty,w_1}(Q)\ll (|A_1|\|Q\|)^{-\frac{1}{2}}$
in \cite[Eq.~(2.6)]{BrowningM}. 

Finally, to obtain a  uniform bound  for the counting function $N^\prime(Q;B)$, we
follow the argument in 
\cite[p.~18]{BrowningM}. Let $Q^\sigma$ denote the quadratic forms obtained by  permuting 
the coefficients of $Q$, for any  $\sigma \in S_4$. On decomposing the interval $[-B,B]$ into dyadic intervals, we have
\begin{align*}
N^\prime(Q;B)
&\ll 1+\sum_{\sigma\in S_4}\sum_{j=0}^{\infty}N^\prime_{w_1}(Q^\sigma; 2^{-j}B)
\ll_\varepsilon \frac{\mathfrak{S}(Q)B^2}{(m(Q)\|Q\|)^\frac{1}{2}}+\frac{\|Q\|^{11+\varepsilon}}{m(Q)^6|\varDelta_Q|^\frac{1}{2}}B^{\frac{3}{2}+\varepsilon}.
 \end{align*}
An application of Lemma \ref{lem:interview'} now completes the proof.
\end{proof}

\section{Upper bounds for $\#\CM(X,Y)$ and related quantities}\label{s:upper}

Let $X,Y\geq 1$. 
The main goal of this section is to prove Theorem \ref{t:upper}, which is concerned with estimating 
$\#\CM(X,Y)$, where $\CM(X,Y)$ is defined in \eqref{eq:MX}.
 Along the way we shall  establish several auxiliary estimates that will have their own role to play. This section should be seen as an analogy to
\cite[\S~2]{duke}, the principal results of which are 
 \cite[Lemmas~2.1 and 2.7]{duke}. However, 
unlike the variety studied in \cite{duke}, we have less symmetry and fewer variables. This 
prohibits the ability to apply the arguments based on  Hua's inequality that were  used to great effect  in \cite{duke}.
Rather, our proof of Theorem \ref{t:upper} relies on a number of different upper bounds that will be played off against each other and which we proceed to record here.

We begin by dealing with  the  counting problem in which the condition  \eqref{eq:notsquare} 
fails.  
Thus let 
$$
\mathcal{M}^\square(X,Y)=\left\{(\bfx,\bfy)\in\ZZ_{\text{prim}}^{2}\times \ZZ^{4}:
\begin{array}{l}
\text{\eqref{eq:quadbundle} and \eqref{eq:square} hold}\\ 
|\xx|\leq X, ~|\yy|\leq Y
\end{array}\right\},
$$
for $X,Y\geq 1$.
We have  the following estimate.

\begin{lemma}\label{lem:neron}
Let $\ve>0$. Then $\#\mathcal{M}^\square(X,Y)\ll_\ve X^\ve Y^{2+\ve}$.
\end{lemma}

\begin{proof}
For a fixed choice of $\x$, the quadric in \eqref{eq:quadbundle} admits 
$O_\ve(Y^{2+\ve})$ solutions $\y\in \ZZ^4$ with $|\y|\leq Y$, thanks to 
work of Heath-Brown \cite[Thm.~2]{HB02}. It is important to emphasise here that the implied constant is only allowed to depend on $\ve>0$, and not on $\x$. 
It remains to estimate the number of zeros $(z,\x)\in \ZZ^3$ of the equation
$
z^{2}=L_{1}(\bfx)\cdots L_{4}(\bfx),
$
with $\gcd(x_1,x_2)=1$ and $|\x|\leq X$. This equation defines a genus $1$ curve in weighted projective space $\PP(2,1,1)$ and so it 
has $O_\ve(X^\ve)$ solutions by the theory of N\'eron heights. 
The statement of the lemma follows. 
\end{proof}

We now return to the task of estimating $\#\CM(X,Y)$.
For any $\xx\in\BZprim^2$ such that $L_1(\xx)\cdots L_4(\x)\neq 0$, we define 
\begin{equation}\label{eq:deltabadx}
	\deltabad(\xx)=\prod_{\substack{
	p^e\| L_1(\xx)\cdots L_4(\x)\\
	e\geqslant 2}}p^e.
\end{equation} 
For given $D\geq 1$, we let 
\begin{equation}\label{eq:M1X}
	\CM_1(X,Y;D)=\{(\xx,\yy)\in\CM(X,Y):\deltabad(\xx)\leqslant D\}
\end{equation}
and 
\begin{equation}\label{eq:M2X}
	\CM_2(X,Y;D)=\{(\xx,\yy)\in\CM(X,Y):\deltabad(\xx)> D\}.
\end{equation}
We shall prove the following upper bound for the size of  the first set.

\begin{proposition}\label{prop:M1Xsmall}
Let $\ve>0$ and assume that  $D\leqslant Y^{\frac{1}{20}}$. Then 
	$$\#\CM_1(X,Y;D)\ll_{\varepsilon} XY^2+X^2Y^\frac{4}{3}  +(D X)^\varepsilon \left(D^{\frac{3}{4}}(X^\frac{15}{8}Y^\frac{4}{3}+X^\frac{7}{8}Y^2)\right).$$
\end{proposition}

The main tool in the proof of this result is 
 \cite[Thm.~1.1]{DA}, which  
 is recorded in Lemma~\ref{thm:discreteanalysis} and which
  requires $\deltabad(\xx)$ to be sufficiently small. 
It is worth taking a moment to compare with the analogous situation in \cite{duke}. 
There, a version of Proposition \ref{prop:M1Xsmall} is  proved 
using \cite[Thm.~1.1]{DA},
in which there is no appearance of any power  of $D$.  In our situation, the factor $\deltabad(\x)^{1/4+\varepsilon}$ in  \cite[Thm.~1.1]{DA} 
becomes a major technical issue. 
At the expense of allowing an additional $(XY)^\varepsilon$-factor, we will show that the argument behind 
Proposition \ref{prop:M1Xsmall} can be adjusted  to prove the following result.

\begin{proposition}\label{prop:M2Xbig}
Let $\ve>0$. Then 
	$$\#\CM_2(X,Y;D)\ll_\varepsilon(XY)^\varepsilon\left( \frac{XY^2+X^\frac{5}{2}Y}{D^\frac{1}{16}}+X^\frac{3}{2}Y+X^\frac{1}{2}Y^{\frac{4}{3}}\right).$$
\end{proposition}

Unfortunately, Propositions \ref{prop:M1Xsmall} and \ref{prop:M2Xbig} are not quite enough to provide a satisfactory estimate for 
$\#\CM(X,Y)$ when $X$ is a very small power of $Y$. However, in this particular case, we can invoke  the following  upper bound.

\begin{proposition}\label{prop:MXXsmall}
Let $\ve>0$. Then 
	$$\#\CM(X,Y)\ll_\varepsilon XY^2+X^{11} Y^{\frac{3}{2}+\varepsilon}.$$
\end{proposition}

This result can be viewed as a weak version of Theorem \ref{t:upper}
and is based on  Lemma \ref{co:uniformquad}.
While the principal terms $XY^2$ agree, Proposition \ref{prop:MXXsmall}  is much worse when $X$ is large. 
Later in our argument it will  be useful to have a good upper bound for the quantity
\begin{equation}\label{eq:MX*}
	\cM^*(X,Y)=\left\{(\xx,\yy)\in\BZprim^2\times\ZZ^4:
	\text{\eqref{eq:quadbundle} holds}, ~ 
	|\xx|\leqslant X,~|\yy|\leqslant Y\right\},
\end{equation}
for any  $X,Y\geqslant 1$.  Note that $\#\CM(X,Y)\leq \#\CM^*(2X,Y)$, since we have merely dropped 
from $\CM(X,Y)$ the constraint that $\y$ be primitive, as well as the condition \eqref{eq:notsquare}. 
We can combine Propositions \ref{prop:M1Xsmall}--\ref{prop:MXXsmall} to deduce the following result.

\begin{corollary}\label{cor:M*}
Let $\ve>0$ and let $X,Y\geq 1$ such that $X\leq Y^{2/3}\log Y$. Then 
$$
\#\CM^*(X,Y)\ll_\ve Y^{2+\ve}+  XY^2+X^2Y^{4/3}.
$$
\end{corollary}

\begin{proof}
We would like to insert the condition \eqref{eq:notsquare} into $\cM^*(X,Y)$. But the 
overall  contribution from those $(\x,\y)$ for which \eqref{eq:notsquare} fails is 
$O_\ve(Y^{2+\ve})$,  thanks
to  Lemma~\ref{lem:neron} and the fact that $X\leq Y^{2/3}\log Y\ll Y$. Sorting the remaining contribution according to the greatest common divisor of the coordinates of $\y$, and breaking the $\x$-sum into dyadic intervals,
it easily follows that 
\begin{equation}\label{eq:bread-b}
\#\CM^*(X,Y)\ll_\ve Y^{2+\ve} +\sum_{d\leq Y}\sum_{X_0 \nearrow X}\#\CM(X_0,Y/d),
\end{equation}
for any $\ve>0$. 
We claim that there exists $\ve>0$ such that 
\begin{equation}\label{eq:goat}
\#\CM(X,Y/d)\ll_\ve
\frac{XY^2+X^2Y^{4/3}}{d^{1+\ve}},
\end{equation}
if $X\leq Y^{2/3}\log Y$. Once inserted into \eqref{eq:bread-b}, this will clearly suffice to 
complete the proof of the corollary.

Now if $X\leq Y^{1/100}$, it is clear that \eqref{eq:goat} follows from Proposition \ref{prop:MXXsmall}. Thus we may proceed under the assumption that 
$Y^{1/100}\leq X$. Under the further assumption $X\leq Y^{2/3}\log Y$, we proceed by inspecting some of the terms in Proposition~\ref{prop:M2Xbig}. Note that 
$$
(XY/d)^\ve X^\frac{3}{2}\left(\frac{Y}{d}\right)\ll_\ve \frac{XY^{4/3+2\ve}}{d^{1+\ve}}\ll \frac{XY^2}{
d^{1+\ve}},
$$
on assuming that $\ve\leq 1/3$.
Moreover, 
$(XY/d)^\ve X^\frac{1}{2}(Y/d)^{\frac{4}{3}}\leq d^{-\frac{4}{3}}XY^2$ and 
$$
(XY/d)^\ve X^{\frac{5}{2}}\left(\frac{Y}{d}\right)\ll_\ve \frac{XY^{2+2\ve}}{d^{1+\ve}}.
$$
Turning to the terms in Proposition \ref{prop:M1Xsmall}, we have
\begin{align*}
(D X)^\varepsilon  \left(D^{\frac{3}{4}}(X^\frac{15}{8}(Y/d)^\frac{4}{3}+X^\frac{7}{8}(Y/d)^2)\right)
& \ll  
\frac{D^{\frac{3}{4}+\ve}}{d^{\frac{4}{3}}} \left( X (Y^{2/3})^{\frac{7}{8}+\ve} Y^{\frac{4}{3}} +X^{\frac{7}{8}+\ve}Y^2\right)\\
& \ll  \frac{XY^2
D^{\frac{3}{4}+\ve} }{d^{\frac{4}{3}}}\left( \frac{1}{Y^{\frac{1}{12}-\ve}} + \frac{1}{X^{\frac{1}{8}-\ve}}\right)\\
& \ll \frac{XY^2}{d^{\frac{4}{3}}}
\frac{D^{\frac{3}{4}+\ve} }{X^{\frac{1}{8}-2\ve}},
\end{align*}
if  $X\leq Y^{2/3}\log Y$. 

Hence, on combining 
Propositions \ref{prop:M1Xsmall} and 
\ref{prop:M2Xbig}, we deduce that 
 $$
 \#\CM(X,Y/d)\ll_\varepsilon  
 \frac{XY^2+X^2Y^{\frac{4}{3}}}{d^{1+\ve}}\left(1+
\min\left(
 \frac{D^{\frac{3}{4}+\ve} }{X^{\frac{1}{8}-2\ve}},
\frac{Y^{2\ve}}{D^{\frac{1}{16}}}
\right)\right),
$$
for any $D\leq Y^{1/20}$ and $Y^{1/100}<X\leq Y^{2/3}\log Y$.
But then 
\begin{align*}
\min\left(
 \frac{D^{\frac{3}{4}+\ve} }{X^{\frac{1}{8}-2\ve}},
\frac{Y^{2\ve}}{D^{\frac{1}{16}}}
\right)
&\leq 
\min\left(
 \frac{D^{\frac{3}{4}} }{Y^{\frac{1}{800}}},
\frac{1}{D^{\frac{1}{16}}}
\right)Y^{2\ve}.
\end{align*}
We can ensure that this is $O(1)$ by taking $D=Y^{1/800}$ and choosing a sufficiently small value of $\ve$.  This completes the proof of \eqref{eq:goat}.
\end{proof}

The results so far are efficient when $X$ is small compared to $Y$. The following result
is proved using completely different methods and allows us to handle the opposite case.

\begin{proposition}\label{prop:poisson}
We have
$$\#\CM(X,Y)
\ll	XY^{2}+\left(\frac{Y^{4}}{X^2} +X^{\frac{1}{4}}Y^{\frac{5}{2}}\right)\frac{\log Y}{\log\log Y}.
$$
\end{proposition}

Once in possession of  Propositions \ref{prop:M1Xsmall}--\ref{prop:poisson}, we are now positioned to 
prove our main upper bound for $\#\CM(X,Y)$.

\begin{proof}[Proof of Theorem \ref{t:upper}]
Let us put $J=(\log Y)/(\log\log Y)$.  
If  $X>Y^{2/3}J^{4/7}$,  then the statement of the theorem  follows from Proposition \ref{prop:poisson}.
If  $X\leq Y^{2/3}J^{4/7}$, on the other hand, then it follows from taking $d=1$ in \eqref{eq:goat}.
\end{proof}

The rest of this section is concerned with proving Propositions \ref{prop:M1Xsmall}, \ref{prop:M2Xbig}, 
\ref{prop:MXXsmall}
 and \ref{prop:poisson}.
Propositions \ref{prop:M1Xsmall}
and \ref{prop:M2Xbig}  will be proved 
using Lemmas \ref{thm:discreteanalysis} and 
\ref{prop:N(B)discreteanalyis1} in Section~\ref{s:DA}. Proposition 
\ref{prop:MXXsmall} will be proved in Section \ref{s:monats} and uses the  circle method bound proved in Lemma~\ref{co:uniformquad}.
Finally, in Section \ref{s:poisson} we shall prove Proposition \ref{prop:poisson}.
During the course of this work, given $X_1,X_2\geq 1$, we shall often use the elementary inequality
\begin{equation}\label{eq:convex}
	X_1+X_2\geqslant \max(X_1,X_2)\geqslant X_1^\alpha X_2^{1-\alpha},
\end{equation}
for any  $0\leqslant \alpha\leqslant 1$.

\subsection{Proof of Propositions  \ref{prop:M1Xsmall}
and \ref{prop:M2Xbig}}\label{s:DA}

Let us begin by  fixing some notation. Let $L_1,\dots,L_4\in \ZZ[x_1,x_2]$ be the pairwise non-proportional linear forms featuring in \eqref{eq:quadbundle}. We define
\begin{equation}\label{eq:badprimes}
	\CD=\prod_{1\leq i<j\leq 4}\Res(L_{i},L_{j}),
\end{equation} 
where  $\Res(L_i,L_j)$ is the resultant of $L_i$ and $L_j$, which is  defined to be the
absolute value of the determinant of the $2\times2$ matrix formed from the coefficient vectors.
Then, in what follows,  we shall make frequent use of the fact that 
\begin{equation}\label{le:discform}
\gcd(L_i(\xx),L_j(\xx))\mid \Res(L_i,L_j), \quad \text{for any $\xx\in\BZprim^2$,}
\end{equation}
for each choice $i\neq j\in \{1,\dots,4\}$. We shall also exploit  the following  elementary 
observation.

\begin{lemma}\label{le:smallvalueLi}
Let $L_1,\dots,L_4\in \ZZ[x_1,x_2]$ be pairwise non-proportional linear forms, as above.
Let  $\xx\in\BR^2$ such that  $|\xx|\sim X$. 
Then there exist $c_0\in (0,1)$ and $d_0>0$, both depending  on the coefficients of $L_1,\dots,L_4$, such that
$$
|L_{i_0}(\xx)|\leqslant c_0 X \Longrightarrow  |L_i(\xx)|\geqslant d_0 X \text{ for every $i\neq i_0$}.
$$
\end{lemma}
\begin{proof}
Assume without loss of generality that  $|\xx|=|x_1|$. 
Suppose that  $|L_1(\xx)|\leqslant c_1 X$, say, for a certain $0<c_1<1$, to be specified in due course. 
Then necessarily $L_1(\x)\neq \pm nx_1$ for any $n\in\BZ_{\neq 0}$.  Then, for  $i\in \{2,3,4\}$,
there exist $(\lambda_i,\mu_i)\in \QQ\times \QQ^\times$ such that 
$x_1=\lambda_iL_1(\x)+\mu_iL_i(\x)$. But then it follows that 
$
X\leq |x_1| \leq
|\lambda_i|c_1X+|\mu_i||L_i(\x)|,$
which implies that $|L_i(\x)|\geq |\mu_i|^{-1} (1-|\lambda_i|c_1)X$.
The lemma follows on demanding that  $c_1<1/|\lambda_i|$ for $2\leq i\leq 4$.
\end{proof}

\begin{proof}[Proof of  Proposition \ref{prop:M2Xbig}] 
We now estimate  the cardinality of $\CM_2(X,Y;D)$, as defined in \eqref{eq:M2X}. We can assume that $D\gg 1$ for an implied constant that depends  on $L_1,\dots,L_4$.
According to the estimate  following Lemma 2.7 in \cite{duke}, the  condition $\deltabad(\xx) >D$ implies that either $\gcd(L_i(\xx),L_j(\xx))> D^\frac{1}{24}$ for certain indices $i\neq j$, or else 
$L_i(\xx)^\square >D^\frac{1}{8}$. In view of \eqref{le:discform}, only the second possibility can happen if $D\gg 1$. Hence there exists $e\geqslant D^\frac{1}{16}$ and  $i_0\in \{1,\dots,4\}$ such that $e^2\mid L_{i_0}(\xx)$. Without loss of generality we study the contribution corresponding to $i_0=1$.
It therefore suffices to estimate the size of the  set
\begin{equation}\label{eq:M2Xbigi0}
	\CM_{2}(X,Y;e)=\{(\xx,\yy)\in\CM(X,Y):e^2\mid L_{1}(\xx)\},
\end{equation}
for each 
$e\geqslant D^\frac{1}{16}$.  On recalling \eqref{eq:notsquare} and \eqref{eq:MX}, we see that 
any $(\x,\y)\in \CM(X,Y)$ satisfies 
$L_1(\x)\neq 0$. Hence  $e\ll X^\frac{1}{2}$, since $e^2\mid L_{1}(\xx)$.

It will be convenient to define the set
\begin{equation}\label{eq:S(X)}
\CS(X)=\{\x=(x_1,x_2)\in \ZZp^2: |\x|\leq 2X, \text{ \eqref{eq:notsquare} holds}\}.
\end{equation}
Then  for any $\xx\in \CS(X)$, we may define the quaternary quadratic form
\begin{equation}\label{eq:Qx}
Q_{\xx}(\y)=L_1(\xx)y_1^2+\dots+L_4(\xx)y_4^2.
\end{equation}
Adopting the notation \eqref{eq:NQB} and applying 
 Lemma~\ref{prop:N(B)discreteanalyis1}, we therefore obtain
\begin{align*}
	\#\CM_{2}(X,Y;e)
	&\ll_\ve Y\sum_{\substack{\xx\in \CS(X)\\ e^2\mid L_{1}(\xx)}}1~+~Y^\varepsilon\sum_{\substack{\xx\in \CS(X) \\ e^2\mid L_{1}(\xx)}}\sum_{\substack{\cc\in\BZprim^4\\ |\cc|\ll Y^\frac{1}{3}}}1\\ &\quad +XY^{1+\varepsilon}\sum_{\substack{\xx\in \CS(X) \\ e^2\mid L_{1}(\xx)}}\sum_{\substack{\cc\in\BZprim^4\\ |\cc|\ll Y^\frac{1}{3},~Q^*_{\xx}(\cc)\neq 0}}\frac{\gcd(\deltabad(\xx)^3,Q^*_{\xx}(\cc)^2)^\frac{1}{6}}{|\cc|\prod_{i=1}^4|L_i(\xx)|^\frac{1}{2}}\\ &=W_0(X,Y;e)+W_1(X,Y;e)+W_2(X,Y;e),
\end{align*}
say. 
On appealing to  \cite[Lemma 2]{H-B}, it easily follows that  
\begin{equation}\label{eq:W1}
	W_0(X,Y;e)\ll \left(\frac{X^2}{e^2}+1\right)Y \quad \text{ and } \quad W_1(X,Y;e)\ll \left(\frac{X^2}{e^2}+1\right)Y^{\frac{4}{3}+\varepsilon}.
\end{equation}

From now on we focus on estimating $W_2(X,Y;e)$. Let us introduce various dyadic parameters $S,T$ corresponding respectively to  the ranges of $\xx$ and $\cc$. Then, in the light of  Lemma \ref{le:smallvalueLi}, we obtain
$$
W_2(X,Y;e)\ll XY^{1+\varepsilon}\sum_{\substack{S\nearrow  X\\T\nearrow  Y^\frac{1}{3}}} \frac{1}{S^\frac{3}{2}T}\sum_{i_1=1}^{4}W_{2,i_1}(S,T;e),
$$ 
where
$$
W_{2,i_1}(S,T;e)=\sum_{\substack{\xx\in \CS(X), ~e^2\mid L_{1}(\xx)\\ |\xx|\sim S\\
i\neq i_1 \Rightarrow 
|L_{i}(\xx)|\gg |L_{i_1}(\xx)|}}\sum_{\substack{\cc\in\BZprim^4\\ |\cc|\sim T,~ Q^*_{\xx}(\cc)\neq 0}}\frac{\gcd(\deltabad(\xx)^3,Q^*_{\xx}(\cc)^2)^\frac{1}{6}}{|L_{i_1}(\xx)|^\frac{1}{2}}.
$$
By further decomposing the size of  $L_{i_1}(\xx)$ into dyadic intervals using a dyadic parameter $R$, we obtain 
$$
W_{2,i_1}(S,T;e)\leq \sum_{R \nearrow S}\frac{1}{R^\frac{1}{2}}W_{2,i_1}(R,S,T;e),
$$ 
where 
\begin{equation}\label{eq:Wi11RSTe}
	W_{2,i_1}(R,S,T;e)=\sum_{\substack{\xx\in \CS(X), ~e^2\mid L_{1}(\xx)\\ |\xx|\sim S,~
	|L_{i_1}(\xx)|\sim R}}\sum_{\substack{\cc\in\BZprim^4\\ |\cc|\sim T,~Q^*_{\xx}(\cc)\neq 0}}\gcd(\deltabad(\xx)^3,Q^*_{\xx}(\cc)^2)^\frac{1}{6}.
\end{equation}

Recall the definition 
\eqref{eq:badprimes} of $\CD$ and 
let $d\mid \gcd(\deltabad(\xx)^3,Q^*_{\xx}(\cc)^2)$. We claim that there exists $d_\mathcal{D}=O(1)$ and 
$d_1,\dots,d_4\in \NN$ such that $d_\mathcal{D}d=d_1\cdots d_4$ and 
$d_i\mid \gcd(L_i(\xx)^3,c_i^6)$.
If $d\mid \deltabad(\xx)^3$ then $d\mid (L_1(\xx)\cdots L_4(\xx) )^3$. Let us put $d_i=\gcd(L_i(\x)^3,d)$ for $1\leq i\leq 4$. Then $\gcd(d_i,d_j)\mid \mathcal{D}$ for $i\neq j$. Hence 
$d_1\cdots d_4=d d_{\mathcal{D}}$, for a suitable positive integer $d_\mathcal{D}=O(1)$.
It remains to prove that $d_i\mid c_i^6$. Suppose 
that $p^\lambda\| d_i$ with $\lambda\in \NN$.
Let 
$p^\nu\| L_1(\xx)$. Then $\lambda\leqslant 3\nu$. On the other hand, we have 
$p^\lambda\mid Q^*_{\xx}(\cc)^2$, whence
\begin{align*}
	p^{\lceil\frac{\lambda}{2}\rceil}\mid Q^*_{\xx}(\cc)
	&=L_1(\xx)\sum_{\{j,k,l\}=\{2,3,4\}}L_j(\xx)L_k(\xx)c_l^2+L_2(\xx)L_3(\xx)L_4(\xx)c_1^2.
\end{align*} 
Under the condition \eqref{eq:notsquare}, it follows that 
$\prod_{i=1}^{4}L_i(\xx)\neq 0$, whence  
$p^{\min(\nu,\lceil\frac{\lambda}{2}\rceil)}\mid c_1^2$. If $\nu\leq \lceil\frac{\lambda}{2}\rceil$
then $p^{3\nu}\mid c_1^6$, whence $p^\lambda\mid c_1^6$.  If $\nu> \lceil\frac{\lambda}{2}\rceil$, then $p^{\lceil\frac{\lambda}{2}\rceil}\mid c_1^2$, whence $p^{\lambda}\mid c_1^4$. 
This completes the proof of the claim.

We now continue estimating $W_{2,i_1}(R,S,T;e)$ in \eqref{eq:Wi11RSTe}. 
Using the claim in the previous paragraph, we therefore obtain
\begin{align*}
	W_{2,i_1}(R,S,T;e)&=\sum_{\substack{\xx\in \CS(X), ~e^2\mid L_{1}(\xx)\\ |\xx|\sim S,~
	|L_{i_1}(\xx)|\sim R}}
	\sum_{\substack{d\mid \deltabad(\xx)^3}}d^\frac{1}{6}\sum_{\substack{\cc\in\BZprim^4,~|\cc|\sim T\\ d\mid Q^*_{\xx}(\cc)\neq 0}}1\\ 
	&\ll  \sum_{\substack{\xx\in \CS(X), ~e^2\mid L_{1}(\xx)\\ 
	|\xx|\sim S,~|L_{i_1}(\xx)|\sim R}}\sum_{\substack{d\mid \deltabad(\xx)^3
	}}
	{d}^\frac{1}{6}\sum_{\substack{d_\CD d=d_1\cdots d_4\\	d_i\mid L_i(\x)^3}}
	\sum_{\substack{\cc\in\BZprim^4,~|\cc|\sim T\\ d_i\mid c_i^6}}1. 
	\end{align*}
The condition $|\xx|\sim S$ implies each $d_i\ll S^3$, since $\prod_{i=1}^{4}L_i(\xx)\neq 0$.
Hence
\begin{align*}
	W_{2,i_1}(R,S,T;e)
	\ll \sum_{\substack{\xx\in \CS(X),~e^2\mid L_{1}(\xx)\\ |\xx|\sim S,~|L_{i_1}(\xx)|\sim R}}\sum_{\substack{d\mid \deltabad(\xx)^3}}d^\frac{1}{6}\sum_{\substack{d_\CD d=d_1\cdots d_4\\ \min_{1\leqslant i\leqslant 4}d_i\ll T^6\\
	\max_{1\leqslant i\leqslant 4}d_i\ll S^2}}\prod_{i=1}^{4}\left(\frac{T}{d_i^\frac{1}{6}}+1\right),
\end{align*}
where the condition $\min_{1\leqslant i\leqslant 4}d_i\ll T^6$ is deduced from the fact that at least one of the components of $\cc$ must be non-zero.
Therefore, for each factorisation $d_\CS d=d_1\cdots d_4$, we have 
\begin{align*}
	d^\frac{1}{6}\prod_{i=1}^{4}\left(\frac{T}{d_i^\frac{1}{6}}+1\right)
	&\ll T^4+T^3\sum_{j=1}^4 d_j^\frac{1}{6}+
	T^2	\hspace{-0.3cm}
	\sum_{k\neq l\in\{1,2,3,4\}}(d_kd_l)^\frac{1}{6}+T
	\hspace{-0.3cm}
	\sum_{\substack{k,l,m\in\{1,2,3,4\}\\ \text{distinct}}}(d_kd_ld_m)^\frac{1}{6}\\ &\ll T^4+T^3S^\frac{1}{2}+T^2S+TS^\frac{3}{2}.
\end{align*} 
This is $O(T^4+TS^\frac{3}{2})$.
(Here, the  absence of the constant term in the product is thanks to the fact that
$\min_{1\leqslant i\leqslant 4}d_i\ll T^6$.)

Returning  to $W_{2,i_1}(R,S,T;e)$ and using  \cite[Lemma 2]{H-B} once more, we arrive at the bound
\begin{equation}\label{eq:Wi1RSTebound}
	\begin{split}
	W_{2,i_1}(R,S,T;e)
	&\ll_\ve (T^4+TS^\frac{3}{2})
		 \sum_{\substack{\xx\in \CS(X),~e^2\mid L_{1}(\xx)\\ |\xx|\sim S,~|L_{i_1}(\xx)|\sim R}}\deltabad(\xx)^\ve\\
	&\ll_\ve S^\ve(T^4+TS^\frac{3}{2})\left(\frac{SR}{e^2}+1\right),
\end{split}
\end{equation}
for any $\ve>0$.  Summing over the dyadic parameter $R$, we therefore  obtain 
$$
W_{2,i_1}(S,T;e)=
\sum_{R \nearrow S}\frac{1}{R^\frac{1}{2}}W_{2,i_1}(R,S,T;e)\ll_\varepsilon S^{2\varepsilon  }(T^4+TS^\frac{3}{2})
\left(\frac{S^{\frac{3}{2}}}{e^2}+1\right)
.$$ 

Recall that our dyadic parameter $S$ goes to $X$, and $T$ goes to $Y^\frac{1}{3}$. Moreover, we have seen that $e\ll S^\frac{1}{2}$, if 
$e^2\mid L_{1}(\xx)$.
Thus, on  summing over $S,T$,  we obtain
$$\sum_{S,T}\frac{1}{S^\frac{3}{2}T} 
\frac{S^{\frac{3}{2}+2\varepsilon  }}{e^2}(T^4+TS^\frac{3}{2})
\ll
(XY)^{2\varepsilon} \left(\frac{Y+X^\frac{3}{2}}{e^2}\right)
$$
and 
$$
\sum_{S,T}\frac{1}{S^\frac{3}{2}T} S^{2\varepsilon} (T^4+TS^\frac{3}{2})\ll X^{2\varepsilon}\sum_{T}\left(\frac{T^3}{e^3}+1\right)\ll (XY)^{2\varepsilon}\left(\frac{Y}{e^3}+1\right).
$$
On redefining the choice of $\ve$, we finally obtain
$$
W_2(X,Y;e)\ll_\varepsilon  (XY)^\varepsilon\left(\frac{XY^2+X^\frac{5}{2}Y}{e^2}+XY\right).
$$
Taking $\alpha=\frac{1}{3}$ in  \eqref{eq:convex}, it follows that 
$
X^2 Y^\frac{4}{3}\ll XY^2+X^\frac{5}{2}Y$. Hence, on combining our bound for 
$W_2(X,Y;e)$ with  the contributions \eqref{eq:W1}, we are  led to the bound
$$
\#\CM_{2}(X,Y;e)\ll_\varepsilon (XY)^\varepsilon\left(\frac{XY^2+X^\frac{5}{2}Y}{e^2}+XY+Y^\frac{4}{3}\right),
$$
for the cardinality of   \eqref{eq:M2Xbigi0}. Finally, it follows that 
\begin{align*}
\#\CM_2(X,Y;D)&\ll 
\hspace{-0.3cm}
\sum_{D^\frac{1}{16}<e\ll\sqrt{X}}
\hspace{-0.35cm}
\#\CM_{2}(X,Y;e)
\ll_\varepsilon(XY)^\varepsilon\left( \frac{XY^2+X^\frac{5}{2}Y}{D^\frac{1}{16}}+X^\frac{3}{2}Y+X^\frac{1}{2}Y^{\frac{4}{3}}\right).
\end{align*}
This completes the proof of Proposition 
\ref{prop:M2Xbig}.
\end{proof}

Later in our argument we will also need to deal with summations over $\xx$ in which one of  the linear forms $L_1,\dots,L_4$  takes a particularly small  value. We take this opportunity to prove the  following analogue of \cite[Lemma 2.1]{duke}, whose proof is a minor modification of the one that we use to prove Proposition \ref{prop:M2Xbig}. 

\begin{lemma}\label{lem:auxbounddelta}
Let $\ve>0$.
Define $\CM_{i_1,\delta}(X,Y)=\{(\xx,\yy)\in \CM(X,Y) : |L_{i_1}(\xx)|\leq |\xx|^\delta\}$,
for any  $\delta\in (0,1)$ and $i_1\in \{1,\dots,4\}$, where
 $\CM(X,Y)$ is given by \eqref{eq:MX}. Then 
$$
\#\CM_{i_1,\delta}(X,Y)\ll_{\varepsilon,\delta} (XY)^\varepsilon(X^{\frac{1}{2}(1+\delta)} Y^2+X^{\frac{1}{2}(4+\delta)}Y).
$$
\end{lemma}
\begin{proof}
Assume without loss of generality that $i_1=1$. Then Lemma \ref{le:smallvalueLi} implies that $|L_i(\xx)|\gg_\delta X$ for any 
	$i\in \{2,3,4\}$.
We maintain the notation from  the proof of Proposition \ref{prop:M2Xbig}. On recalling \eqref{eq:Wi11RSTe} and \eqref{eq:Wi1RSTebound}, we introduce an extra dyadic parameter $R$ for the range of $L_{1}(\xx)$ and similarly obtain
	\begin{align*}
	\#\CM_{1,\delta}(X,Y)
& \ll_{\varepsilon,\delta} X^{1+\delta}Y^{\frac{4}{3}+\varepsilon} +\frac{Y^{1+\varepsilon}}{X^\frac{1}{2}}\sum_{\substack{R \nearrow X^\delta\\ T\nearrow Y^\frac{1}{3}}}\frac{1}{R^\frac{1}{2}T}W_{2,1}(R,X,T;1)\\ &\ll_{\varepsilon,\delta} X^{1+\delta}Y^{\frac{4}{3}+\varepsilon} + \frac{X^{1+\varepsilon} Y^{1+\varepsilon}}{X^\frac{1}{2}}\sum_{\substack{R \nearrow X^\delta\\ T\nearrow Y^\frac{1}{3}}}\frac{R(T^4+TX^\frac{3}{2})}{R^\frac{1}{2}T}\\
	&\ll_{\varepsilon,\delta} (XY)^\varepsilon(X^{1+\delta}Y^\frac{4}{3}+X^{\frac{1}{2}(1+\delta)} Y^2+X^{\frac{1}{2}(4+\delta)}Y).
\end{align*} 
Using \eqref{eq:convex} with $\alpha=\frac{1}{3}$, we have $X^{\frac{1}{2}(1+\delta)} Y^2+X^{\frac{1}{2}(4+\delta)}Y\geqslant X^{1+\delta}Y^\frac{4}{3}.$
\end{proof}

It is now time to analyse the size of the set $\CM_1(X,Y;D)$ that was defined in \eqref{eq:M1X}.

\begin{proof}[Proof of Proposition \ref{prop:M1Xsmall}]
We note that 
$$
\#\CM_1(X,Y;D)=\sum_{\xx\in \CS(X), ~\deltabad(\xx)\leqslant D}N(Q_\x;Y),
$$
where $\CS(X)$ is
given by \eqref{eq:S(X)}, 
 $Q_\x$   by \eqref{eq:Qx} and $N(Q_\x;Y)$ is the counting function defined in 
\eqref{eq:NQB}.
Let us put 
\begin{equation}\label{eq:normLideltaLi}
\|Q_\x\|=\max_{1\leqslant i\leqslant 4}|L_i(\xx)|,\quad \varDelta(\xx)=\prod_{i=1}^{4}L_i(\xx).
\end{equation}
Under the assumption
that  $\deltabad(\xx)\leqslant Y^\frac{1}{20}$,  it  follows from Lemma \ref{thm:discreteanalysis} that
	$$
	N(Q_\x;Y)
	\ll_\varepsilon\varpi\left(\Delta(\x)\right)\deltabad(\xx)^{\frac{1}{4}+\varepsilon}\left(Y^\frac{4}{3}N_1(\xx,Y)+Y^2N_2(\xx,Y)\right),
	$$ 
	for any $\ve>0$, 
	where the arithmetic function $\varpi$ is defined by \eqref{eq:pim}, and
	$$
	N_1(\xx,Y)=\frac{\|Q_\xx\|^\frac{5}{2}}{|\varDelta(\xx)|^\frac{5}{8}}L(\sigma_Y,\chi_{Q_\xx}) \quad \text{ and } \quad 
	N_2(\xx,Y)=\frac{\|Q_\xx\|^\frac{5}{2}}{|\varDelta(\xx)|^\frac{7}{8}}L(\sigma_Y,\chi_{Q_\xx}).
	$$
We can alternatively write $\varpi(m)=\sum_{t\mid m}\frac{\mu^2(t)}{t}.$ Under the assumption $D\leqslant Y^\frac{1}{20}$, it therefore follows that 
\begin{align*}
\#\CM_1(X,Y;D)	 &\ll_\varepsilon
	\sum_{\substack{r\in\BN\\ r~\square\text{-full}\\
	r\leqslant D}}r^{\frac{1}{4}+\varepsilon}\sum_{\substack{\xx\in \CS(X)\\\deltabad(\xx)=r}}\left(\sum_{\substack{t\mid \varDelta(\x)}}\frac{\mu^2(t)}{t}\right) \left(Y^\frac{4}{3}N_1(\xx,Y)+Y^2N_2(\xx,Y)\right).  
\end{align*}
Hence
\begin{equation}\label{eq:pink-bag}
\#\CM_1(X,Y;D)\ll	\sum_{\substack{r,t\in\BN\\ t~\square\text{-free},~t\ll X^4\\ r~\square\text{-full},~r\leqslant D}}\frac{r^{\frac{1}{4}+\varepsilon}}{t} 
	\sum_{\substack{\dd\in\BN^4\\\ d_1\cdots d_4=[r,t]}}N_{\dd}(X,Y),
\end{equation}
where
$$
N_{\dd}(X,Y)=\sum_{\substack{\xx\in \CS(X)\\ d_i\mid L_i(\xx)}}\left(Y^\frac{4}{3}N_1(\xx,Y)+Y^2N_2(\xx,Y)\right).
$$

We recall that 
$$
L(\sigma_Y,\chi_{Q_\xx})=\sum_{n=1}^{\infty}\frac{\chi_{Q_\xx}(n)}{n^{\sigma_Y}},
$$
where $\sigma_Y=1+\frac{1}{\log Y}>1$. We have $L(\sigma_Y,\chi_{Q_\xx})\geq 0$.
By Lemma \ref{le:smallvalueLi} we can assume without loss of generality that   $|L_i(\x)|\gg |L_1(\x)|$
for $i\geq 2$. On  introducing dyadic decomposition parameters $S$  for $|\xx|$ and  $R$ for  $|L_{1}(\xx)|$, and on  dropping the primitivity condition on $\xx$, it follows that 
\begin{equation}\label{eq:varSR}
N_{\dd}(X,Y) \ll
\sum_{S\nearrow X}\sum_{R \nearrow S}\left(\frac{S^\frac{5}{8}Y^\frac{4}{3}}{R^\frac{5}{8}}+\frac{Y^2}{S^\frac{1}{8}R^\frac{7}{8}}\right)N_{\dd}(S,R,Y),
\end{equation}
where 
\begin{equation}\label{eq:Nprimei1dd}
\begin{split}
	N_{\dd}(S,R,Y)
	&=\sum_{\substack{\xx\in\BZ^2\\ 
	|\xx|\sim S,~|L_{1}(\xx)|\sim R\\\prod_{i=1}^{4}L_i(\xx)\neq\square\\ d_{i}\mid L_i(\xx)}}
	\sum_{n=1}^{\infty}\frac{\chi_{Q_\xx}(n)}{n^{\sigma_Y}}.
	\end{split}
\end{equation}
In particular, we may assume that $d_1\cdots d_4\ll S^4$.
We have  $L_i(\xx)=a_ix_1+b_ix_2$ for $1\leqslant i\leqslant 4$, 
where  $\gcd(a_i,b_i)=1$.
But then there  exists  a matrix
$\mathbf{M}\in \mathrm{SL}_2(\ZZ)$ with first row equal to $(a_1,b_1)$. Making the change of variables 
$\y=\mathbf{M}\x$, we let $J_i(\yy)=L_i(\mathbf{M}^{-1}\yy)$, for $1\leq i\leq 4$.
We can thus  rewrite \eqref{eq:Nprimei1dd} as
$$N_{\dd}(S,R,Y)=
\sum_{\substack{\yy\in\BZ^2\\ |y_1|\sim R, ~|\mathbf{M}^{-1}\y|\sim S\\
\prod_{i=1}^{4}J_i(\yy)\neq\square\\ d_{i}\mid J_i(\yy)}}\sum_{n=1}^{\infty}\frac{\chi_{\yy}(n)}{n^{\sigma_Y}},
$$
where
$
\chi_{\yy}(\cdot)=(\frac{J_1(\y)\cdots J_4(\y)}{\cdot}).
$

Let  $\theta>\frac{3}{16}$. We introduce the truncation parameter \begin{equation}\label{eq:N1}
	N_1=(S^3R)^{2\theta}.
\end{equation} Then since $\prod_{i=1}^{4}J_i(\yy)\neq\square$, the character $\chi_{\yy}$ is a non-principal Dirichlet character  of modulus at most $
\prod_{i=1}^{4}|J_i(\yy)|\ll \prod_{i=1}^{4}|L_i(\xx)|\ll S^3R.$
The Burgess bound in  Lemma \ref{le:Burgess} implies that 
$$
\sum_{n>N_1}\frac{\chi_{\yy}(n)}{n^{\sigma_Y}}\ll_\theta N_1^{-\frac{1}{2}}(S^3R)^\theta
\ll 1.
$$
Thus it follows that 
\begin{equation}\label{eq:largeN1}
\sum_{\substack{\yy\in\BZ^2\\ |y_1|\sim R, ~|\mathbf{M}^{-1}\y|\sim S\\
\prod_{i=1}^{4}J_i(\yy)\neq\square\\ d_{i}\mid J_i(\yy)}}
\sum_{n>N_1}\frac{\chi_{\yy}(n)}{n^{\sigma_Y}}\ll_\theta \frac{SR}{d_1\cdots d_4}+S.
\end{equation}

It remains to study the contribution
$$
\sum_{\substack{\yy\in\BZ^2\\ |y_1|\sim R, ~|\mathbf{M}^{-1}\y|\sim S\\
\prod_{i=1}^{4}J_i(\yy)\neq\square\\ d_{i}\mid J_i(\yy)}}
\sum_{n\leq N_1}\frac{\chi_{\yy}(n)}{n^{\sigma_Y}}.
$$
We observe that $\gcd(d_i,d_j)\mid \CD$ for each  $1\leqslant i< j\leqslant 4$.
We need to separate out the contribution from those $\y$ for which  $\prod_{i=1}^{4}J_i(\yy)=\square$. For such $\y$, the $n$-sum contributes $O_\ve(N_1^\ve).$
Moreover, there are $O(Y^\ve)$ primitive vectors $|\y|\leq Y$ for which 
$\prod_{i=1}^{4}J_i(\yy)=\square$, 
by the proof of Lemma \ref{lem:neron}, 
leading to an overall contribution 
$
O_\ve (S^{1+\ve}),
$
on extracting possible common divisors from $y_1$ and $y_2$.
Hence
\begin{equation}\label{eq:white-bag}
\sum_{\substack{\yy\in\BZ^2\\ |y_1|\sim R, ~|\mathbf{M}^{-1}\y|\sim S\\
\prod_{i=1}^{4}J_i(\yy)\neq\square\\ d_{i}\mid J_i(\yy)}}
\sum_{n\leq N_1}\frac{\chi_{\yy}(n)}{n^{\sigma_Y}}
= O_\ve(S^{1+\ve} N_1^\ve)+
\sum_{n\leq N_1}\frac{1}{n^{\sigma_Y}}T_{\dd}(n;S,R),
\end{equation}
where
$$
T_{\dd}(n;S,R)=\sum_{
\substack{
\yy\in\BZ^2\\ 
|y_1|\sim R, ~|\mathbf{M}^{-1}\y|\sim S\\
d_{i}\mid J_i(\yy)}}	\left(\frac{J_1(\yy)\cdots J_4(\yy)}{n}\right).
$$

Note that once $y_1$ is fixed, there exists an interval $K_{y_1}$ of length $O( S)$, such that 
$|\mathbf{M}^{-1}\y|\sim S$ if and only if $y_2\in K_{y_1}$.
For fixed  $y_1$, we are now in a position to apply Lemma \ref{lem:PV-prep} to estimate the character sum involving $y_2$. 
Note that there exists a factorisation 
 $n=n_0n_1^2$, with $n_0$ square-free. 
 Applying Lemma \ref{lem:PV-prep} with $A=\CD$, $q=1$ and $k=4$, we therefore obtain
\begin{align*}
 T_{\dd}(n;S,R)
 &\ll_\ve
 \sum_{|y_1|\sim R}
\left(\frac{S}{n_0^{\frac{1}{2}}d_1\cdots d_4} +n_0^{\frac{1}{2}}\log(d_1\cdots d_4)\right) n_0^{\ve/2} \gcd(y_1,n_0d_1\cdots d_4)\\
&\ll_\ve (d_1\cdots d_4n_0)^\ve
\left(\frac{RS}{n_0^{\frac{1}{2}}d_1\cdots d_4} +n_0^{\frac{1}{2}}R\right), 
\end{align*}
for any $\ve>0$,
since 
$
\sum_{y\leq Y}\gcd(y,a)\leq \sum_{d\mid a} d\#\{y\leq Y: d\mid y\}\leq \tau(a)Y,
$
for any  $a\in \NN$.
Since $d_1\cdots d_4\ll S^4$, 
it now follows that  
\begin{align*}
\sum_{n\leq N_1}\frac{1}{n^{\sigma_Y}}T_{\dd}(n;S,R)
&\ll_\ve (d_1\cdots d_4)^\ve\sum_{n_0\leq N_1} \frac{1}{n_0} \left(
\frac{SR}{n_0^{\frac{1}{2}-\varepsilon}d_1\cdots d_4}+Rn_0^{\frac{1}{2}+\varepsilon}
\right)\\
&\ll_\ve \frac{SR}{(d_1\cdots d_4)^{1-\ve}}+R N_1^{\frac{1}{2}+\varepsilon}S^{4\ve}.
\end{align*}

We may now record our final estimate for \eqref{eq:Nprimei1dd}.
Combining the previous line with \eqref{eq:largeN1} and \eqref{eq:white-bag}, 
and recalling 
the choice of $N_1$ made in \eqref{eq:N1}, 
we therefore deduce that 
\begin{align*}
N_{\dd}(S,R,Y)
&\ll_{\ve, \theta}
 \frac{SR}{(d_1\cdots d_4)^{1-\ve}}+S+
S^{1+\ve} N_1^\ve+
R N_1^{\frac{1}{2}+\varepsilon}S^{4\ve} \\
&\ll_{\ve, \theta}
 \frac{SR}{(d_1\cdots d_4)^{1-\ve}}+
S^{1+2\ve} +
R (S^3R)^{\theta}S^{8\ve}.
\end{align*}
On rescaling  $\varepsilon$, we finally obtain
\begin{equation}\label{eq:NdiSRY}
	N_{\dd}(S,R,Y)\ll_{\varepsilon,\theta} \frac{SR}{(d_1\cdots d_4)^{1-\ve}}+ S^{\varepsilon}\left(S+S^{3\theta}R^{1+\theta}\right),
\end{equation}
for any  $\theta>\frac{3}{16}$.

We are now ready to sum over all dyadic intervals in \eqref{eq:varSR}. 
Inserting \eqref{eq:NdiSRY}, it follows that 
\begin{align*}
N_{\dd}(X,Y)
&\ll_{\ve,\theta}
\sum_{S\nearrow X}\sum_{R \nearrow S}\left(\frac{S^\frac{5}{8}Y^\frac{4}{3}}{R^\frac{5}{8}}+\frac{Y^2}{S^\frac{1}{8}R^\frac{7}{8}}\right)
\left(\frac{SR}{(d_1\cdots d_4)^{1-\ve}}+ S^{\varepsilon}\left(S+S^{3\theta}R^{1+\theta}\right)\right)\\
& \ll_{\varepsilon,\theta} Y^\frac{4}{3}\left(\frac{X^2}{(d_1\cdots d_4)^{1-\ve}}+X^\varepsilon\left(X^{\frac{13}{8}}+X^{1+4\theta}\right)\right)\\
&\qquad +Y^2\left(\frac{X}{(d_1\cdots d_4)^{1-\ve}}+X^\varepsilon\left(X^{\frac{7}{8}}+X^{4\theta}\right)\right).
\end{align*}
Making  the choice $\theta=\frac{7}{32}$, we obtain
$$N_{\dd}(X,Y)\ll_{\varepsilon} \frac{X^2Y^\frac{4}{3}+XY^2}{(d_1\cdots d_4)^{1-\ve}}+X^\varepsilon\left(X^\frac{15}{8}Y^\frac{4}{3}+X^\frac{7}{8}Y^2\right)
$$
in \eqref{eq:pink-bag}.
It remains to sum over all $r,t$. Using the trivial bound for the divisor function $\tau_4$, we obtain 
$$
\sum_{\substack{r,t\in\BN\\ t~\square\text{-free},~t\ll X^4\\ r~\square\text{-full},~r\leqslant D}}\frac{r^{\frac{1}{4}+\varepsilon}}{t} \tau_4([r,t])\ll_\varepsilon X^\varepsilon D^{\frac{3}{4}+\varepsilon}.$$
Similarly, in view of the   lower bound
$[r,t]\geqslant \max(r,t)\geqslant r^\frac{7}{8} t^\frac{1}{8}$, we have
$$
\sum_{\substack{r,t\in\BN\\ t~\square\text{-free},~t\ll X^4\\ r~\square\text{-full},~r\leqslant D}}\frac{r^{\frac{1}{4}+\varepsilon}}{t} \frac{\tau_4([r,t])}{[r,t]^{1-\ve}}\ll_\ve
\sum_{\substack{r,t\in\BN\\ t~\square\text{-free},~t\ll X^4\\ r~\square\text{-full},~r\leqslant D}}\frac{1}{r^{\frac{5}{8}-2\ve}t^{\frac{9}{8}-2\ve}} 
\ll_\ve 1.
$$
The statement of Proposition \ref{prop:M1Xsmall} is now clear.
\end{proof}

\subsection{Proof of Proposition \ref{prop:MXXsmall}}\label{s:monats}

Recall the definition \eqref{eq:S(X)} of $\CS(X)$. 
Then we have 
$$
\#\CM(X,Y)=\sum_{\substack{\xx\in \CS(X)}}N(Q_\x;Y),
$$
where $Q_\x$ is given by \eqref{eq:Qx} and $N(Q_\x;Y)$ is the counting function defined in 
\eqref{eq:NQB}.
We continue to adopt the notation $\|Q_\x\|$ and $\varDelta(\xx)$ 
that was introduced in \eqref{eq:normLideltaLi}.
 In this section we see what can be deduced from an application of 
Lemma~\ref{co:uniformquad}. 
Letting  $m(Q_\xx)=\min_{1\leqslant i\leqslant 4} |L_i(\x)|$, we obtain
$$N(Q_\x;Y)\ll_\varepsilon 
	\frac{
	 \deltabad(\xx)^\varepsilon L(1,\chi_{Q_\x}) 	
	}{(m(Q_\x)\|Q_\x\|)^\frac{1}{2}}Y^2+\frac{\|Q_\x\|^{11+\varepsilon}}{m(Q_\x)^6|\varDelta(\x)|^\frac{1}{2}}Y^{\frac{3}{2}+\varepsilon},$$
	for any $\ve>0$.  Lemma \ref{le:smallvalueLi} implies that $|\Delta(\x)|^{\frac{1}{2}}\gg \|Q_\x\|^{\frac{3}{2}} m(Q_\x)^{\frac{1}{2}}.$ Since $\|Q_\x\|\ll X$, it follows that 
\begin{equation}\label{eq:yellow-shoe}
	\#\CM(X,Y)\ll_\ve Y^2\sum_{\substack{\xx\in \CS(X)}}
	\frac{\deltabad(\xx)^\varepsilon L(1,\chi_{Q_\xx})}{(m(Q_{\xx})\|Q_\xx\|)^\frac{1}{2}} +
	X^{\frac{19}{2}+\ve}Y^{\frac{3}{2}+\varepsilon}
\hspace{-0.2cm}
	\sum_{\substack{\xx\in \CS(X)}}\frac{1}{m(Q_{\xx})^{\frac{13}{2}}}.
	\end{equation}
But a standard dyadic decomposition procedure yields 
\begin{align*}
\sum_{\substack{\xx\in \CS(X)}}\frac{1}{m(Q_{\xx})^{\frac{13}{2}}}
&\ll \sum_{i_1=1}^4\sum_{S\nearrow X}\frac{1}{S^\frac{13}{2}}
\#\{\xx\in \CS(X):  |L_{i_1}(\xx)|\sim S\} \ll X.
\end{align*}
Thus the overall contribution from the second term is 
$\ll_\varepsilon X^{\frac{21}{2}+\varepsilon} Y^{\frac{3}{2}+\varepsilon}\ll X^{11} Y^{\frac{3}{2}+\varepsilon}.$

In what follows we focus on the first summand. Much as in the proof of Proposition \ref{prop:M1Xsmall}, we carry out two dyadic decompositions. Then, for fixed $i_1\in \{1,\dots,4\}$, we are reduced to estimating
\begin{equation}\label{eq:SRW}
	 \sum_{S\nearrow X}\sum_{R \nearrow S}\frac{1}{S^\frac{1}{2}R^\frac{1}{2}}
	 W_{i_1}(S,R),
\end{equation}
where
\begin{equation}\label{eq:defW}
	 W_{i_1}(S,R)=	 
	 \sum_{\substack{\xx\in \CS(X)\\ |\xx|\sim S,~|L_{i_1}(\xx)|\sim R}}\deltabad(\xx)^\varepsilon L(1,\chi_{Q_\xx}).
\end{equation}
One of the ingredients we will need in our treatment of 
$W_{i_1}(S,R)$ is a proof that there are relatively few $\x$ for which  
$\deltabad(\xx)
$ is large. This is achieved in the following result. 

\begin{lemma}\label{lem:stain}
Let $\delta\geq 0$ and let $i_1\in \{1,\dots,4\}$. Then 
$$
\#\left\{\xx\in \ZZp^2: |\xx|\sim S,~|L_{i_1}(\xx)|\sim R, ~\deltabad(\x)>(SR)^\delta\right\}
\ll (SR)^{1-\frac{\delta}{8}}.
$$
\end{lemma}
\begin{proof}
Let $N_\delta(S,R)$ denote the quantity that is to be estimated. First,
we observe that the condition $\deltabad(\xx)>(SR)^\delta$ implies that at least one of $L_1(\xx),\dots, L_4(\x)$ has a square-full part that exceeds 
$(SR)^\frac{\delta}{4}$. Let us assume that $L_{i_0}$ has  square-full part $>(SR)^\frac{\delta}{4}$, with $i_0\neq i_1$. Upon a non-singular change of variables, we  may therefore  assume that $L_{i_0}(\xx)=x_1$ and $L_{i_1}(\xx)=x_2$. But then, on summing trivially over $x_2$, we obtain 
 \begin{align*}
	N_\delta(S,R)&\ll R\sum_{\substack{(SR)^\frac{\delta}{4}\ll a\ll S\\ a~\square\text{-full}}}\sum_{\substack{x_1\in\BZ_{\neq 0}\\ |x_1|\leqslant S,~a\mid x_1}} 1
\ll (SR)^{1-\frac{\delta}{8}},
\end{align*}
which is satisfactory.  Alternatively, if  $L_{i_1}(\xx)$
is the term with large square-full part, then a similar manipulation yields
\begin{align*}
	N_\delta(S,R)&\ll  S\sum_{\substack{(SR)^\frac{\delta}{4}\ll a\ll R\\ a~\square\text{-full}}}\sum_{\substack{x_1\in\BZ_{\neq 0}\\ |x_1|\leqslant R,~a\mid x_1}} 1
	\ll (SR)^{1-\frac{\delta}{8}},
\end{align*}
which completes the proof of the lemma.
\end{proof}

We are now ready to estimate the  quantity in 
\eqref{eq:defW}.

\begin{lemma}\label{lem:Wi1}
Let $\delta\geq 0$. Then 
$$
	W_{i_1}(S,R)\ll_{\varepsilon}SR+S^\varepsilon\left((SR)^{\frac{\delta}{2}}\left(S+S^{\frac{21}{32}}R^{\frac{39}{32}}\right)+(SR)^{1-\frac{\delta}{8}}\right).
$$
\end{lemma}
\begin{proof}
For given $\delta\geq 0$, we write 
$$
W_{i_1}(S,R)=W_{i_1,1}(S,R;\delta)+W_{i_1,2}(S,R;\delta),
$$ 
where the sum $W_{i_1,1}(S,R;\delta)$ is subject to the condition $\deltabad(\xx)\leqslant (SR)^\delta$ and $W_{i_1,2}(S,R;\delta)$ has $\deltabad(\xx)>(SR)^\delta$.

Observe that 
\begin{equation}\label{eq:calculus}
\deltabad(\xx)^\varepsilon
L(1,\chi_{Q_\xx}) \ll  |\x|^{4\varepsilon}  \log\left(2+\|Q_\x\|\right)\ll_{\varepsilon} X^{5\varepsilon}.
\end{equation}
Rescaling $\varepsilon$ and applying Lemma \ref{lem:stain}, we get a  satisfactory bound 
for the term $W_{i_1,2}(S,R;\delta)$.
Turning to 
$W_{i_1,1}(S,R;\delta)$, we  begin in the same way as in the proof of 
Proposition \ref{prop:M1Xsmall}.  Thus 
\begin{align*}
	W_{i_1,1}(S,R;\delta)&\leqslant \sum_{\substack{d\leqslant (SR)^\delta\\ d~\square\text{-full}}}d^\varepsilon\sum_{\substack{\dd\in \NN^4\\ d=d_1\cdots d_4}}N^{\prime}_{i_1,\dd}(S,R),
\end{align*}
where
$$
N^{\prime}_{i_1,\dd}(S,R)=\sum_{\substack{\xx\in \CS(X)\\ 
|\xx|\sim S,~|L_{i_1}(\xx)|\sim R\\ d_{i}\mid L_i(\xx)}}L(1,\chi_{Q_\xx}).
$$
This is essentially the same sum 
that we already met in \eqref{eq:Nprimei1dd} and we can directly apply  the bound \eqref{eq:NdiSRY}. 
Taking $\theta=\frac{7}{32}$, we get
$$
	N^{\prime}_{i_1,\dd}(S,R)\ll_{\varepsilon} \frac{SR}{(d_1\cdots d_4)^{1-\ve}}+ S^{\varepsilon}\left(S+S^{\frac{21}{32}}R^{\frac{39}{32}}\right),
$$
whence 
$
W_{i_1,1}(S,R;\delta)\ll_{\varepsilon} 
SR+S^\varepsilon(SR)^{\frac{\delta}{2}}(S+S^{\frac{21}{32}}R^{\frac{39}{32}}).
$
\end{proof}

Inserting Lemma \ref{lem:Wi1} into \eqref{eq:SRW} and 
summing over dyadic intervals, we are now led to the  conclusion that  
$$
\sum_{S\nearrow X}\sum_{R \nearrow S}\frac{1}{S^\frac{1}{2}R^\frac{1}{2}}W_{i_1}(S,R)\ll_{\varepsilon} X+X^\varepsilon\left(X^{\frac{7}{8}+\delta}+X^{1-\frac{\delta}{4}}\right).$$
Taking $\delta=\frac{1}{10}$, the right hand side is $O(X)$, on taking  $\varepsilon$ sufficiently small. Therefore the overall contribution of the first term in 
\eqref{eq:yellow-shoe}
is $O(XY^2)$, thereby completing the proof of Proposition \ref{prop:M1Xsmall}.

\subsection{Proof of Proposition \ref{prop:poisson}}\label{s:poisson}
 
Our proof of Proposition \ref{prop:poisson} relies on viewing the equation \eqref{eq:quadbundle} in the form 
\eqref{eq:xQ}, where $Q_1,Q_2$ are the diagonal quadratic forms defined in 
\eqref{eq:Q1Q2}, where  $\gcd(a_i,b_i)=1$ for $1\leq i\leq 4$.  As usual, we proceed under the assumption that the pencil $Q_1=Q_2=0$ defines a smooth curve $Z\subset \PP^3$ 
of genus $1$, 
such that $Z(\RR)= \emptyset$.

We will be led to make crucial use of properties of the lattice $\Lambda_{[\uu],k}$ that was defined in \eqref{eq:lattice-def}. In particular, we will need to show that its successive minima are not typically 
lop-sided.
We begin, however,  with the following basic estimate.

\begin{lemma}
Let $\ve>0$ and let $D,E>0$ such that $1\ll E\ll D^{3/4}.$ Then 
$$
\sum_{d\asymp D}\sum_{\substack{\vv\in \ZZ^4\\|\bfv|\leq  E\\d\mid Q_{i}(\bfv), ~i=1,2}}1\ll 
E^{2+\ve}
+
\frac{E^{4}}{D}\cdot\frac{\log E}{\log\log E}.
$$
\label{lem : gcd}
\end{lemma}
\begin{proof}
If the left hand side is non-zero, then there exists a vector $\vv\in \ZZ^4$ such that 
$|\vv|\leq E$ and  $d\mid Q_i(\vv)$ for $i=1,2$. But then 
$D\ll E^2$, since 
$Q_1(\vv), Q_2(\vv)$ cannot both  vanish. Thus we may proceed under the assumption that $D\ll E^2$.

We  sort the sum on the left according to the value of $\gcd(Q_{1}(\bfv), Q_{2}(\bfv))$. This gives
\[
\sum_{d\asymp D}\sum_{\substack{|\bfv|\leq E\\d|Q_{i}(\bfv)\text{ for }i=1,2}}1=
\sum_{c=0}^{\log (ME^{2}/D)}\sum_{g\sim 2^{c}D}\tau_{D}(g)\sum_{\substack{|\bfv|\leq E\\g=\gcd(Q_{1}(\bfv),Q_{2}(\bfv))}}1,
\]
where $\tau_{D}(g)=\#\{d\asymp D: d\mid g\}$ and 
$
M=\max\left(|a_{1}|+\cdots+|a_4|,|b_{1}|+\cdots+|b_4|\right).
$ 
We claim that 
$$
\tau_{D}(g)\ll \min\left( 2^c, \e^{\frac{2\log E}{\log\log E}}\right).
$$
Indeed, if $g\in [2^{c}D,2^{c+1}D)$ and $g=fd$ for some $d\asymp D$, then $2^{c}\ll f\leq 2^{c}$. This implies that $\tau_{D}(g)\ll 2^{c}$. On the other hand, 
$$
\tau_{D}(g)\leq\tau (g)\leq \e^{(\log2+o(1))\frac{\log g}{\log\log g}}\ll \e^{\frac{\log g}{\log\log g}},
$$ 
by  Tenenbaum \cite[Thm.~5.4]{Te15}, for  example. The claim follows, since  $g\leq ME^{2}$.

Next, we observe that if $g=\gcd(Q_{1}(\bfv),Q_{2}(\bfv))$ then 
\[
\frac{Q_{1}(\bfv)}{g}\cdot Q_{2}(\bfv)=Q_{1}(\bfv)\cdot\frac{Q_{2}(\bfv)}{g},
\]
where $\gcd(g^{-1}Q_{1}(\bfv),g^{-1}Q_{2}(\bfv))=1$. Moreover, we also have 
\[
\max\left(\frac{|Q_{1}(\bfv)|}{g},\frac{|Q_{2}(\bfv)|}{g}\right)\leq \frac{ME^{2}}{2^{c}D}.
\] 
It follows that 
\begin{equation}\label{eq:paper}
\sum_{g\sim 2^{c}D}\sum_{\substack{|\bfv|\leq E\\g=\gcd(Q_{1}(\bfv),Q_{2}(\bfv))}}1
\leq 
\#\CM^*\left(
 \frac{ME^{2}}{2^{c}D}, E\right),
\end{equation}
in the notation of \eqref{eq:MX*}. 
Hence Corollary \ref{cor:M*} yields
\begin{equation}\label{eq:sofa}
\begin{split}
\#\CM^*\left(
 \frac{ME^{2}}{2^{c}D}, E\right)
&\ll_\ve
E^{2+\ve}
+
\frac{E^{4}}{2^{c}D}+
\frac{E^{16/3}}{4^{c}D^2},
\end{split}
\end{equation}
for any $\ve>0$, 
provided that 
$
ME^{2}/(2^{c}D)\leq E^{2/3}\log E.
$
The latter is equivalent to $ME^{4/3}\leq 2^c D\log E$, which is implied by the hypothesis of the lemma. 

We shall argue differently according to the size of $c$.  Let $L>0$ be a parameter to be selected in due course. For small $c$, it follows from \eqref{eq:paper} and  \eqref{eq:sofa}  that 
\begin{align*}
\sum_{c=0}^L\sum_{g\sim 2^{c}D}\tau_D(g)\sum_{\substack{|\bfv|\leq E\\g=\gcd(Q_{1}(\bfv),Q_{2}(\bfv))}}1
&\ll_\ve  \sum_{c=0}^L 2^c 
\left(E^{2+\ve}
+
\frac{E^{4}}{2^{c}D}+
\frac{E^{16/3}}{4^{c}D^2}\right)\\
&\ll_\ve  
2^L E^{2+\ve}
+
\frac{LE^{4}}{D}+
\frac{E^{16/3}}{D^2}.
\end{align*}
On the other hand, for the remaining  $c$, the \eqref{eq:paper} and \eqref{eq:sofa} yield
\begin{align*}
\sum_{c>L}\sum_{g\sim 2^{c}D}\tau_D(g)\sum_{\substack{|\bfv|\leq E\\g=\gcd(Q_{1}(\bfv),Q_{2}(\bfv))}}1
&\ll_\ve \e^{\frac{2\log E}{\log\log E}} \sum_{c>L} 
\left(E^{2+\ve}
+
\frac{E^{4}}{2^{c}D}+
\frac{E^{16/3}}{4^{c}D^2}\right)\\
&\ll \e^{\frac{2\log E}{\log\log E}}\left( E^{2+\ve} \log E+
\frac{E^{4}}{ 2^L D}+
\frac{E^{16/3}}{4^LD^2}
\right).
\end{align*}
We shall take  
 $L=\frac{4\log E}{\log\log E}$.
 It now follows that
 $$
\sum_{d\asymp D}\sum_{\substack{|\bfv|\leq E\\d|Q_{i}(\bfv)\text{ for }i=1,2}}1\ll_\ve 
E^{2+2\ve}
+
\frac{E^{4}}{D}\cdot\frac{\log E}{\log\log E}+
\frac{E^{16/3}}{D^2},
$$ 
since 
$c^L\log E=O_\ve (E^\ve)$ for any $c\geq 1$. Now $E^{16/3}/D^2\leq E^4/D$ if $E\leq D^{3/4}$. Hence  the lemma follows on redefining the choice of $\ve$.
\end{proof}

We can use this result to assess the average size of the smallest successive minimum 
$s_{1,[\uu],k}$ of the lattice 
 lattice $\Lambda_{[\uu],k}$ that was defined in \eqref{eq:lattice-def}, for $k\in \NN$ and 
 $[\uu]\in V_k^\times$, in the notation of  \eqref{eq:def-Vd}.
 At this point it is convenient to recall the notation \eqref{eq:delta-notation}, where $\Delta$
is the product of bad primes defined in \eqref{eq:DELTA}. The following result is rather general, since it will be used in more than one context in what follows. 

\begin{lemma}
Let $f,e,m,D\in \NN$ and assume that $m\mid e$. Then
\[
\sum_{\substack{d\asymp D\\f|d}}\sum_{[\uu]\in V_{de}^{\times}}\frac{1}{s_{1,[\uu],\frac{de}{mf}}^{3}}\ll
m_{\Delta} f_{\Delta}\cdot(mf)^{\frac{5}{4}}\cdot
(De)^{-\frac{1}{4}}\frac{\log De}{\log\log De}.
\]
\label{lem : sucmin}
\end{lemma}

\begin{proof}
Let $\vv\in \ZZ^4$ be a non-zero vector in the lattice 
$\Lambda_{[\uu],\frac{de}{mf}}$, with Euclidean length equal to 
$s_{1,[\uu],\frac{de}{mf}}$. This implies that 
$\frac{de}{mf} \mid Q_i(\vv)$ for $i=1,2$. 
Since we cannot have $Q_1(\vv)=Q_2(\vv)=0$, so it follows that 
$\frac{de}{mf}\ll |\vv|^2$. Once combined with  \eqref{eq:mink}, this therefore implies that  the smallest successive minimum of 
$\Lambda_{[\uu],\frac{de}{mf}}$ satisfies
$$
 \left(\frac{de}{mf}\right)^{\frac{1}{2}}\ll 
s_{1,[\uu],\frac{de}{mf}}
\ll \left(\frac{de}{mf}\right)^{\frac{3}{4}}.
$$ 
Splitting the sum in the lemma into dyadic intervals, we obtain
\begin{equation}\label{eq:S(E)}
\sum_{\substack{d\asymp D\\f|d}}\sum_{[\uu]\in V_{de}^{\times}}\frac{1}{s_{1,[\uu],\frac{de}{mf}}^{3}}\ll\sum_{ \left(\frac{De}{mf}\right)^{\frac{1}{2}} \swarrow E\nearrow \left(\frac{De}{mf}\right)^{\frac{3}{4}}}\frac{S(E)}{E^{3}},
\end{equation}
where
\begin{align*}
S(E)=
\sum_{\substack{d\asymp D\\f|d}}\sum_{\substack{[\uu]\in V_{de}^{\times}\\s_{1,[\uu],\frac{de}{mf}}\sim E}}1&\leq
 \sum_{\substack{\vv\in \ZZ^4\\|\bfv|\sim E}}\sum_{\substack{d\asymp D\\ f|d}}\sum_{\substack{[\uu]\in V_{de}^{\times}\\ \bfv\in\Lambda_{[\uu],\frac{de}{mf}}}}1\\&\leq\sum_{d'\asymp \frac{D}{f}}\sum_{\ell\mid \frac{d'e}{m}}\sum_{\substack{|\bfv|\sim E\\\gcd(\bfv,\frac{d'e}{m})=\ell\\ \frac{d'e}{m}\mid Q_{i}(\bfv),~i=1,2}}\text{}\sum_{\substack{[\uu]\in V_{d'ef}^{\times}\\
\bfv\in \Lambda_{[\uu],\frac{d'e}{m}}}}
1\\
&=\sum_{d'\asymp \frac{D}{f}}\sum_{\ell|\frac{d'e}{m}}\sum_{\substack{|\bfv'|\sim E/\ell\\\gcd(\bfv',\frac{d'e}{\ell m})=1\\ \frac{d'e}{m}\mid Q_{i}(\ell\cdot \bfv'),~i=1,2}}\text{}\sum_{\substack{[\uu]\in V_{d'ef}^{\times}\\\ell\cdot \bfv'\in \Lambda_{[\uu],\frac{d'e}{m}}}}1.
\end{align*}
We  observe that if $\ell\cdot\bfv '\in \Lambda_{[\uu],\frac{d'e}{m}}$, with $\ell\mid \frac{d'e}{m}$ and $
\gcd(\bfv',\frac{d'e}{\ell m})=1$, then there exists $\lambda\in\mathbb{Z}$ such that $\ell\cdot\bfv'\equiv \lambda\cdot\uu\bmod{\frac{d'e}{m}}$, which  implies that $\lambda=\ell\lambda'$ for some $\lambda'\in\mathbb{Z}$ which is coprime with $\frac{d'e}{\ell m}$. But then it follows that 
 $\bfv'\equiv \lambda'\cdot\uu \bmod{\frac{d'e}{\ell m}}$ and so  $Q_{i}(\bfv')\equiv  0\bmod{\frac{d'e}{\ell m}}$, for $i=1,2$. Hence
 \[
S(E)\leq\sum_{d'\asymp \frac{D}{f}}\sum_{\ell|\frac{d'e}{m}}\sum_{\substack{|\bfv'|\sim E/\ell\\
\gcd(\bfv',\frac{d'e}{\ell m})=1\\ \frac{d'e}{\ell m}\mid Q_{i}(\bfv'),~i=1,2}}
\sum_{\substack{[\uu]\in V_{d'ef}^{\times}\\\bfv'\in \Lambda_{[\uu],\frac{d'e}{\ell m}}}}1.
\]

The first part of 
Lemma  \ref{lem : loc''} implies that the inner sum is $O(  (\ell_\Delta m_{\Delta}  f_{\Delta})\cdot(\ell m f))$.
If write $M_{m,f}=m_{\Delta} f_{\Delta}m f$, then it follows that 
\[
\begin{split}
S(E)&\ll M_{m,f}
\sum_{d'\asymp \frac{D}{f}}\sum_{\ell|\frac{d'e}{m}}
\sum_{\substack{|\bfv'|\sim E/\ell\\\gcd(\bfv',\frac{d'e}{\ell m})=1\\ \frac{d'e}{\ell m}\mid 
Q_{i}(\bfv'),~i=1,2}}
\ell_{\Delta}\ell
\ll M_{m,f}
\sum_{\substack{k\in\NN\\
p\mid k\Rightarrow p\mid \Delta}}
\sum_{\substack{\ell\ll E\\\ell_{\Delta}=k}}k \ell
\sum_{\substack{d'\asymp\frac{D}{f}\\ \ell|\frac{d'e}{m}}}
\sum_{\substack{|\bfv'|\sim E/\ell\\\gcd(\bfv',\frac{d'e}{\ell m})=1\\ \frac{d'e}{\ell m}\mid 
Q_{i}( \bfv'),~i=1,2}}1.
\end{split}
\]
It now follows from
 Lemma \ref{lem : gcd} that 
the inner sum is 
\begin{align*}
\sum_{\substack{d'\asymp\frac{D}{f}\\ \ell|\frac{d'e}{m}}}
\sum_{\substack{|\bfv'|\sim E/\ell\\\gcd(\bfv',\frac{d'e}{\ell m})=1\\ \frac{d'e}{\ell m}\mid 
Q_{i}(\bfv'),~i=1,2}}1
&\leq
\sum_{\substack{d''\asymp\frac{De}{\ell mf}}}
\sum_{\substack{|\bfv'|\sim E/\ell\\ d '' \mid 
Q_{i}( \bfv'),~i=1,2}}1
\ll_\ve \left(\frac{E}{\ell}\right)^{2+\ve}+
 \frac{mfE^{4}}{\ell^{3} De}\cdot\frac{\log E}{\log\log E},
\end{align*}
for any $\ve>0$,
 since clearly 
$$
\frac{E}{\ell}\ll \left(
\frac{De}{\ell mf}
\right)^{3/4}
$$
in \eqref{eq:S(E)}.
Thus 
$$
S(E)\ll_\ve 
 M_{m,f} 
 \left(
  E^{2+\ve} U(1+\ve)
+\frac{mf  E^4}{De}\cdot \frac{\log E}{\log\log E}\cdot U(2)\right),
$$
where
$$
U(\theta)=
\sum_{\substack{k\in\NN\\
p\mid k\Rightarrow p\mid \Delta}}
\sum_{\substack{\ell\ll E\\\ell_{\Delta}=k}}\frac{k}{\ell^\theta}.
$$
Finally, for any $\theta>1$, we note that 
$$
U(\theta)\leq
\sum_{\substack{k \ll E\\
p\mid k\Rightarrow p\mid \Delta}}
\sum_{\substack{\ell'\ll E}}\frac{k^{1-\theta}}{{\ell'}^\theta}\ll
\sum_{\substack{k \ll E\\
p\mid k\Rightarrow p\mid \Delta}}
k^{1-\theta}\ll_\theta 1.
$$
Hence
$$
S(E)\ll_\ve 
 M_{m,f} 
 \left(
  E^{2+\ve} 
+\frac{mf  E^4}{De}\cdot \frac{\log E}{\log\log E}\right).
$$
Returning to \eqref{eq:S(E)} and summing over dyadic intervals for $E$, the statement of the lemma easily follows. 
\end{proof}

Armed with the previous facts about our lattices $\Lambda_{[\uu],k}$, 
we are now ready to establish the following result, which is  a critical step towards
Proposition \ref{prop:poisson}.

\begin{lemma}
Let $\ve>0$, let $D\geq 1$ and  let $e\in \NN$. Then
\[
\sum_{d\asymp D}\sum_{\substack{\y\in \ZZp^4\\ 
|\bfy|\leq Y\\de|Q_{i}(\bfy),~i=1,2}}1\ll_{\varepsilon} \frac{Y^{4}}{De^{2-\varepsilon}}+\frac{Y^3}{(De)^{\frac{1}{4}}}\frac{\log De}{\log\log De}+D^{2}e^{1+\varepsilon}.
\]
\label{lem : upp}
\end{lemma}
\begin{proof}
Dropping the primitivity condition, we first note that 
\[
\begin{split}
\sum_{d\asymp D}\sum_{\substack{\y\in \ZZp^4\\ 
|\bfy|\leq  Y\\de|Q_{i}(\bfy),~i=1,2}}1
&\leq\sum_{d\asymp D}\sum_{[\uu]\in V_{de}^{\times}}\sum_{\substack{|\bfy|\leq Y\\ \bfy\in\Lambda_{[\uu],de}}}1,
\end{split}
\]
where the inner sum is now over all $\y\in \ZZ^4$.  Recalling that $\Lambda_{[\uu],de}$ is an integer lattice of rank $4$ and determinant $(de)^3$, it follows from a lattice point counting result  due to Schmidt
 \cite[Lemma $2$]{Sc68} that 
\begin{align*}
\sum_{\substack{|\bfy|\leq Y\\ \bfy\in\Lambda_{[\uu],de}}}1
&\ll\frac{Y^{4}}{(de)^{3}}+\sum_{j=1}^{3}\frac{Y^{j}}{s_{1,[\uu],de}\cdots s_{j,[\uu],de}}+1
\ll\frac{Y^{4}}{(de)^{3}}+\frac{Y^{3}}{s_{1,[\uu],de}^{3}}+1,
\end{align*}
where $1\leq s_{1,[\uu],de}\leq \cdots \leq s_{4,[\uu],de}$ are the successive minima of $\Lambda_{[\uu],de}$. 

Taking $f=m=1$ in 
Lemma \ref{lem : sucmin}, we obtain 
\[
\sum_{d\asymp D}\sum_{[\uu]\in V_{de}^{\times}}
\frac{Y^{3}}{s_{1,[\uu],de}^{3}}\ll 
\frac{Y^3}{(De)^{\frac{1}{4}}}\frac{\log De}{\log\log De}.
\]
Moreover, 
\[
\begin{split}
\sum_{d\asymp D}\sum_{[\uu]\in V_{de}^{\times}}\left(\frac{Y^{4}}{(de)^{3}}+1\right)
\ll \left( Y^{4}+D^3e^{3}\right)\sum_{d\asymp D}\frac{\#V_{de}^{\times}}{(de)^{3}}.
\end{split}
\]
To estimate $\#V_{de}^{\times}$ we may appeal to  the second part of Lemma \ref{lem : loc''}, which 
gives
\begin{align*}
\#V_{de}^{\times}
&\ll_\ve 
d_{\Delta}^{\varepsilon}e^{1+\varepsilon}\cdot
d\cdot \prod_{\substack{p|d\\p\nmid \Delta}}\left(1+O(p^{-\frac{1}{2}})\right),
\end{align*}
for any $\varepsilon >0$. Hence
$$
\sum_{d\asymp D}\frac{\#V_{de}^{\times}}{(de)^{3}}
\ll_\ve \frac{1}{e^{2-\varepsilon}}\sum_{d\asymp D}\frac{d_{\Delta}^{\varepsilon}}{d^{2}}\cdot\prod_{\substack{p\mid d\\p\nmid \Delta}}\left(1+O(p^{-\frac{1}{2}})\right)
\ll_\ve \frac{1}{D e^{2-\varepsilon}}.
$$
The statement of the lemma  follows on collecting together the various estimates.
\end{proof}

Combing the latter with our earlier work, we can now record the following bound.

\begin{corollary}
Let $D,Y\geq 1$. Then
$$
\sum_{d\asymp D}\sum_{\substack{\y\in \ZZp^4\\ 
|\bfy|\leq Y\\d \mid Q_{i}(\bfy),~i=1,2}}1\ll_\ve Y^{2+\ve}+
 \frac{Y^{4}}{D}\log YD,
$$
\label{cor : M}
for any $\ve>0$.
\end{corollary}

\begin{proof}
If $D\ll Y^{\frac{4}{3}}$, then we apply Lemma \ref{lem : upp} with $e=1$. Otherwise, 
if $D\gg Y^{\frac{4}{3}}$, 
the desired bound is a consequence of 
 Lemma \ref{lem : gcd}.
This completes the proof.
\end{proof}

We now have all the tools in place to complete the proof of 
 Proposition \ref{prop:poisson}.
On appealing to Lemma \ref{lem:neron} and 
breaking the range of $|\y|$ into dyadic intervals,
we find that 
$$
\#\CM(X,Y)\ll_\ve X^\ve Y^{2+\ve} +
\sum_{Y_0\nearrow Y}
\sum_{\substack{\bfy\in\mathbb{Z}_{\text{prim}}^{4}\\ |\bfy|\sim Y_0}}M(X,\bfy),
$$
for any $\ve>0$, 
where
$
M(X,\bfy)=\#\left\{\bfx\in\ZZ_{\text{prim}}^{2}: \text{\eqref{eq:xQ} holds}, ~|\x|\sim  X\right\}.
$
Since $\gcd(x_1,x_2)=1$,  \eqref{eq:xQ} implies that 
$
x_{1}=\pm Q_{2}(\bfy)/d$ and $x_{2}=\pm Q_{1}(\bfy)/d$, 
 where $d=\gcd(Q_{1}(\bfy),Q_{2}(\bfy))$. In particular, we must have 
$\max(|Q_{1}(\bfy)|,|Q_{2}(\bfy)|)\sim Xd$.
 Let
$$
C=\inf_{\mathbf{t}\in\RR^4, |\mathbf{t}|\sim 1}\max\left(|Q_{1}(\mathbf{t})|,|Q_{2}(\mathbf{t})|\right).
$$
Our assumption that $Z(\RR)=\emptyset$ implies that  $C>0$, for a constant $C$ that depends only on the coefficients of $Q_1,Q_2.$ It follows from homogeneity that 
$
\max(|Q_{1}(\bfy)|,|Q_{2}(\bfy)|)\asymp Y_0^2,
$ 
for any $|\y|\sim Y_0$.
In this way we deduce that
$$
\#\CM(X,Y)\ll_\ve X^\ve Y^{2+\ve} +
\sum_{Y_0\nearrow Y}
\sum_{d\asymp D}
\sum_{\substack{\bfy\in\mathbb{Z}_{\text{prim}}^{4}\\ |\bfy|\sim  Y_0\\ 
d\mid Q_{i}(\bfy),~i=1,2}}1, 
$$
for any $\ve>0$, 
where $D=Y_0^{2}/X$.
We now apply 
Lemma \ref{lem : upp}, which gives
\begin{align*}
\#\CM(X,Y)&\ll_\ve X^\ve Y^{2+\ve} +
\sum_{Y_0\nearrow Y} \left(\frac{Y_0^{4}}{D}+\frac{Y_0^{3}}{D^{\frac{1}{4}}}\frac{\log D}{\log\log D}+D^{2}\right)\\
&=X^\ve Y^{2+\ve}
\sum_{Y_0\nearrow Y}\left( XY_0^2+
X^{\frac{1}{4}}Y_0^{\frac{5}{2}}\frac{\log Y_0}{\log\log Y_0}+\frac{Y_0^4}{X^2}\right),
\end{align*}
on taking $\ve=\frac{1}{4}$ and $D=Y_0^2/X$. Proposition~\ref{prop:poisson} readily follows on summing  over $Y_0$.

\section{Asymptotics via the geometry of numbers}\label{sec:geometry}

The purpose of this section is to prove Proposition \ref{pro:L1}, which provides an asymptotic formula for
$$
\#\CL_1(B)
=\#\left\{(\xx,\yy)\in \CL(B):
B^{\frac{1}{4}+\eta}\leq |\x|\leq 
B^{\frac{1}{2}-\eta}
 \right\},
$$
where  $\CL(B)$ is given by  \eqref{eq:L}. 
In particular, we note that any $(\x,\y)$ counted by 
$\#\CL_1(B)$ must satisfy 
$|\y|\leq B^{\frac{3}{8}-\frac{\eta}{2}}$. Breaking into dyadic intervals, we therefore find that 
$$
\#\CL_1(B)
=
\hspace{-0.3cm}\sum_{Y\nearrow B^{\frac{3}{8}-\frac{\eta}{2}}}
\hspace{-0.3cm}
\#
\left\{(\xx,\yy)\in\BZprim^2\times\BZprim^4:	
\begin{array}{l}
\text{\eqref{eq:quadbundle} and \eqref{eq:notsquare} hold},~ |\y|\sim Y\\
	B^{\frac{1}{4}+\eta}\leq |\xx|\leq  \min(B/|\y|^2, B^{\frac{1}{2}-\eta}) 
	\end{array}\right\}.
$$
We claim that there is a satisfactory overall contribution from $Y$
in the range $Y\leq B^{\frac{1}{4}+\frac{\eta}{2}}.$
But it follows from Theorem 
\ref{t:upper} that this contribution is 
$$\ll \sum_{Y\nearrow B^{\frac{1}{4}+\frac{\eta}{2}}}
\left(B^{\frac{1}{2}-\eta}\cdot Y^2 +\left(\frac{B}{Y^2}\right)^{\frac{1}{4}}Y^{\frac{5}{2}}
\frac{\log B}{\log\log B}\right)\ll
B,
$$
on breaking the sum over $\x$ into dyadic intervals.
We can use Lemma \ref{lem:neron} to handle the overall contribution from those $\x,\y$ for which 
\eqref{eq:notsquare} fails. Hence 
\begin{equation}\label{eq:L1-step1}
\#\CL_1(B)
=
\sum_{B^{\frac{1}{4}+\frac{\eta}{2}}\swarrow Y\nearrow B^{\frac{3}{8}-\frac{\eta}{2}}}\#\tilde\cM(Y) + O(B),
\end{equation}
where 
\begin{equation}\label{eq:MXXX}
	\tilde\cM(Y)=\left\{(\xx,\yy)\in\BZprim^2\times\BZprim^4:	
	\begin{array}{l}
	\text{\eqref{eq:quadbundle} holds},~ |\y|\sim Y\\
B^{\frac{1}{4}+\eta}\leq|\xx|\leq B/|\y|^2
\end{array}
\right\}.
\end{equation}
The main goal of this section is to produce the following asymptotic formula for the cardinality of 
$\tilde\cM(Y)$.

\begin{proposition}\label{pro:latt}
Let $Y\geq 1$ such that $
B^{\frac{1}{4}+\frac{\eta}{2}}\ll 
Y\ll B^{\frac{3}{8}-\frac{\eta}{2}}$.
Then
$$
\#\tilde\cM(Y)=2
\mathfrak{S}_1 B
\left(\sigma_\infty(Y)+O\left(\frac{Y^2}{B^{\frac{3}{4}-\eta}}\right)\right)
 +O\left(\frac{B}{\sqrt{\log B}} \right),
$$
where 
$\mathfrak{S}_1$ is given by \eqref{eq:series} and
\begin{equation}\label{eq:path}
\sigma_\infty(Y)=\int_{\substack{\y\in \RR^4\\ |\y|\sim Y}} \frac{\mathrm{d} \y}{|\y|^2\max(|Q_1(\y)|,|Q_2(\y)|)} .
\end{equation}
\end{proposition}

We may insert this result into \eqref{eq:L1-step1}, noting  that 
$
\sum_{Y\nearrow B^{\frac{3}{8}-\frac{\eta}{2}}}
Y^2/B^{\frac{3}{4}-\eta}\ll 1.
$
Hence we obtain
$$
\#\CL_1(B)
=2\mathfrak{S}_1 B\cdot 
\sum_{B^{\frac{1}{4}+\frac{\eta}{2}}\swarrow Y\nearrow B^{\frac{3}{8}-\frac{\eta}{2}}} 
\sigma_\infty(Y)
 +O\left(B\sqrt{\log B} \right).
$$
This therefore completes the proof of Proposition \ref{pro:L1}, subject to 
Proposition \ref{pro:latt}.

\subsection{Preliminary steps}

We now turn to the task of estimating the cardinality of 
\eqref{eq:MXXX}.
Rewriting \eqref{eq:quadbundle} as \eqref{eq:xQ} and extracting the greatest  common divisor $d$ of $Q_1(\y)$ and $Q_2(\y)$, we begin our proof of Proposition \ref{pro:latt}
by writing
$$
\#\tilde\cM(Y)=
\sum_{d=1}^\infty\sum_{\substack{\bfy\in \ZZp^4\cap \cB(Y,d)\\d=\gcd(Q_{1}(\bfy),Q_{2}(\bfy))}}2,
$$
where the factor $2$ corresponds to the  possible parameterisations 
$
(x_1,x_2)=d^{-1}(Q_2(\y),Q_1(\y))$  and $(x_1,x_2)=-d^{-1}(Q_2(\y),Q_1(\y))$,
and we have put 
\begin{equation}\label{eq:BXY}
\mathcal{B}(Y,d)=\left\{\bfy\in\mathbb{\RR}^{4}: |\bfy|\sim Y,~ B^{\frac{1}{4}+\eta}\leq
\frac{\max(|Q_{1}(\bfy)|,|Q_2(\y)|)}{d}\leq  \frac{B}{|\y|^2}\right\}.
\end{equation}
The assumption  $Z(\RR)=\emptyset$ implies that $\max(|Q_{1}(\bfy)|,|Q_2(\y)|)\asymp Y^2$ for any 
$\y\in \RR^4$ such that
$|\y|\sim Y$.
Hence
$
D_1\ll d\ll D_2
$
where
\begin{equation}
\label{eq:D1D2}
D_1=\frac{Y^4}{B} \quad \text{ and }\quad  D_2=\frac{Y^2}{B^{\frac{1}{4}+\eta}}.
\end{equation}
It follows that 
$$
\#\tilde\cM(Y)=
\sum_{D_1\ll d\ll D_2}\sum_{\substack{\bfy\in \ZZp^4\cap \cB(Y,d)\\d=\gcd(Q_{1}(\bfy),Q_{2}(\bfy))}}2.
$$

Let $M,z>0$ be parameters to be chosen in due course and define $
P(z)=\prod_{p\leq z}p.
$
Let
\begin{equation}\label{eq:T1}
T_{1}=
\sum_{D_1\ll d\ll D_2}\sum_{\substack{\bfy\in \ZZp^4\cap \cB(Y,d)\\
d\mid Q_i(\y),~i=1,2\\
\gcd(d^{-1}Q_{1}(\bfy),d^{-1}Q_{2}(\bfy),P(z))=1}}1.
\end{equation}
Then 
we clearly have
\begin{equation}\label{eq:T1-3}
2T_{1}-T_{2}-T_{3}\leq \#\tilde{\mathcal{M}}(X,Y)\leq 2T_{1},
\end{equation}
where
\begin{align*}
T_{2}&=
\sum_{D_1\ll d\ll D_2}\sum_{\substack{\bfy\in \ZZp^4\cap \cB(Y,d)\\
d\mid Q_i(\y),~i=1,2\\
\exists p\in(z,M] \text{ s.t. } p\mid  \frac{Q_{i}(\bfy)}{d},~i=1,2}}1, \qquad
T_{3}=
\sum_{D_1\ll d\ll D_2}\sum_{\substack{\bfy\in \ZZp^4\cap \cB(Y,d)\\
d\mid Q_i(\y),~i=1,2\\
\exists p>M \text{ s.t. } p\mid  \frac{Q_{i}(\bfy)}{d},~i=1,2}}1.
\end{align*}
We shall produce upper bounds for $T_2$ and $T_3$, and  an asymptotic formula for $T_1$. 

\begin{lemma}\label{lem:T2T3}
Let $\ve>0$. Then 
\begin{align*}
T_{2}
&\ll_\eta 
\frac{B}{z^{1-\varepsilon}}+
B^{1-\frac{\eta}{2}}M^3
\quad \text{ and }\quad
T_{3}
\ll_\ve Y^{2+\ve}+ \frac{B(\log B)^{2}}{M}.
\end{align*}
\end{lemma}
\begin{proof}
We start by noting that $T_{2}\leq\sum_{p\in (z,M]}U_{p}$, where
\[
U_{p}= 
\sum_{D_1\ll d\ll D_2}\sum_{\substack{\bfy\in \ZZp^4\\ |\y|\sim Y\\ dp\mid Q_{i}(\bfy),~i=1,2}}1,
\]
for any prime  $p$.
We may use 
Lemma \ref{lem : upp} to estimate $U_p$, 
on breaking the $d$-sum into dyadic intervals. In this way, on recalling \eqref{eq:D1D2},  we see that 
\begin{align*}
T_2
&\ll_{\varepsilon} \sum_{p\in(z,M]}\left(\frac{Y^{4}}{D_1p^{2-\varepsilon}}+\frac{Y^{3}}{(D_1p)^{\frac{1}{4}}}\frac{\log D_2p}{\log\log D_2p}+D_2^{2}p^{1+\varepsilon}\right)\\
&\ll\frac{Y^{4}}{D_1z^{1-\varepsilon}} +\frac{Y^{3}M^{\frac{3}{4}}\log B}{D_1^{\frac{1}{4}}}+D_2^{2}M^{2+\varepsilon}\\
&\ll \frac{B}{z^{1-\varepsilon}}+B^{\frac{1}{4}}Y^{2}M^{\frac{3}{4}}\log B+\frac{Y^{4}M^{2+\varepsilon}}{B^{\frac{1}{2}+2\eta}},
\end{align*}
for any $\ve>0$.
Since $Y\ll B^{\frac{3}{8}-\frac{\eta}{2}}$, it follows that 
$$
B^{\frac{1}{4}}Y^{2}M^{\frac{3}{4}}\log B\ll 
B^{1-\eta}M^{\frac{3}{4}}\log B\ll_\eta B^{1-\frac{\eta}{2}}M^3.
$$
Similarly, 
$$
\frac{Y^{4}M^{2+\varepsilon}}{B^{\frac{1}{2}+2\eta}}
\ll  B^{1-4\eta} 
M^{2+\varepsilon}\ll B^{1-\frac{\eta}{2}}M^3,
$$
which completes the proof of the first part of the lemma.

We now turn to 
$
T_{3}
\leq\sum_{p>M}U_{p}.
$
If $pd\mid Q_i(\y)$ for $i=1,2$ then 
 there exists $q\gg D_1M$ such that $q\mid Q_{i}(\bfy)$ for $i=1,2$. 
 Under our assumption  $Z(\RR)=\emptyset$,  it follows that 
$$
T_{3}\ll \log Y \sum_{D_1M\ll q\ll Y^{2}}
\sum_{\substack{\bfy\in \ZZp^4\\ |\y|\sim Y\\ q\mid Q_{i}(\bfy),~i=1,2}}1,
$$
since there  $O(\log Y)$ primes  divisors of $q$. 
Splitting the range of summation over $q$ into dyadic intervals, 
Corollary \ref{cor : M} yields
\begin{align*}
T_3
&\ll \log Y\sum_{D_1M \swarrow  G\nearrow Y^{2}}
\left(Y^{2+\frac{\ve}{2}}+\frac{Y^{4}}{G}\log B\right)
\ll_\ve Y^{2+\ve}+
\frac{Y^{4}}{D_1M}(\log B)^2.
\end{align*}
Recalling \eqref{eq:D1D2}, 
this is also satisfactory for the lemma.
\end{proof}

\subsection{Asymptotic formula for  $T_{1}$}
It remains to deal with the sum \eqref{eq:T1}.
The first step is to reduce the  primitivity condition on $\bfy$ to the requirement that  $\gcd(\bfy,dP(z))=1$. This is the purpose of the following result. 

\begin{lemma}
Let $\ve>0$. Then 
\[
T_{1}=\Sigma_1
+ O_\ve\left( Y^{2+\ve}+\frac{B\log B}{z^{3}}\right),
\]
where
$$
\Sigma_1=
\sum_{D_1\ll d\ll D_2}\sum_{\substack{\bfy\in \ZZ^4\cap \cB(Y,d)\\
\gcd(\y,dP(z))=1\\
d\mid Q_i(\y),~i=1,2\\
\gcd(d^{-1}Q_{1}(\bfy),d^{-1}Q_{2}(\bfy),P(z))=1}}
\hspace{-0.5cm}
1.
$$
\end{lemma}

\begin{proof}
The proof hinges on the observation that 
$$
T_1-
\sum_{D_1\ll d\ll D_2}\sum_{\substack{\bfy\in \ZZ^4\cap \cB(Y,d)\\
\gcd(\y,dP(z))=1\\
d\mid Q_i(\y),~i=1,2\\
\gcd(d^{-1}Q_{1}(\bfy),d^{-1}Q_{2}(\bfy),P(z))=1}}
\hspace{-0.5cm}
1\leq
\sum_{\substack{k>1\\ \gcd(k,P(z))=1}} R_k,
$$
where now
$$
R_k=\sum_{\substack{D_1\ll d\ll D_2\\ \gcd(d,k)=1}}\sum_{\substack{\bfy\in \ZZ^4\cap \cB(Y,d)\\
k=\gcd(y_1,\dots,y_4)\\
d\mid Q_i(\y),~i=1,2\\
\gcd(d^{-1}Q_{1}(\bfy),d^{-1}Q_{2}(\bfy),P(z))=1}}
\hspace{-0.5cm}
1.
$$
But clearly 
\begin{align*}
R_k
&\leq \sum_{\substack{D_1\ll d\ll D_2\\ \gcd(d,k)=1}}
\sum_{\substack{\bfy'\in \ZZp^4\cap k^{-1}\cB(Y,d)\\
d\mid Q_i(\y'),~i=1,2\\
\gcd(d^{-1}Q_{1}(\bfy'),d^{-1}Q_{2}(\bfy'),P(z))=1}}
\hspace{-0.5cm}
1
\leq \sum_{\substack{D_1\ll d\ll D_2}}\sum_{\substack{\bfy'\in \ZZp^4\\ |\yy'|\sim Y/k\\
d\mid Q_i(\y'),~i=1,2}}
\hspace{-0.5cm}
1.
\end{align*}
Breaking the $d$-sum into dyadic intervals, it follows from 
Corollary \ref{cor : M} and \eqref{eq:D1D2} that 
$$
R_k\ll_\ve \left(\frac{Y}{k}\right)^{2+\ve} +\frac{(Y/k)^4\log Y}{D_1}\ll 
\frac{Y^{2+\ve}}{k^2} + \frac{B \log B}{k^4},
$$
for any $\ve>0$. 
Hence
$$
\sum_{\substack{k>1\\ \gcd(k,P(z))=1}} R_k\ll_\ve Y^{2+\ve} + B\log B\sum_{k>z} 
\frac{1}{k^4}.
$$
The remaining sum is $O(z^{-3})$, which thereby  completes the proof of the lemma.
\end{proof}

We now turn our attention to an asymptotic evaluation of the main term $\Sigma_1$ in the previous lemma.
We use the geometry of numbers to prove the following result.

\begin{lemma}\label{lem:sigma1-ev}
We have 
$$
\Sigma_1=
\sum_{\substack{\ell_{1}\mid P(z)\\ 
\ell_{2}\mid P(z)\\
\ell_3\ll D_2
}}\frac{\mu (\ell_{1}\ell_2\ell_{3})}{\ell_1^4\ell_{2}\ell_{3}}
\hspace{-0.3cm}
\sum_{\substack{e\mid \frac{P(z)}{\ell_{1}}\\\ell_{2}\mid e\\\gcd(e,\ell_{3})=1}}
\hspace{-0.3cm}
\frac{\mu (e)}{e^3}
S_{\ell_1,\ell_3,e}+O_\eta\left( 
\left(2^z 
+P(z)^4\right)B^{1-\frac{\eta}{2}}\right),
$$
where
$$
S_{\ell_1,\ell_3,e}=
\sum_{\substack{D_1\ll d\ll D_2\\\gcd(d,\ell_{1})=1\\\ell_{3}\mid d}}
\frac{
\# V_{de}^{\times}
\cdot \Vol\mathcal{B}(Y,d)}{d^3}.
$$
\end{lemma}
\begin{proof}
Using M\"obius inversion to handle
$\gcd(d^{-1}Q_{1}(\bfy),d^{-1}Q_{2}(\bfy),P(z))=1
$, we see that 
\begin{align*}
\Sigma_1&=
\sum_{e\mid P(z)} \mu (e)\sum_{D_1\ll d\ll D_2}\sum_{\substack{\bfy\in \ZZ^4\cap \cB(Y,d)\\
\gcd(\y,dP(z))=1\\
de\mid Q_i(\y),~i=1,2}}
1
=\sum_{e\mid P(z)}\mu (e)\sum_{\substack{D_1\ll d\ll D_2}}\sum_{[\uu]\in V_{de}^{\times}}\sum_{\substack{\bfy\in \ZZ^4\cap \cB(Y,d)\\\gcd(\bfy,dP(z))=1\\\bfy\in\Lambda_{[\uu],de}}}1,
\end{align*}
where $V_{de}^{\times}$ and $\Lambda_{[\uu],de}$ are given by 
\eqref{eq:def-Vd} and \eqref{eq:lattice-def}, respectively. 
Appealing to M\"obius inversion once more, we obtain
\[
\Sigma_{1}=\sum_{e\mid P(z)}\mu (e)\sum_{\substack{D_1\ll d\ll D_2}}\sum_{[\uu]\in V_{de}^{\times}}\sum_{\substack{\bfy\in \ZZ^4\cap \cB(Y,d)\\\bfy\in\Lambda_{[\uu],de}}}\sum_{\substack{\ell\mid \bfy\\\ell \mid dP(z)}}\mu(\ell).
\]
It will be convenient to observe that
\[
\begin{split}
\sum_{\substack{\ell\mid \bfy\\\ell \mid dP(z)}}\mu(\ell)
&=\sum_{\substack{c\mid \bfy\\ c\mid de}}\sum_{\substack{\ell_{1}\mid\bfy\\ \ell_{1}\mid P(z)\\ \gcd(\ell_{1},de)=1}}\mu (\ell_{1})\mu (c)
=\sum_{\substack{\ell_{1}\mid \bfy\\ \ell_{1}\mid P(z)\\ \gcd(\ell_{1},de)=1}}\sum_{\substack{\ell_{2}\mid \bfy\\ \ell_{2}\mid e}}\sum_{\substack{\ell_{3}\mid \bfy\\ \ell_{3}\mid d\\ \gcd(\ell_{3},e)=1}}\mu (\ell_{1})\mu (\ell_{2})\mu (\ell_{3}).
\end{split}
\]
Note that $\mu (\ell_{1})\mu (\ell_{2})\mu (\ell_{3})=\mu (\ell_{1}\ell_2\ell_{3})$ in the summand. 
But then it follows that 
\[
\begin{split}
\Sigma_{1}
&=\sum_{\substack{\ell_{1}\mid P(z)\\ 
\ell_{2}\mid P(z)\\
\ell_3\ll D_2
}}\mu (\ell_{1}\ell_2\ell_{3})
\sum_{\substack{e\mid \frac{P(z)}{\ell_{1}}\\\ell_{2}\mid e\\\gcd(e,\ell_{3})=1}}\mu (e)\sum_{\substack{D_1\ll d\ll D_2\\\gcd(d,\ell_{1})=1\\\ell_{3}\mid d}}
\sum_{[\uu]\in V_{de}^{\times}}\sum_{\substack{\bfy\in \ZZ^4\cap \cB(Y,d)
\\
\ell_{1}\ell_{2}\ell_{3}\mid \bfy\\\bfy\in\Lambda_{[\uu],de}}}1\\
&=\sum_{\substack{\ell_{1}\mid P(z)\\ 
\ell_{2}\mid P(z)\\
\ell_3\ll D_2
}}\mu (\ell_{1}\ell_2\ell_{3})
\sum_{\substack{e\mid \frac{P(z)}{\ell_{1}}\\\ell_{2}\mid e\\\gcd(e,\ell_{3})=1}}\mu (e)\sum_{\substack{D_1\ll d\ll D_2\\\gcd(d,\ell_{1})=1\\\ell_{3}\mid d}}
\sum_{[\uu]\in V_{de}^{\times}}
N,
\end{split}
\]
where
$$
N=
\#\left(\ZZ^4\cap 
(\ell_1\ell_2\ell_3)^{-1}\cB(Y,d)\cap 
\Lambda_{[\uu],\frac{de}{\ell_2\ell_3}}\right).
$$
Recalling the definition \eqref{eq:BXY} of $\cB(Y,d)$, 
we now appeal to the lattice point counting result worked out by Schmidt \cite[Lemma $2$]{Sc68}. This yields
$$
N=
\frac{\Vol\left((\ell_1\ell_2\ell_3)^{-1}\mathcal{B}\left(Y,d\right)\right)\cdot(\ell_{2}\ell_{3})^{3}}{(de)^{3}}+O\left(1+\frac{Y^{3}}{(\ell_{1}\ell_{2}\ell_{3})^{3}s_{1,[\uu],\frac{de}{\ell_{2}\ell_{3}}}^{3}}\right),
$$
where $s_{1,[\uu],\frac{de}{\ell_{2}\ell_{3}}}$ is the smallest successive minimum of the lattice
$\Lambda_{[\uu],\frac{de}{\ell_2\ell_3}}$.

Since 
$
\Vol\left((\ell_1\ell_2\ell_3)^{-1}\mathcal{B}\left(Y,d\right)\right)=
(\ell_1\ell_2\ell_3)^{-4}\Vol\mathcal{B}(Y,d),
$
it follows that 
\begin{equation}\label{eq:schmidt}
N=
\frac{\Vol\mathcal{B}(Y,d)}{(de)^{3}\ell_1^4\ell_{2}\ell_{3}}+O\left(1+\frac{Y^{3}}{(\ell_{1}\ell_{2}\ell_{3})^{3}s_{1,[\uu],\frac{de}{\ell_{2}\ell_{3}}}^{3}}\right).
\end{equation}
We begin by handling the overall contribution to $\Sigma_1$ from the error terms.
Let $E_1$ denote the overall contribution from the term $O(1)$ and let $E_2$ denote the contribution from the term involving the first successive minimum. Beginning with the latter, 
it follows from 
Lemma \ref{lem : sucmin} that
\[
\sum_{\substack{D_1\ll d\ll D_2\\\ell_{3}\mid d}}\sum_{[\uu]\in V_{de}^{\times}}\frac{Y^{3}}{(\ell_{1}\ell_{2}\ell_{3})^{3}s_{1,[\uu],\frac{de}{\ell_{2}\ell_{3}}}^{3}}
\ll\frac{Y^{3}}{\ell_{1}^{3}(\ell_{2}\ell_{3})^{\frac{7}{4}}}\cdot(\ell_{2,\Delta}\ell_{3,\Delta})
\cdot(D_1e)^{-\frac{1}{4}}\frac{\log B}{\log\log B},
\]
making the overall contribution
\[
E_2\ll B^{1-\eta}\frac{\log B}{\log\log B}\sum_{\ell_{2}\mid P(z)}\frac{\ell_{2,\Delta}}{\ell_{2}^{\frac{7}{4}}}
\sum_{\ell_{3}\ll D_2}\frac{\ell_{3,\Delta}}{\ell_{3}^{\frac{7}{4}}}
\sum_{\substack{e\mid P(z)\\\ell_{2}\mid e}}e^{-\frac{1}{4}},
\]
since $D_1$ satisfies \eqref{eq:D1D2} and $Y\ll B^{\frac{3}{8}-\frac{\eta}{2}}$.
The inner sum is at most $2^z$ and the sums over $\ell_2,\ell_3$ are
$O(1)$, since
$$
\sum_{\ell\in \NN}\frac{\ell_{\Delta}}{\ell^{\frac{7}{4}}}\leq \sum_{\substack{k\in \NN\\ p\mid k\Rightarrow p\mid \Delta}} \frac{1}{k^{\frac{3}{4}}} \sum_{\ell'\in \NN} \frac{1}{{\ell'}^{\frac{7}{4}}}\ll 1.
$$
Hence 
$E_2=O_\eta(2^z B^{1-\frac{\eta}{2}})$, which is satisfactory.

The remaining error term in the lattice point counting result makes the contribution
\begin{equation}\label{eq:booty}
E_1\ll \sum_{\substack{\ell_{1}\mid P(z)\\ 
\ell_{2}\mid P(z)\\
\ell_3\ll D_2
}}|\mu (\ell_{1}\ell_2\ell_{3})|
\sum_{\substack{e\mid P(z)\\\ell_{2}\mid e}}\sum_{\substack{D_1\ll d\ll D_2
\\\ell_{3}\mid d}}
\#V_{de}^{\times}
\end{equation}
to $\Sigma_1$.
We claim that  
\begin{equation}\label{eq:boot}
\sum_{\substack{e\mid P(z)\\\ell_{2}\mid e}}\sum_{\substack{d\ll D_2
\\\ell_{3}\mid d}}
\#V_{de}^{\times}
\ll P(z)^2D_2^2
\frac{(\ell_{2,\Delta}\ell_{3,\Delta})^\ve}{\ell_2\ell_3} \prod_{\substack{p\mid \ell_2\ell_3\\ p\nmid \Delta}} \left(1+O(p^{-\frac{1}{2}})\right).
\end{equation}
To prove this, we  appeal to the second part of Lemma \ref{lem : loc''}, which implies that 
\begin{align*}
\sum_{\substack{d\leq  D
\\\ell_{3}\mid d}}
\#V_{de}^{\times}
&\ll_\ve
\sum_{\substack{d\leq D
\\\ell_{3}\mid d}}
de d_\Delta^\ve e_\Delta^\ve \prod_{\substack{p\mid de\\ p\nmid \Delta}} \left(1+O(p^{-\frac{1}{2}})\right)\\
&\ll 
e e_{\Delta}^\ve 
\sum_{\substack{d'\leq D/\ell_3}}
\ell_3d' (\ell_{3,\Delta}d_\Delta')^\ve \prod_{\substack{p\mid \ell_3d'e\\ p\nmid \Delta}} \left(1+O(p^{-\frac{1}{2}})\right),
\end{align*}
for any $D\geq 1$, 
on writing   $d=\ell_{3}d'$.
Now
\begin{align*}
\sum_{\substack{d'\leq  D/\ell_3}}d'd_\Delta'^\ve 
\prod_{\substack{p\mid d'\\ p\nmid \Delta}} \left(1+O(p^{-\frac{1}{2}})\right)
&\ll \sum_{\substack{k\in \NN\\ 
p\mid k\Rightarrow p\mid \Delta
}} k^{1+\ve} \sum_{d''\leq  D/(k\ell_3)} d''
\prod_{\substack{p\mid d''}} \left(1+O(p^{-\frac{1}{2}})\right)\\
&\ll \frac{D^2}{\ell_3^2} \sum_{\substack{k\in \NN\\ 
p\mid k\Rightarrow p\mid \Delta
}} k^{-1+\ve}.
\end{align*}
This is $O_\ve ({D^2}/{\ell_3^2})$.
Hence we have proved that 
\begin{equation}\label{eq:toe}
\sum_{\substack{d\leq  D
\\\ell_{3}\mid d}}
\#V_{de}^{\times}
\ll 
\frac{e e_{\Delta}^\ve \ell_{3,\Delta}^\ve D^2}{\ell_3} \prod_{\substack{p\mid \ell_3 e\\ p\nmid \Delta}} \left(1+O(p^{-\frac{1}{2}})\right),
\end{equation}
for any $D\geq 1$. 
Making the change
of variable $e=\ell_{2}e'$, 
and noting that
\[
\sum_{\substack{e'\mid \frac{P(z)}{\ell_{2}}}}e'e_\Delta'^{\varepsilon}
\prod_{\substack{p\mid e'\\ p\nmid \Delta}} \left(1+O(p^{-\frac{1}{2}})\right)\ll 
 \frac{P(z)^{2}}{\ell_{2}^{2}},
\]
the claimed bound \eqref{eq:boot} readily follow.

It
follows from  \eqref{eq:D1D2} and
$Y\ll B^{\frac{3}{8}-\frac{\eta}{2}}$ that 
$D_2^2\ll B^{1-4\eta}$.
On inserting \eqref{eq:boot}  into \eqref{eq:booty} and summing trivially over $\ell_1$, a similar analysis leads to the conclusion that 
$$
E_1\ll_\ve P(z)^3D_2^2
\sum_{\ell_{2}\mid P(z)}\sum_{\ell_3\ll D_2}
\frac{(\ell_{2,\Delta}\ell_{3,\Delta})^\ve}{\ell_2\ell_3} \prod_{\substack{p\mid \ell_2\ell_3\\ p\nmid \Delta}} \left(1+O(p^{-\frac{1}{2}})\right)\\
\ll_\eta  P(z)^4B^{1-\frac{\eta}{2}},
$$
which is satisfactory for the lemma.

Finally, we note that the main term in our asymptotic formula \eqref{eq:schmidt} for $N$ gives 
$$
\sum_{\substack{\ell_{1}\mid P(z)\\ 
\ell_{2}\mid P(z)\\
\ell_3\ll D_2
}}\mu (\ell_{1}\ell_2\ell_{3})
\sum_{\substack{e\mid \frac{P(z)}{\ell_{1}}\\\ell_{2}\mid e\\\gcd(e,\ell_{3})=1}}\mu (e)\sum_{\substack{D_1\ll d\ll  D_2\\\gcd(d,\ell_{1})=1\\\ell_{3}\mid d}}
\sum_{[\uu]\in V_{de}^{\times}}
\frac{\Vol\mathcal{B}(Y,d)}{(de)^{3}\ell_1^4\ell_{2}\ell_{3}},
$$
once inserted into our expression for $\Sigma_1$.
Finally, on rearranging the terms, we are led to the main term in the lemma. 
\end{proof}

We now have everything in place to complete the first step in the  
proof of Proposition \ref{pro:latt}. We shall take
$$
z=\frac{\log B}{(\log\log B)^{2}} \quad \text{ and }\quad M=B^{\frac{\eta}{10}}.
$$
Mertens' theorem  implies that $P(z)\ll \e^{\frac{\log B}{(\log\log B)^2}}\ll_\eta B^{\frac{\eta}{16}}$.
Hence we obtain
$$
\Sigma_1=
\hspace{-0.1cm}
\sum_{\substack{\ell_{1}\mid P(z)\\ 
\ell_{2}\mid P(z)\\
\ell_3\ll D
}}\frac{\mu (\ell_{1}\ell_2\ell_{3})}{\ell_1^4\ell_{2}\ell_{3}}
\hspace{-0.3cm}
\sum_{\substack{e\mid \frac{P(z)}{\ell_{1}}\\\ell_{2}\mid e\\\gcd(e,\ell_{3})=1}}
\hspace{-0.3cm}
\frac{\mu (e)}{e^3}
S_{\ell_1,\ell_3,e}+
O_\eta\left(B^{1-\frac{\eta}{4}}\right),
$$
in Lemma \ref{lem:sigma1-ev}.
Next, it follows from Lemma \ref{lem:T2T3} that 
\begin{align*}
T_{2}+T_3
&\ll_\ve 
\frac{B}{z^{1-\varepsilon}}+
B^{1-\frac{\eta}{2}}M^3+
Y^{2+\ve}+ \frac{B(\log B)^{2}}{M}
\ll
 \frac{B}{\sqrt{\log B}}.
\end{align*}
since $Y\ll B^{\frac{3}{8}-\frac{\eta}{2}}$.
Inserting these estimates into \eqref{eq:T1-3}, we obtain
$$
\#\tilde\cM(Y)=2
\sum_{\substack{\ell_{1}\mid P(z)\\ 
\ell_{2}\mid P(z)\\
\ell_3\ll D_2
}}\frac{\mu (\ell_{1}\ell_2\ell_{3})}{\ell_1^4\ell_{2}\ell_{3}}
\sum_{\substack{e\mid \frac{P(z)}{\ell_{1}}\\\ell_{2}\mid e\\\gcd(e,\ell_{3})=1}}
\frac{\mu (e)}{e^3}
S_{\ell_1,\ell_3,e}
+\left(\frac{B}{\sqrt{\log B}}\right).
$$

We next  show that the sum over $\ell_3$ can be truncated with acceptable error,
as in the following result.

\begin{lemma}\label{lem:latt}
We have 
$$
\#\tilde\cM(Y)=2T_0+\left(\frac{B}{\sqrt{\log B}}\right),
$$
where 
$$
T_0=\sum_{\substack{\ell_{1}\mid P(z)\\ 
\ell_{2}\mid P(z)\\
\ell_3\mid P(z)
}}\frac{\mu (\ell_{1}\ell_2\ell_{3})}{\ell_1^4\ell_{2}\ell_{3}}
\hspace{-0.3cm}
\sum_{\substack{e\mid \frac{P(z)}{\ell_{1}}\\\ell_{2}\mid e\\\gcd(e,\ell_{3})=1}}
\hspace{-0.3cm}
\frac{\mu (e)}{e^3}
\sum_{\substack{D_1\ll d\ll D_2\\\gcd(d,\ell_{1})=1\\\ell_{3}\mid d}}
\frac{
\# V_{de}^{\times}
\cdot \Vol\mathcal{B}(Y,d)}{d^3}.
$$
\end{lemma}

\begin{proof}
Recall the definition of $S_{\ell_1,\ell_3,e}$ from the statement of Lemma \ref{lem:sigma1-ev}.
Then, in order to prove the lemma, it will suffice to bound
\begin{align*}
E&=
\sum_{\substack{\ell_{1}\mid P(z)\\ 
\ell_{2}\mid P(z)
}}
\sum_{\substack{\ell_3\ll D_2\\ \ell_3\nmid P(z)}}
\frac{|\mu (\ell_{1}\ell_2\ell_{3})|}{\ell_1^4\ell_{2}\ell_{3}}
\sum_{\substack{e\mid P(z)\\\ell_{2}\mid e}}
\frac{|\mu (e)|}{e^3}
S_{\ell_1,\ell_3,e}\\
&\ll 
Y^4
\sum_{\substack{\ell_{1}\mid P(z)\\ 
\ell_{2}\mid P(z)
}}
\sum_{\substack{\ell_3\ll D_2\\ \ell_3\nmid P(z)}}
\frac{1}{\ell_1^4\ell_{2}\ell_{3}}
\sum_{\substack{e\mid P(z)\\\ell_{2}\mid e}}
\frac{1}{e^3}
\sum_{\substack{D_1\ll d\ll D_2\\\ell_{3}\mid d}}
\frac{\# V_{de}^{\times}}{d^3}.
\end{align*}
A modest reworking of the proof of 
\eqref{eq:boot}, using partial summation and \eqref{eq:toe}  to incorporate the weight $(de)^{-3}$, easily yields
$$
E\ll_\ve
\frac{Y^4}{D_1 }
\sum_{\substack{\ell_{1}\mid P(z)\\ 
\ell_{2}\mid P(z)
}}
\sum_{\substack{\ell_3\ll D_2\\ \ell_3\nmid P(z)}}
\frac{1}{\ell_1^4\ell_{2}^{3-\ve}\ell_{3}^{2-\ve}}\ll_\ve
\frac{Y^4}{D_1 }
\sum_{\substack{\ell_3>z}}
\frac{1}{\ell_{3}^{2-\ve}}\ll \frac{B}{z^{1-\ve}},
$$
by \eqref{eq:D1D2}.
Our choice of $z$ ensures this is satisfactory.
\end{proof}

\subsection{Asymptotic formula for  $T_{0}$}

The final step in the proof of Proposition \ref{pro:latt}
is to  analyse the main term in 
Lemma \ref{lem:latt}, in which we recall that 
$\mathcal{B}(Y,d)$ is given by \eqref{eq:BXY}.
Making the change of variables $e=\ell_2e'$ and $d=\ell_3d'$, we find that 
\begin{equation}\label{eq:T0''}
T_0=\sum_{\substack{\ell_{1}\mid P(z)\\ 
\ell_{2}\mid P(z)\\
\ell_3\mid P(z)
}}\frac{\mu (\ell_{1}\ell_2\ell_{3})\mu(\ell_2)}{\ell_1^4\ell_{2}^4\ell_{3}^4}
\hspace{-0.3cm}
\sum_{\substack{e'\mid \frac{P(z)}{\ell_{1}\ell_2}\\\gcd(e',\ell_{3})=1}}
\hspace{-0.3cm}
\frac{\mu (e')}{e'^3}Q_{\ell_1,\ell_2,\ell_3,e'},
\end{equation}
where
\begin{equation}\begin{split}\label{eq:check}
Q_{\ell_1,\ell_2,\ell_3,e'}
&=
\sum_{\substack{D_1/\ell_3\ll d'\ll D_2/\ell_3\\\gcd(d',\ell_{1})=1}}
\frac{
\# V_{\ell_2\ell_3 d'e'}^{\times}
\cdot \Vol\mathcal{B}(Y,\ell_3 d')}{d'^3}\\
&=\int_{\substack{\y\in \RR^4\\ |\y|\sim Y}} \sum_{\substack{
D_{1,\y}\leq d' \leq D_{2,\y}\\
\gcd(d',\ell_{1})=1}}
\frac{
\# V_{\ell_2\ell_3 d'e'}^{\times}
}{d'^3} \mathrm{d}\y
\end{split}
\end{equation}
and 
\begin{equation}\label{eq:Dy}
D_{1,\y}= \frac{|\y|^2\max(|Q_1(\y)|,|Q_2(\y)|)}{\ell_3 B}, \quad D_{2,\y}= \frac{\max(|Q_1(\y)|,|Q_2(\y)|)}{\ell_3 B^{\frac{1}{4}+\eta}}.
\end{equation}
Here, we have observed that  the condition  $D_1/\ell_1\ll d'\ll D_2/\ell_3$ is implied by the condition 
$D_{1,\y}\leq d' \leq D_{2,\y}$, since $Z(\RR)= \emptyset$.
We are therefore led to prove the following result.

\begin{lemma}\label{lem:quoc}
Let $Z_2>Z_1>0$ and let  $c,\ell\in \NN$ be square-free coprime integers. Then 
\[
\begin{split}
\sum_{\substack{Z_1\leq d\leq Z_2\\(d,\ell)=1}}\frac{\#V_{cd}^{\times}}{ d^{3}}&=C\#V_{c}^{\times}h_1(c)
h_2(\ell)\left(\frac{1}{Z_1}-\frac{1}{Z_2}\right)+O_{\varepsilon}(c^{1+\varepsilon}Z_1^{-\frac{3}{2}+\ve}),
\end{split}
\]
for any $\varepsilon >0$, where
\[
\begin{split}
C&=\prod_{p}\left(1-\frac{1}{p}\right)\left(1+\sum_{a=1}^{\infty}\frac{\#V_{p^{a}}^{\times}}{p^{2a}}\right)\\
h_1(c)
&=\prod_{p|c}\left(1+\sum_{a=1}^{\infty}\frac{\#V_{p^{a}}^{\times}}{p^{2a}}\right)^{-1}\left(1+\frac{1}{\#V_{p}^{\times}}\sum_{a=1}^{\infty}\frac{\#V_{p^{a+1}}^{\times}}{p^{2a}}\right)\\
h_2(\ell)
&=\prod_{p|\ell}\left(1+\sum_{a=1}^{\infty}\frac{\#V_{p^{a}}^{\times}}{p^{2a}}\right)^{-1}.
\end{split}
\]
Moreover, $h_1(c)=O_\ve(c^\ve)$ and $h_2(\ell)=O_\ve(\ell^\ve)$, for any $\ve>0$.
\end{lemma}
\begin{proof}
If $0<Z_1<1$, the result is trivial, so we may assume $Z_1\geq 1$. We begin by defining the multiplicative function
\[
f(p^{a})=
\begin{cases}
p^{-a}\#V_{p^{a}}^{\times}								&\text{ if $p\nmid c\ell$,}\\
p^{-a}\#V_{p^{a+1}}^{\times}/\#V_{p}^{\times}		&\text{ if $p\mid c$,}\\
0													&\text{ if $p\mid \ell$.}
\end{cases}
\]
It follows form the Hasse bound that 
$
\#V_{p}^{\times}=p-T_{p}+1,
$
where
 $T_{p}=O(\sqrt{p})$.
Moreover, the Chinese remainder theorem implies that 
$$
f(d)=\frac{
\#V_{cd}^{\times}}{d\#V_{c}^{\times}}
$$ 
for any 
 $d\in\mathbb{N}$ such that $\gcd(d,\ell)=1$.
 Hence we can write
 $$
 \sum_{\substack{Z_1\leq d\leq Z_2\\(d,\ell)=1}}\frac{\#V_{cd}^{\times}}{ d^{3}}=
\#V_{c}^{\times} \sum_{\substack{Z_1\leq d\leq Z_2}}\frac{f(d)}{ d^{2}}.
$$ 
We claim that 
\begin{equation}\label{eq:sand}
\sum_{d\leq x} f(d) =C h_1(c)h_2(\ell) x +O\left(x^{\frac{1}{2}+\ve}\right),
\end{equation}
for any $\ve>0$, 
where $C, h_1,h_2$ are defined in the statement of the lemma.
Once achieved, on recalling the bound $\#V_c^\times=O_\ve(c^{1+\ve})$ from Lemma \ref{lem : loc''}, 
an application of partial summation easily leads to the statement of the lemma. 
Finally, the bounds on $h_1$ and $h_2$ in the last part of the lemma follow easily from 
Lemma \ref{lem : loc''}.

To prove \eqref{eq:sand}, we write  $f=1*g$  as a Dirichlet convolution, noting that 
$
g(p^a)=
f(p^a)-f(p^{a-1})$
for any prime $p$ and $a\in 
\NN$. 
Suppose first that $p\nmid \Delta$. Then it follows from Corollary~\ref{cor:stone} that 
$$
g(p^a)
=
\begin{cases}
O(p^{-\frac{1}{2}})  &\text{ if $a=1$ and $p\nmid \ell$,}\\
-1  &\text{ if $a=1$ and $p\mid \ell$,}\\
0 &\text{ otherwise}.
\end{cases}
$$
On the other hand, if $p\mid \Delta$ then \eqref{eq:stone}
and Lemma~\ref{lem : loc2'} together yield
$
g(p^a)
=O(a).
$
Given any $k\in \NN$, 
it therefore follows that 
 $
 g(k)\ll_\ve k^\ve/\sqrt{k_1},
 $
 where $k_1$ is the  part of $k$ that is coprime to $\ell \Delta$.
 Given this, we easily conclude that 
\begin{align*}
\sum_{d\leq x} f(d) 
&=\sum_{k\leq x} g(k) \left(\frac{x}{k} +O(1)\right)
=\gamma x+O(x^{\frac{1}{2}+\ve}), 
\end{align*}
for any $\ve>0$, where
$$
\gamma=\sum_{k=1}^\infty \frac{g(k)}{k}=
\prod_p \left(1+\sum_{a=1}^\infty \frac{g(p^a)}{p^a}\right)=
\prod_p \left(1-\frac{1}{p}\right)\left(1+\sum_{a=1}^\infty \frac{f(p^a)}{p^a}\right).
$$
Inserting the definition of $f(p^a)$, it is straightforward to see that 
$\gamma=Ch_1(c)h_2(\ell)$, as required to complete the proof of the lemma.
\end{proof}

We now seek to 
apply this  in \eqref{eq:check}.
Note from  \eqref{eq:Dy}
 that   $D_{1,\y}\gg Y^4/(\ell_3B)$, since we are assuming that $Z(\RR)=\emptyset$. Hence
we obtain
\begin{align*}
Q_{\ell_1,\ell_2,\ell_3,e'}
=~&
C\#V_{\ell_2\ell_3e'}^{\times}h_1(\ell_2\ell_3e')
h_2(\ell_1) \nu(Y)  +O_\ve\left( \frac{(\ell_2e')^{1+\ve}\ell_3^{\frac{5}{2}}B^{\frac{3}{2}+\ve}}{Y^{2}}\right),
\end{align*}
where
$$
\nu(Y)=\int_{\substack{\y\in \RR^4\\ |\y|\sim Y}} 
\left(\frac{1}{D_{1,\y}}-\frac{1}{D_{2,\y}}\right) \mathrm{d}\y 
$$
and  $D_{1,\y}, D_{2,\y}$ are given by \eqref{eq:Dy}.
Clearly
$
\nu(Y)
=\ell_3 B \sigma_\infty(Y)
 +O(\ell_3 B^{\frac{1}{4}+\eta} Y^2),
$
in the notation of \eqref{eq:path}.
Thus $Q_{\ell_1,\ell_2,\ell_3,e'}$ is equal to
\begin{align*}
C\#V_{\ell_2\ell_3e'}^{\times}h_1(\ell_2\ell_3e')
h_2(\ell_1) \ell_3  B\left(\sigma_\infty(Y)+O\left(\frac{Y^2}{B^{\frac{3}{4}-\eta}}\right)\right) +O_\ve\left( \frac{(\ell_2e')^{1+\ve}\ell_3^{\frac{5}{2}}B^{\frac{3}{2}+\ve}}{Y^{2}}\right).
\end{align*}

It is now time to insert this estimate into 
\eqref{eq:T0''}. First, the overall contribution from the error term is 
\begin{align*}
\ll_\ve 
\frac{B^{\frac{3}{2}+\ve}}{Y^{2}}
\sum_{\substack{\ell_{1}\mid P(z)\\ 
\ell_{2}\mid P(z)\\
\ell_3\mid P(z)
}}\frac{1}{\ell_1^4\ell_{2}^{3-\ve}\ell_{3}^\frac{3}{2}}
\sum_{\substack{e'\mid \frac{P(z)}{\ell_{1}\ell_2}}}
\frac{1}{e'^{2-\ve}}
\ll_\ve \frac{B^{\frac{3}{2}+\ve}}{Y^{2}}.
\end{align*}
Adjoining the  contribution from the main term, we therefore obtain
\begin{equation}\label{eq:lip}
T_0=
C  \cdot J(z) \cdot B\left(\sigma_\infty(Y)+O\left(\frac{Y^2}{B^{\frac{3}{4}-\eta}}\right)\right)
+ O_\ve\left(
 \frac{B^{\frac{3}{2}+\ve}}{Y^{2}}
\right),
\end{equation}
for any $\ve>0$, 
where
\begin{align*}
J(z)
&=
\sum_{\substack{\ell_{1}\mid P(z)\\ 
\ell_{2}\mid P(z)\\
\ell_3\mid P(z)
}}
\frac{\mu (\ell_{1}\ell_2\ell_{3})\mu(\ell_2) 
h_1(\ell_2\ell_3e')
h_2(\ell_1)
}{\ell_1^4\ell_{2}^4\ell_{3}^3}
\hspace{-0.3cm}
\sum_{\substack{e'\mid \frac{P(z)}{\ell_{1}\ell_2}\\\gcd(e',\ell_{3})=1}}
\hspace{-0.3cm}
\frac{\mu (e')}{e'^3}
\cdot \#V_{\ell_2\ell_3e'}^{\times}\\
&=
\sum_{\substack{m\mid P(z)}}
\sum_{\substack{\ell_1,\dots,\ell_4\in \NN\\
\ell_1\cdots \ell_4 =m}}
\hspace{-0.2cm}
\frac{\mu(\ell_1)h_2(\ell_1)}{\ell_1^4}
\cdot
\frac{\mu(\ell_2)^2h_1(\ell_2)\#V_{\ell_2}^{\times}}{\ell_2^4}
\frac{\mu(\ell_3)h_1(\ell_3) \#V_{\ell_3}^{\times}}{\ell_3^3}
\cdot 
\frac{\mu (\ell_4)h_1(\ell_4) \#V_{\ell_4}^{\times}}{\ell_4^3}.	
\end{align*}
In view of the second part of  Lemma \ref{lem : loc''} and the bounds on $h_1$ and $h_2$ from Lemma 
\ref{lem:quoc}, 
we can extend the sum over $m$ to all square-free integers, finding that 
$$
J(z)=J+O(z^{-1}),
$$
where
\begin{align*}
J&=
\prod_{p}
\left(1-\frac{h_2(p)}{p^{4}}+\frac{h_1(p)\#V_{p}^{\times}}{p^{4}}-\frac{2h_1(p)\#V_{p}^{\times}}{p^{3}}\right)
\\
&=\prod_{p}
\left(1+\sum_{a=1}^{\infty}\frac{\#V_{p^{a}}^{\times}}{p^{2a}}\right)^{-1}
\left(1+\sum_{a=1}^{\infty}\frac{\#V_{p^{a}}^{\times}}{p^{2a}}-\frac{1}{p^{4}}+\left(\frac{1}{p^{2}}-\frac{2}{p}\right)\sum_{a=1}^{\infty}\frac{\#V_{p^{a}}^{\times}}{p^{2a}}\right).
\end{align*}
Recalling the definition of $C$ from the statement of Lemma \ref{lem:quoc}, it follows that 
$CJ=\mathfrak{S}_1$, in the notation of \eqref{eq:series}.
Returning to \eqref{eq:lip} and observing that $z\geq \sqrt{\log B}$, it therefore follows that 
$$
T_0=
\mathfrak{S}_1 B
\left(\sigma_\infty(Y)+O\left(\frac{Y^2}{B^{\frac{3}{4}-\eta}}\right)\right)
 +O\left(\frac{B}{\sqrt{\log B}} \right)
+ O_\ve\left(
 \frac{B^{\frac{3}{2}+\ve}}{Y^{2}}
\right),
$$
where $\sigma_\infty(Y)$ is given by \eqref{eq:path}. 
Substituting this into 
 Lemma \ref{lem:latt}, 
 and using the lower bound $Y\gg B^{\frac{1}{4}+\frac{\eta}{2}}$,  
 we are finally led to the statement of  Proposition \ref{pro:latt}.

\section{Asymptotics via the circle method}\label{sec:circle-a}

The goal of this section is to prove Proposition \ref{pro:L2}.
Recall the notation 
$
\varDelta(\xx)=\prod_{i=1}^4L_i(\xx)
$
from \eqref{eq:normLideltaLi}.
For any  $\xx$ and any compactly supported weight function $w:\RR^4\to \RR_{\geq 0}$, 
the  singular integral is defined to be 
$$
\sigma_{\infty,w}(\xx)=\int_{-\infty}^{+\infty}\int_{\BR^4}w(\yy)\operatorname{e}\left(-\theta(L_1(\xx)y_1^2+\cdots + L_4(\xx)y_4^2)\right)\operatorname{d}\yy\operatorname{d}\theta.
$$
In the special case that  $w_0$ is the characteristic function of $[-1,1]^4$, we set 
$
\sigma_{\infty}(\xx)=\sigma_{\infty,w_0}(\xx).
$
We have $\sigma_\infty(\lambda \x)=\lambda^{-1}\sigma_\infty(\x)$, for any $\lambda>0$.
Moreover, it follows from \cite[Lemma~4.12]{duke} that 
\begin{equation}\label{eq:brown-bag}
\sigma_{\infty}(\xx) \ll \frac{1}{|\varDelta(\x)|^{1/4}} \quad \text{ and }\quad
\sigma_{\infty,w}(\xx) \ll \frac{1}{|\varDelta(\x)|^{1/4}},
\end{equation}
for any 
 compactly supported smooth weight function $w:\RR^4\to \RR_{\geq 0}$.

Finally, we put 
$$
\mathfrak{S}(\xx)=\mathfrak{S}(Q_{\xx})=\prod_p \sigma_p(\xx),
$$ where 
$$
\sigma_p(\xx)=\lim_{k\to \infty}\frac{\#\left\{\y\in (\ZZ/p^k\ZZ)^4: L_1(\xx)y_1^2+\cdots + L_4(\xx)y_4^2\equiv 0 \bmod{p^k}
\right\}}{p^{3k}}.
$$
With this notation to hand we may now record the first main result of this section, which closely follows the strategy in  \cite[\S~5]{duke}.

\begin{proposition}\label{thm:countingxsmall}
Let $\eta>0$. Then 
	\begin{align*}
	\#\CL_2(B)
=~&\frac{B}{\zeta(2)}\sum_{\substack{\xx\in\BZprim^2\\ 
	B^{2\eta}\leqslant |\xx|\leqslant B^\frac{1}{4}\\ \prod_{i=1}^{4}L_i(\xx)\neq\square }}\frac{\sigma_{\infty}(\xx)\mathfrak{S}(\xx)}{|\xx|}+O(\eta^\frac{1}{2} B\log B)+O_\eta(B^{1-\frac{\eta^2}{32}}).
	\end{align*}
\end{proposition}

The asymptotic behaviour of the  leading term in Proposition \ref{thm:countingxsmall} is  our next milestone and 
is summarised in the following result. 

\begin{proposition}\label{prop:lid}
Let $\eta>0$. Then 
\begin{align*}
	\sum_{\substack{\xx\in\BZprim^2\\ B^{2\eta}\leqslant |\xx|\leqslant B^\frac{1}{4}\\ \prod_{i=1}^{4}L_i(\xx)\neq\square }}\frac{\sigma_{\infty}(\xx)\mathfrak{S}(\xx)}{|\xx|}
= \frac{\tau_\infty \mathfrak{S}_2}{4\zeta(2)}\log B+O(\eta\log B)+O_\eta(1),
\end{align*}
where $\tau_\infty$ is given by \eqref{eq:tau-inf} and $\mathfrak{S}_2$ is given by \eqref{eq:def-SS2}.
\end{proposition}

Combining Propositions \ref{thm:countingxsmall} and \ref{prop:lid} concludes the proof of Proposition \ref{pro:L2}.

\subsection{Proof of Proposition \ref{thm:countingxsmall}}

 Let 
$w:\RR^4\to \RR_{\geq 0}$ be a compactly supported  weight function.
Then for any $X\geq 1$, we define the  weighted counting function
\begin{equation}\label{eq:NwBX}
	L_w(B,X)=\sum_{\substack{\xx\in\BZprim^2\\ |\xx|\sim X\\
	\varDelta(\xx)\neq\square}}\sum_{\substack{\yy\in\BZprim^4\\\eqref{eq:quadbundle} \text{ holds}}} w\left(\frac{|\xx|^\frac{1}{2}}{B^\frac{1}{2}}\yy\right).
\end{equation}
When  $w=w_0$, as above,  then
$$
\#\CL_2(B)=\sum_{X\nearrow B^{\frac{1}{4}}}L_{w_0}(B,X).
$$
Our strategy for proving Proposition \ref{thm:countingxsmall} is to first produce an asymptotic formula for 
$L_w(B,X)$ when $w$ is a suitable smooth weight, before finally showing how to approximate the counting function
$L_{w_0}(B,X)$ by smoothly weighted ones. 
It turns out that the circle method tools required to produce an asymptotic formula for 
$L_w(B,X)$ are already in exactly the right form in \cite{duke}. 
This allows us to prove the following analogue of 
 \cite[Lemma 5.3]{duke}.
\begin{lemma}\label{prop:countsmoothweight}
Let 
$w:\RR^4\to \RR_{\geq 0}$ be a compactly supported  weight function 
which vanishes on $[-\eta,\eta]^4$. 
	Suppose that $X\geq 1$ satisfies \begin{equation}\label{eq:Xrange}
		B^{2\eta}\ll X\ll B^{\frac{1}{4}-4\eta}.
	\end{equation}
Then, for $\eta>0$ sufficiently small, we have
$$
	L_w(B,X)=\frac{B}{\zeta(2)}\sum_{\substack{\xx\in\BZprim^2\\ |\xx|\sim X\\\varDelta(\xx)\neq\square}}\frac{\sigma_{\infty,w}(\xx)\mathfrak{S}(\xx)}{|\xx|}+O_{\eta,w}(B^{1-\frac{1}{16}\eta^2}).
$$
\end{lemma}
\begin{proof}
	For $\xx\in\BZprim^2$ and 
	$Y\gg 1$, we define
	$$
	N_w(Q_\xx;Y)=\sum_{\substack{\yy\in\BZprim^4\\  Q_\xx(\y)=0}} w\left(\frac{\yy}{Y}\right),
	$$
	where $Q_\x$ is given by \eqref{eq:Qx}. Then 
	$$
	L_w(B,X)=\sum_{\substack{\xx\in\BZprim^2\\ 
	|\xx|\sim X\\\varDelta(\xx)\neq\square}}N_w\left(Q_\xx;\sqrt{\frac{B}{|\xx|}}\right).
	$$
	As in \eqref{eq:normLideltaLi}, we write  $\|Q_\xx\|=\max_{1\leqslant i\leqslant 4}|L_i(\xx)|$ and we recall the definition \eqref{eq:deltabadx} of $\deltabad(\x)$.

It now follows from  \cite[Lemma 5.2]{duke} 
$$
N_w(Q_\xx;Y)=
\frac{\sigma_{\infty,w}(\xx)\mathfrak{S}(\xx)}{\zeta(2)}Y^2+O_{\eta,w}\left(\frac{Y^{\frac{5}{3}+5\eta}}{\|Q_\xx\|^\frac{1}{2}}\right),
$$
for any $Y\geq 1$, provided that 
$Y^\eta\leqslant \|Q_\x\|\leqslant Y^\frac{2}{3}$, 
$|L_i(\xx)|\geqslant \|Q_\x\|^{1-\eta}$ for every  $1\leqslant i\leqslant 4$, and 
	$\deltabad(\xx)\leqslant \|Q_\x\|^\eta$.
	Observing that $|\xx|\ll \|Q_\xx\|\ll |\xx|$, we deduce that 
$$
	N_w(Q_\xx;Y)=\frac{\sigma_{\infty,w}(\xx)\mathfrak{S}(\xx)}{\zeta(2)}Y^2+O_{\eta,w}\left(\frac{Y^{\frac{5}{3}+5\eta}}{|\xx|^\frac{1}{2}}\right),
	$$
provided that 
	\begin{equation}\label{eq:cond1}
	Y^\eta\ll |\xx|\ll Y^\frac{2}{3},
\end{equation}
\begin{equation}\label{eq:cond2}
	|L_i(\xx)|\gg  |\xx|^{1-\eta}, \text{ for every } 1\leqslant i\leqslant 4,
\end{equation}
and 
\begin{equation}\label{eq:cond3}
	\deltabad(\xx)\ll |\xx|^\eta.
\end{equation}
Under the assumption that $X$ satisfies  \eqref{eq:Xrange}, the condition \eqref{eq:cond1} is always satisfied with $Y=\sqrt{B/|\xx|}$. Hence it follows that 
\begin{equation}\label{eq:eqinter1}
	\begin{split}
	\sum_{\substack{\xx\in\BZprim^2,~|\xx|\sim X\\\varDelta(\xx)\neq\square\\ \eqref{eq:cond2}\eqref{eq:cond3}\text{ hold}}}\sum_{\substack{\yy\in\BZprim^4\\\eqref{eq:quadbundle} \text{ holds}}} w\left(\frac{|\xx|^\frac{1}{2}}{B^\frac{1}{2}}\yy\right)
	=~&\frac{B}{\zeta(2)}\sum_{\substack{\xx\in\BZprim^2,|\xx|\sim X\\\varDelta(\xx)\neq\square\\ \eqref{eq:cond2}\eqref{eq:cond3}\text{ hold}}}\frac{\sigma_{\infty,w}(\xx)\mathfrak{S}(\xx)}{|\xx|}
	+O_{\eta,w}(B^{\frac{5}{6}+\frac{5}{2}\eta}X^\frac{2}{3}).
		\end{split}
\end{equation} 
Note that the  error term is $O( B^{1-\frac{1}{6}\eta})$, since  $X\ll B^{\frac{1}{4}-4\eta}$ in \eqref{eq:Xrange}. 

It remains to treat the  cases where either \eqref{eq:cond2} or \eqref{eq:cond3} fails.
We start with   \eqref{eq:cond2} and assume without loss of generality that  $|L_{1}(\xx)|\ll |\xx|^{1-\eta}$. 
Then it follows from Lemma  \ref{lem:auxbounddelta}, with $\delta=1-\eta$ and $Y=\sqrt{B/X}$, that the overall contribution is 
\begin{align*}
&\ll_{\eta,\varepsilon}  B^\varepsilon X^{-\frac{1}{2}\eta}\left(B+B^\frac{1}{2}X^2\right)
\ll X^{-\frac{1}{2}\eta} B^{1+\varepsilon}
\ll B^{1-\eta^2+\varepsilon},
\end{align*} 
under the assumption \eqref{eq:Xrange}.
On taking $\varepsilon=\eta^2/2$, the contribution from this case is therefore $O_{\eta}(B^{1-\eta^2/2})$.
To handle the situation  when  \eqref{eq:cond3} fails, we use Proposition~\ref{prop:M2Xbig}. Taking $D\gg X^\eta$, we therefore  obtain the  contribution 
\begin{align*}
&\ll_{\varepsilon} B^\varepsilon\left(\frac{B+B^\frac{1}{2}X^2}{X^{\frac{\eta}{16}}}+XB^\frac{1}{2}+ B^\frac{2}{3}\right)
\ll X^{-\frac{\eta}{16}}B^{1+\varepsilon}+B^{\frac{3}{4}+\varepsilon}
\ll B^{1-\frac{\eta^2}{8}+\varepsilon}.
\end{align*}
Taking  $\varepsilon=\eta^2/{16}$, we obtain the overall  contribution  $O_\eta(B^{1-\eta^2/16})$.

It remains to extend the sum of $\xx\in\BZprim^2$ in \eqref{eq:eqinter1} to the whole range. 
For this purpose, we consider the sums
\begin{equation}\label{eq:SX1}
	S_1(X)=\sum_{\substack{\xx\in\BZprim^2, ~\varDelta(\xx)\neq\square\\|\xx|\sim X,~|L_1(\xx)|\ll X^{1-\eta} }}\frac{\mathfrak{S}(\xx)}{|\xx||\prod_{i=1}^4L_i(\xx)|^\frac{1}{4}}
\end{equation}
and 
\begin{equation}\label{eq:SX2}
	S_2(X)=\sum_{\substack{\xx\in\BZprim^2,~\varDelta(\xx)\neq\square\\|\xx|\sim X,~\deltabad(\xx)\gg X^{\eta} }}\frac{\mathfrak{S}(\xx)}{|\xx||\prod_{i=1}^4L_i(\xx)|^\frac{1}{4}}.
\end{equation}
In both of these sums 
we can apply Lemma \ref{lem:interview'} to   estimate $\mathfrak{S}(\xx)$.

We start by  estimating \eqref{eq:SX1}. Combining  Lemmas \ref{lem:interview'} and \ref{le:smallvalueLi}
with \eqref{eq:calculus}, we obtain
\begin{align*}
	S_1(X)
	&\ll_{\ve,\eta} \frac{1}{X^\frac{7}{4}}\sum_{\substack{\xx\in\BZprim^2,~\varDelta(\xx)\neq\square\\|\xx|\sim X,~|L_1(\xx)|\ll X^{1-\eta}}}\frac{\deltabad(\xx)^\varepsilon L(1,\chi_{Q_\xx})}{|L_1(\xx)|^\frac{1}{4}}\\
	&\ll_{\varepsilon,\eta}
	\frac{X^{5\ve}}{X^\frac{7}{4}} \sum_{R\nearrow X^{1-\eta}} \frac{1}{R^\frac{1}{4}}\sum_{\substack{\xx\in\BZprim^2,~\varDelta(\xx)\neq\square\\|\x|\sim X,~|L_1(\x)|\sim R}} 1,
\end{align*}
on introducing a dyadic parameter for the range of $|L_1(\x)|$.
It follows easily that 
$$
S_1(X) 
	\ll_{\varepsilon,\eta}
	\frac{X^{5\ve}}{X^\frac{7}{4}}
\sum_{R\nearrow X^{1-\eta}} 
XR^{\frac{3}{4}} \ll  X^{-\frac{3}{4}\eta+5\ve}
\ll X^{-\frac{\eta}{2}},
$$
by fixing $\varepsilon$ to be sufficiently small. 

Now we handle  $S_2(X)$ similarly. It follows from Lemma  \ref{lem:interview'} and \eqref{eq:calculus} that 
\begin{align*}
	S_2(X)&\ll_\ve \frac{X^{5\ve}}{X^\frac{7}{4}}
	\sum_{i_1=1}^4
	\sum_{R\nearrow X} \frac{1}{R^\frac{1}{4}}\sum_{\substack{\xx\in\BZprim^2,~\varDelta(\xx)\neq\square\\ |\x|\sim X, ~|L_{i_1}(\x)|\sim R\\\deltabad(\xx)\gg X^\eta}}1.
\end{align*}
The condition $\deltabad(\xx)\gg X^\eta$ implies in particular $\deltabad(\xx)\gg (XR)^{\frac{\eta}{2}}$. Applying  Lemma \ref{lem:stain} with
$\delta=\frac{\eta}{2}$ and $S=X$, we obtain
$$
S_2(X)\ll_\varepsilon \frac{X^{5\ve}}{X^\frac{7}{4}}\sum_{R\nearrow X} \frac{X R^{1-\frac{\eta}{16}}}{R^\frac{1}{4}}\ll X^{-\frac{\eta}{16}+5\varepsilon}.
$$
Thus 
$
S_2(X)\ll X^{-\frac{\eta}{32}},
$
on taking $\ve$ sufficiently small.

Invoking the bound \eqref{eq:brown-bag} for $\sigma_{\infty,w}(\xx)$, we may now apply 
our  bounds for $S_1(X)$ and $S_2(X)$ to deduce that 
there is a satisfactory  overall contribution to the main term in \eqref{eq:eqinter1}, corresponding to the failure of   \eqref{eq:cond2} or \eqref{eq:cond3}.
The proof  of the lemma is now  completed.
\end{proof}

It remains to remove the smooth weights, using the previous result to deduce a similar asymptotic formula for the counting function $L_{w_0}(B,X)$, where $w_0$ is the characteristic function of $[-1,1]^4$.

\begin{lemma}\label{prop:smoothtoconstantweight}
Assume that $X$ lies in the range \eqref{eq:Xrange}.
Then, for $\eta>0$ sufficiently small, we have
	$$
	L_{w_0}(B,X)=\frac{B}{\zeta(2)}\sum_{\substack{\xx\in\BZprim^2\\ |\xx|\sim X\\\varDelta(\xx)\neq\square}}\frac{\sigma_{\infty}(\xx)\mathfrak{S}(\xx)}{|\xx|}+O(\eta^\frac{1}{2}B)+O_{\eta}(B^{1-\frac{1}{16}\eta^2}).
	$$
\end{lemma}
\begin{proof}
We mimic  the procedure of \cite[\S5.3]{duke}, which  relates the counting function 
$L_{w_0}(B,X)$ to one in which smooth weights appear. For each $\eta>0$ sufficiently small, we fix two smooth weight functions $w_1,w_2$ satisfying the requirements of \cite[Lemma 4.13]{duke}. Thus 
\begin{equation}\label{eq:sigmaw0wi}
	\sigma_{\infty,w_i}(\xx)-\sigma_{\infty}(\xx)\ll \frac{\eta^\frac{1}{2}}{\prod_{i=1}^{4}|L_i(\xx)|^\frac{1}{4}}.
\end{equation}
Moreover,
$$
L_{w_1}(B,X)\leqslant L_{w_0}(B,X)\leq 
L_{w_0}(4\eta^2B,X)+L_{w_2}(B,X).
$$
We can apply Lemma 
\ref{prop:countsmoothweight} to estimate $L_{w_1}(B,X)$ and $L_{w_2}(B,X)$.
 Moreover, on recalling \eqref{eq:NwBX}, we deduce from   Theorem \ref{t:upper} that 
\begin{align*}
L_{w_0}(4\eta^2B,X)
&\ll \eta^2 B +\eta^{4/3}X^{4/3}B^{2/3}\ll \eta B,
\end{align*}
since  $X\ll B^\frac{1}{4}$ in 
\eqref{eq:Xrange}.

In view of  \eqref{eq:sigmaw0wi}, for $i=1,2$,
it  remains to show that
\begin{equation}\label{eq:SX}
	S_0(X)=\sum_{\substack{\xx\in\BZprim^2\\ |\xx|\sim X\\\varDelta(\xx)\neq\square}}\frac{\mathfrak{S}(\xx)}{|\xx||\prod_{i=1}^4L_i(\xx)|^\frac{1}{4}}=O(1),
\end{equation}
in order to  complete the proof of the lemma.  We can adopt a similar argument to the treatment of $S_1(X)$ and $S_2(X)$ in \eqref{eq:SX1} and \eqref{eq:SX2}, respectively.
Appealing to   Lemma \ref{lem:interview'} and breaking into dyadic intervals, we obtain
\begin{align*}
	S_0(X)&\ll \sum_{i_1=1}^{4}\frac{1}{X^{\frac{7}{4}}}\sum_{R \nearrow X}\frac{1}{R^\frac{1}{4}}\sum_{\substack{\xx\in\BZprim^2\\|\xx|\sim X,~|L_{i_1}(\xx)|\sim R\\\varDelta(\xx)\neq\square}}\deltabad(\xx)^\varepsilon L(1,\chi_{Q_\xx}).
	\end{align*}
	Recalling 
	\eqref{eq:defW}, 
an application of Lemma \ref{lem:Wi1} with $\delta=\frac{1}{16}$ and $S=X$, therefore yields
\begin{align*}
S_0(X)	
&\ll_{\varepsilon} \frac{1}{X^{\frac{7}{4}}}\sum_{R \nearrow X}\frac{1}{R^\frac{1}{4}}
\left(XR
+X^{\frac{33}{32}+\ve}R^{\frac{1}{32}}+X^{\frac{11}{16}+\ve}R^{\frac{5}{4}}\right)\\
	&\ll\frac{1}{X^{\frac{7}{4}}}\left(X^\frac{7}{4}+X^\varepsilon(X^\frac{27}{16}+X^{\frac{33}{32}})\right).
\end{align*} 
Thus $S_0(X)=O(1)$ on  fixing $\varepsilon$ to be small enough. This establishes \eqref{eq:SX}, thereby completing the proof of the lemma.
\end{proof}

We are now ready to deduce Proposition \ref{thm:countingxsmall}.
Lemma \ref{prop:smoothtoconstantweight} implies that
\begin{align*}
	\sum_{B^{2\eta}\ll X\ll B^{\frac{1}{4}-4\eta}}L_{w_0}(B,X)&=\frac{B}{\zeta(2)}\sum_{\substack{\xx\in\BZprim^2\\\varDelta(\xx)\neq\square\\ B^{2\eta}\ll |\xx|\ll B^{\frac{1}{4}-4\eta}}}\frac{\sigma_{\infty}(\xx)\mathfrak{S}(\xx)}{|\xx|}\\ 
	& \quad+O(\eta^\frac{1}{2}B\log B)+O_{\eta}(B^{1-\frac{1}{16}\eta^2}\log B).
\end{align*}
We may clearly take $\log B=O(B^{\frac{1}{32}\eta^2})$ in the second error term.
Moreover, on recalling \eqref{eq:SX}, we obtain
\begin{align*}
	\frac{B}{\zeta(2)}
	\sum_{\substack{\xx\in\BZprim^2\\
	\varDelta(\xx)\neq\square\\ B^{\frac{1}{4}-4\eta} \ll |\xx|\ll B^\frac{1}{4}}}\frac{\sigma_{\infty}(\xx)\mathfrak{S}(\xx)}{|\xx|} &\ll B\sum_{B^{\frac{1}{4}-4\eta} \ll X\ll B^\frac{1}{4}} S_0(X)\ll \eta B\log B.
\end{align*}
Moreover, with further recourse to \eqref{eq:SX}, it also follows that 
\begin{align*}
	\sum_{B^{2\eta}\ll X\ll B^{\frac{1}{4}-4\eta}}L_{w_0}(B,X)&=\frac{B}{\zeta(2)}\sum_{\substack{\xx\in\BZprim^2\\\varDelta(\xx)\neq\square\\ B^{2\eta}\leq |\xx|\leq  B^{\frac{1}{4}}}}\frac{\sigma_{\infty}(\xx)\mathfrak{S}(\xx)}{|\xx|}
	+O(\eta^\frac{1}{2}B\log B)+O_{\eta}(B^{1-\frac{1}{32}\eta^2}).
\end{align*}
Finally, it follows from Theorem \ref{t:upper} that 
$$
\sum_{\substack{X\ll B^{2\eta}\text{ or}\\ B^{\frac{1}{4}-4\eta}\ll X\ll B^\frac{1}{4}}}L_{w_0}(B,X)\ll \eta B\log B.$$
This therefore finishes the proof of Proposition  \ref{thm:countingxsmall}. 

\subsection{Proof of Proposition \ref{prop:lid}: preliminaries}

In this section we are concerned with the asymptotic behaviour of the term
\begin{equation}\label{eq:M1B}
	M(B)=	\sum_{\substack{\xx\in\BZprim^2\\ B^{2\eta}\leqslant |\xx|\leqslant B^\frac{1}{4}\\ \prod_{i=1}^{4}L_i(\xx)\neq\square }}\frac{\sigma_{\infty}(\xx)\mathfrak{S}(\xx)}{|\xx|}.
\end{equation}
The line of attack follows \cite[\S6.2]{duke}, but we  face  extra difficulties that are similar to the ones we encountered in  \S\ref{s:DA}. The basic idea is to restrict each series $\mathfrak{S}(\xx)$ to a sum over small moduli, before interchanging the order of summation. To achieve this, it will be crucial to achieve sufficient cancellation when averaging over the $\xx$-sum, which is harder in this setting, since 
we have half the number of $\xx$-variables compared to 
the variety \eqref{eq:duke} considered in
\cite{duke}.

	For $\xx\in\BZprim^2$ and $q\in\BN$, we let 
	$$
	S_q(\xx)=\sum_{\substack{a \bmod{q}\\ 
	\gcd(a,q)=1}
	}\sum_{\bb\in(\BZ/q\BZ)^4}\operatorname{e}_q\left(a\sum_{i=1}^{4}L_i(\xx)b_i^2\right).
	$$
This is  multiplicative in $q$ and we have 
$
\mathfrak{S}(\xx)=\sum_{q=1}^\infty q^{-4} S_q(\xx).
$ 
It will be useful to collect together some estimates for $S_q(\x)$ that can be extracted from  \cite[\S4.2]{duke}.
\begin{lemma}\label{le:Sqxx}
Let $\xx\in\BZprim^2$ and let $q\in \NN$. Then the following estimates hold for any $\ve>0$.
\begin{enumerate}
	\item[(i)] $S_q(\xx)\ll q^3\gcd(q,\varDelta(\xx))^\frac{1}{2}.$
	\item[(ii)] If $p\nmid 2\varDelta(\xx)$ and $r\in\BN$, then
	$$S_{p^r}(\xx)=\left(\frac{\varDelta(\xx)}{p^r}\right)p^{3r}\phi^*(p^r),$$
	where $\phi^*(n)=\phi(n)/n$. Moreover, 
	$S_{p^r}(\xx)=0$ 
	if $p\mid \varDelta(\xx)$ but $p\nmid 2\deltabad(\xx)$. 
	\item[(iii)] We have 
	$$
	\sum_{q\leqslant X}\frac{S_q(\xx)}{q^3}
\ll_{\varepsilon}
|\varDelta(\xx)|^{\frac{3}{16}+\varepsilon}\deltabad(\xx)^\frac{3}{8}X^{\frac{1}{2}+\varepsilon},
$$
if  $\varDelta(\xx)\neq \square$.
	\item[(iv)] We have 
	$$
	\sum_{q\leqslant X}\frac{|S_q(\xx)|}{q^4}
\ll_{\varepsilon}
X^\varepsilon.
$$
\end{enumerate}
\end{lemma}
\begin{proof}
It follows from  \cite[Lemma 4.5]{duke} that  $S_q(\xx)\ll q^3\prod_{i=1}^{4}\gcd(q,L_i(\xx))^\frac{1}{2}.$ 
But then part (i) is a consequence of the observation \eqref{le:discform} and the fact that $\x$ is primitive.
The formulae in (ii) follow from \cite[Lemmas 4.6 and 4.8]{duke}. 
The case $j=3$ of part  (iii) is the  same as \cite[Lemma~4.9]{duke}. It remains to prove the estimate when $j=4$.
Appealing to part  (ii), we obtain 
	\begin{align*}
	\sum_{q\leqslant X}\frac{S_q(\xx)}{q^4}
	&\ll \left|\sum_{\substack{q_2\leqslant X\\ q_2\mid (2\deltabad(\xx))^\infty}}\frac{S_{q_2}(\xx)}{q_2^4}\right|\left|\sum_{q_1\leqslant X/q_2}\left(\frac{\varDelta(\xx)}{q_1}\right)\frac{\phi^*(q_1)}{q_1}\right|
	\\
&\ll  \#\left\{q_2\in \NN: q_2\leqslant X, ~ q_2\mid (2\deltabad(\xx))^\infty\right\}\log X,
	\end{align*}
since part (i) implies that $S_{q_2}(\xx)\ll q_2^{\frac{7}{2}}$.
The remaining cardinality is easily seen to be $O_\ve(X^\ve\deltabad(\xx)^\ve)$, which thereby 
completes the proof.
\end{proof}

We now carry out the proof of Proposition \ref{prop:lid} in a series of steps

\begin{lemma}[Reduction to small $\deltabad(\xx)$]\label{le:smalldeltabad}
	We have
	$$\sum_{\substack{\xx\in\BZprim^2\\ B^{2\eta}\leqslant |\xx|\leqslant B^\frac{1}{4}\\ \varDelta(\xx)\neq\square }}\frac{\sigma_{\infty}(\xx)\mathfrak{S}(\xx)}{|\xx|}=\sum_{\substack{\xx\in\BZprim^2\\ B^{2\eta}\leqslant |\xx|\leqslant B^\frac{1}{4}\\ \varDelta(\xx)\neq \square\\ \deltabad(\xx)\leqslant B^{\eta/1000}}}\frac{\sigma_{\infty}(\xx)\mathfrak{S}(\xx)}{|\xx|}+O_\eta(1).$$
\end{lemma}
\begin{proof}
	We apply the  bound \eqref{eq:brown-bag} for $\sigma_\infty(\x)$, together with the bound
	$\mathfrak{S}(\xx)\ll_{\varepsilon}|\x|^{5\ve}$,
	which follows from Lemma \ref{lem:interview'} and \eqref{eq:calculus}.
	On  executing the dyadic decomposition in the same way as in the proof of Proposition \ref{s:monats}, it follows from Lemma~\ref{lem:stain} that 
	\begin{align*}
		\sum_{\substack{\xx\in\BZprim^2\\ B^{2\eta}\leqslant |\xx|\leqslant B^\frac{1}{4}\\ \varDelta(\xx)\neq \square\\
		\deltabad(\xx)> B^{\eta/1000}}}\frac{\sigma_{\infty}(\xx)\mathfrak{S}(\xx)}{|\xx|}
		&\ll_{\varepsilon}	\sum_{B^{2\eta}\swarrow S\nearrow B^\frac{1}{4}}\sum_{R \nearrow S} \frac{S^{5\varepsilon}(SR)^{1-\eta/400}}{S^\frac{7}{4}R^\frac{1}{4}}
\ll_{\varepsilon}B^{-\eta^2/100+5\varepsilon},
	\end{align*} 
	since  $\deltabad(\xx)>B^{\eta/1000}\geqslant (SR)^{\eta/500}$. The result follows on choosing $\ve$ small enough.
\end{proof}

\begin{lemma}[First truncation of $\mathfrak{S}(\xx)$]\label{le:finiteqsum}
	We have
	$$\sum_{\substack{\xx\in\BZprim^2\\ B^{2\eta}\leqslant |\xx|\leqslant B^\frac{1}{4}\\ \varDelta(\xx)\neq\square }}\frac{\sigma_{\infty}(\xx)\mathfrak{S}(\xx)}{|\xx|}=\sum_{\substack{\xx\in\BZprim^2\\ B^{2\eta}\leqslant |\xx|\leqslant B^\frac{1}{4}\\\varDelta(\xx)\neq\square }}\frac{\sigma_{\infty}(\xx)\mathfrak{S}(\xx;B^{100})}{|\xx|}+O(B^{-1}),$$
	where
	$
	\mathfrak{S}(\x;N) =\sum_{q\leq N}q^{-4}S_q(\x).$
\end{lemma}
\begin{proof}
Applying partial summation, 
it follows from part (iii) of Lemma  
 \ref{le:Sqxx} that 
	$$
	\sum_{q>B^{100}}\frac{S_q(\xx)}{q^4}\ll |\varDelta(\xx)|^\frac{3}{16}\deltabad(\xx)^\frac{3}{8}B^{-20},
	$$
uniformly for any $\xx\in\BZprim^2$  such that $\Delta(\x)\neq \square$ and $|\xx|\leq B^{\frac{1}{4}}$.
The bound \eqref{eq:brown-bag} now yields
	$$\sum_{\substack{\xx\in\BZprim^2\\ B^{2\eta}\leqslant |\xx|\leqslant B^\frac{1}{4}\\ \varDelta(\xx)\neq\square }}\frac{\sigma_{\infty}(\xx)}{|\xx|}\sum_{q>B^{100}}\frac{S_q(\xx)}{q^4}
	\ll B^{-20}\sum_{\substack{\xx\in\BZ^2\\ |\xx|\leqslant B^\frac{1}{4}\\ \varDelta(\xx)\neq 0}}|\varDelta(\xx)|^{\frac{5}{16}}\ll B^{-1},$$
since  $\deltabad(\xx)\leqslant |\varDelta(\xx)|\ll |\xx|^4.$ 
\end{proof}

Having truncated the $q$-sum, we proceed to show that there is a  negligible  contribution from $\xx$ such that $\varDelta(\xx)=\square$.

\begin{lemma}\label{le:squareaway}
We have 
	$$
	\sum_{\substack{\xx\in\BZprim^2\\ B^{2\eta}\leqslant |\xx|\leqslant B^\frac{1}{4}\\\varDelta(\xx)=\square }}\frac{\sigma_{\infty}(\xx)\mathfrak{S}(\xx;B^{100})}{|\xx|}=O_\eta(B^{-\eta}).$$
\end{lemma}
\begin{proof}
To begin with, it follows from part (iv) of Lemma \ref{le:Sqxx} that 
$
\mathfrak{S}(\xx;B^{100})=O_{\varepsilon} (B^\varepsilon)$,  
uniformly for $\xx\in\BZprim^2$ with $|\xx|\leqslant B^\frac{1}{4}$. Moreover, 
\eqref{eq:brown-bag} implies that  $\sigma_\infty(\x)=O(1)$. Thus 
	$$
		\sum_{\substack{\xx\in\BZprim^2\\ B^{2\eta}\leqslant |\xx|\leqslant B^\frac{1}{4}\\\varDelta(\xx)=\square }}\frac{\sigma_{\infty}(\xx)\mathfrak{S}(\xx;B^{100})}{|\xx|}
	\ll_\ve B^{-2\eta+\ve} \#\left\{\xx\in\BZprim^2: |\xx|\leqslant B^\frac{1}{4}, ~\varDelta(\xx)=\square\right\}.
	$$
	In the spirit of  the proof of Proposition \ref{prop:M1Xsmall}, the condition $\varDelta(\xx)=\square$ implies that  $(y,\xx)$ lies on the genus one curve $y^2=\prod_{i=1}^{4}L_i(\xx)$. 
Thus the number of $\xx\in\BZprim^2$ with $|\xx|\leqslant B^\frac{1}{4}$ which verify 
this condition is $O_\ve (B^\varepsilon)$.  The  lemma now follows on taking $\ve=\eta/2$. 
\end{proof}

We have now come to the most difficult step in the proof of Proposition \ref{prop:lid}. 
\begin{proposition}[Second truncation of $\mathfrak{S}(\x)$]\label{prop:fksum}
	$$\sum_{\substack{\xx\in\BZprim^2\\ B^{2\eta}\leqslant |\xx|\leqslant B^\frac{1}{4}\\ \deltabad(\xx)\leqslant B^{\eta/1000} }}\frac{\sigma_{\infty}(\xx)}{|\xx|}\left(\sum_{B^{\eta/10}<q\leqslant B^{100}}\frac{S_q(\xx)}{q^4}\right)=O_\eta(B^{-\eta/500}).$$
\end{proposition}

This result is a direct analogue of \cite[Lemma 6.6]{duke}. However, in that setting a higher power  $|\x|^3$ appears in
the denominator, which has the effect of making the proof a relatively simple application of the large sieve for real characters. The 
 proof of Proposition~\ref{prop:fksum} is more delicate and we have divided it into several steps.

Following the template laid out to prove  Proposition \ref{prop:M1Xsmall}, we will  execute a dyadic decomposition of the range of $\xx$, according to the  smallest value of  $|L_i(\xx)|$.
Since $|L_{i}(\xx)|= |L_j(\xx)|$ for any indices $i\neq j$ only if $\xx$ takes values in a finite set, 
we see that there is an overall contribution $O(1)$ to the sum in the proposition from such $\xx$.
This therefore allows us to partition the $\xx$-sum into four sums where
$\min_{i\neq i_1}|L_i(\xx)|>|L_{i_1}(\xx)|,$
for $i_1\in \{1,\dots,4\}$.      We shall   assume, without loss of generality, that $i_1=1$. 
We introduce a dyadic parameter $S$ for $|\xx|$, and $R$ for $|L_1(\xx)|$, for $R\ll S$ and $B^{2\eta}\ll S\ll B^{\frac{1}{4}}$.
Let 
\begin{equation}\label{eq:SSR}
	\CS(S,R)=\left\{\xx\in\BZprim^2:|\xx|\sim S, ~|L_i(\xx)|>|L_1(\x)|\sim R \text{ for $i\geq 2$}\right\}.
\end{equation}
Then we shall be interested in bounding
$$
\Sigma(S,R)=
\sum_{\substack{\xx\in\CS(S,R)\\  \deltabad(\xx)\leqslant B^{\eta/1000} }}\frac{\sigma_{\infty}(\xx)}{|\xx|}\left(\sum_{B^{\eta/10}<q\leqslant B^{100}}\frac{S_q(\xx)}{q^4}\right),
$$
for given $R,S$
such that  $R\ll S$ and $B^{2\eta}\ll S\ll B^{\frac{1}{4}}$.
We shall prove the following result.

\begin{lemma}\label{prop:lace}
Let $\eta>0$ and 
let $R,S\geq 1$ be such that $B^{2\eta}\ll S\ll B^{\frac{1}{4}} $ and $R\ll S$. Then 
$$
\Sigma(S,R)= O_\eta(B^{-\eta/400}).
$$
\end{lemma}

The statement of Proposition 
\ref{prop:fksum} is an easy consequence of this, on summing over dyadic intervals for $R$ and $S$.
Before proving it, we take the opportunity to record a basic estimate for the partial derivative of the real valued  analytic  function that weights our sum   $\Sigma(S,R)$.

\begin{lemma}\label{le:derivative}
Let $j\in \{1,2\}$ and let $K\in \RR[x_1,x_2]$ be a non-zero linear form. Then the following hold:
\begin{itemize}
\item[(i)] 
$$
		\frac{\partial\sigma_{\infty}(\xx)}{\partial x_j}
\ll |\varDelta(\xx)|^{-\frac{1}{3}}(\min|L_i(\xx)|)^{-\frac{2}{3}}.
$$
\item[(ii)] 
	$$\frac{\partial}{\partial x_j}\frac{\sigma_{\infty}(\xx)}{K(\xx)}\ll_K |K(\xx)|^{-1}|\varDelta(\xx)|^{-\frac{1}{3}}(\min|L_i(\xx)|)^{-\frac{2}{3}}+|K(\xx)|^{-2}|\varDelta(\xx)|^{-\frac{1}{4}},
	$$
\end{itemize}
\end{lemma}
\begin{proof}
We shall assume without loss of generality that $j=2$. 
	We have
	$$
\frac{\partial}{\partial x_2}\frac{\sigma_{\infty}(\xx)}{K(\xx)}
	\ll_K\left|\frac{\partial\sigma_{\infty}(\xx)}{\partial x_2}\frac{1}{K(\xx)}\right|+\left|\frac{\sigma_{\infty}(\xx)}{K(\xx)^2}\right|.$$
	In view of \eqref{eq:brown-bag}, the second term gives rise to the second error term in part (ii) of the lemma. 
	Thus part (ii) follows from part (i). 
	
	For any $\psi\in\BR$, we write $I(\psi)=\int_{-1}^{1}\operatorname{e}(\psi y^2)\operatorname{d}y$. 
	We have  $I(\psi)\ll \min(1,|\psi|^{-1/2}),$
	as recorded in  \cite[Lemma 4.4]{duke}, for example. 
	We have 
	$$\sigma_{\infty}(\xx)=\int_\BR\left(\prod_{i=1}^{4}\int_{-1}^{1}\operatorname{e}\left(-\theta L_i(\xx)y_i^2\right)\operatorname{d}y_i\right)\operatorname{d}\theta=\int_\BR\left(\prod_{i=1}^{4} I(-\theta L_i(\xx))\right)\operatorname{d}\theta.
	$$
	On the other hand, 
	$$\frac{\partial I(-\theta L_i(\xx))}{\partial x_2}=\frac{\partial L_i(\xx)}{\partial x_2}\frac{1}{2L_i(\xx)}\int_{-1}^{1}y\frac{\partial}{\partial y}\operatorname{e}(-\theta L_i(\xx)y^2)\operatorname{d}y.$$
	The integral on the right hand side is uniformly bounded, whence $$\frac{\partial I(-\theta L_i(\xx))}{\partial x_2}\ll |L_i(\xx)|^{-1}.$$
In now follows that 
	\begin{align*}
		\frac{\partial\sigma_{\infty}(\xx)}{\partial x_2}
		&\ll\sum_{i=1}^{4}|L_i(\xx)|^{-1}\int_\BR \min\left(1,|\theta|^{-\frac{3}{2}}|\prod_{j\neq i}|L_j(\xx)|^{-\frac{1}{2}}\right)\operatorname{d}\theta \ll |\varDelta(\xx)|^{-\frac{1}{3}}(\min|L_i(\xx)|)^{-\frac{2}{3}},
	\end{align*}
which establishes part (i).
\end{proof}

\begin{proof}[Proof of Lemma \ref{prop:lace}]
Throughout the proof we may assume that the parameter $\eta>0$ is fixed but arbitrarily small.
Let $\theta>\frac{3}{16}$ be a parameter to be decided upon in due course. It follows from part (iv) of
 Lemma~\ref{le:Sqxx} that 
 $$
 \sum_{q\leq B^{100}} \frac{|S_q(\x)|}{q^4}\ll_\ve B^\ve.
 $$
 Combining \eqref{eq:brown-bag} with Lemma \ref{lem:stain}, we deduce that the overall contribution to 
$\Sigma(S,R)$ from $\x$ such that $\deltabad(\x)>(SR)^\theta$ is 
$$
\ll_\ve\frac{B^\ve (SR)^{1-\frac{\theta}{8}}}{S^{\frac{7}{4}}R^{\frac{1}{4}}}\ll B^\ve S^{-\frac{\theta}{4}}\ll B^{-\frac{\theta \eta}{2}+\ve},
$$
since $R\ll S$ and $S\gg B^{2\eta}$. This is a satisfactory contribution since $\theta>\frac{3}{16}$.

	Using the multiplicativity of $S_q(\xx)$ in $q$, it follows from 
	part (ii) of Lemma~\ref{le:Sqxx}  that we may proceed under the assumption that
$$
\Sigma(S,R)=	\sum_{\substack{\xx\in\CS(S,R)\\ \deltabad(\xx)\leqslant \Theta}}
\frac{\sigma_{\infty}(\xx)}{|\xx|}\sum_{\substack{q_2\leqslant B^{100}\\ q_2\mid (2\deltabad(\xx))^\infty}}\frac{S_{q_2}(\xx)}{q_2^4}
\hspace{-0.4cm}
\sum_{\substack{B^{\eta/10}/q_2\leq q_1\leq B^{100}/q_2\\ 2\nmid q_1}}
\hspace{-0.3cm}
\left(\frac{\varDelta(\xx)}{q_1}\right)\frac{\phi^*(q_1)}{q_1},
$$
where
$$
\Theta=\min \left(B^{\eta/1000}, (SR)^{\theta}\right).
$$
We will need to show that the sums over $q_1$ and $q_2$ can be truncated satisfactorily. 
The inner sum over $q_1$ is $O(\log B)$. Hence part (i) of
Lemma~\ref{le:Sqxx}  
and \eqref{eq:brown-bag} 
implies that the contribution from  $q_2>B^{\eta/100}$ is 
\begin{align*}
\ll 
	\frac{\log B}{B^{\eta/200}} \sum_{\substack{\xx\in\CS(S,R)\\ \deltabad(\xx)\leqslant \Theta}}\frac{1}{|\xx||\Delta(\xx)|^{\frac{1}{4}}}
\#\{
q_2\leqslant B^{100}: q_2\mid (2\deltabad(\xx))^\infty \}
&\ll_\ve B^{-\eta/200+\ve} \frac{\#\CS(S,R)}{S^{\frac{7}{4}}R^\frac{1}{4}}\\
&\ll_\ve B^{-\eta/200+\ve},
\end{align*}
since $R\ll S$. Taking $\ve=\eta/400$, this is satisfactory for Lemma \ref{prop:lace}.
Next, we put 
	\begin{equation}\label{eq:N1new}
		N_1=(S^3 R)^{2\theta}B^{\eta/50}.
	\end{equation} 
	Then it follows from the  Burgess bound, in the form of Lemma \ref{le:Burgess}, that 	
	$$
	\sum_{\substack{N_1<q_1\leqslant B^{100}/q_2\\
	2\nmid q_1}}\left(\frac{\varDelta(\xx)}{q_1}\right)\frac{\phi^*(q_1)}{q_1}
	\ll_\theta N_1^{-\frac{1}{2}}|\varDelta(\xx)|^\theta\ll B^{-\eta/100},
	$$
	since $\theta>\frac{3}{16}$.
Since the contribution from the $q_2$ sum is $O_\ve(B^\ve)$, we obtain the overall contribution 
	$$\ll_\theta B^{-\eta/100+\varepsilon}\frac{\#\CS(S,R)}{S^{\frac{7}{4}}R^\frac{1}{4}}\ll B^{-\eta/100+\varepsilon}.$$ 
This is satisfactory for Lemma \ref{prop:lace} on  taking  $\varepsilon=\eta/200$.

In summary, it suffices to proceed under the assumption that 
$$
\Sigma(S,R)=	\sum_{\substack{\xx\in\CS(S,R)\\ \deltabad(\xx)\leqslant \Theta}}\frac{\sigma_{\infty}(\xx)}{|\xx|}\sum_{\substack{q_2\leqslant B^{\eta/100}\\ q_2\mid (2\deltabad(\xx))^\infty}}\frac{S_{q_2}(\xx)}{q_2^4}\sum_{\substack{q_1\in I_{q_2}\cap \ZZ\\ 2\nmid q_1}}\left(\frac{\varDelta(\xx)}{q_1}\right)\frac{\phi^*(q_1)}{q_1},
$$
where 
\begin{equation}\label{eq:int-q2}
I_{q_2}=\left[\frac{B^{\eta/10}}{q_2}~,~ \min\left(N_1,\frac{B^{100}}{q_2}\right)\right].
\end{equation}
Next, we sort this sum according to the value of  $\deltabad(\xx)$.
Thus 
\begin{equation}\label{eq:yellow-bach}
\Sigma(S,R)=
\sum_{\substack{r\leqslant \Theta\\ 
	r \text{ squarefull}}} \Sigma_r(S,R),
	\end{equation}
	where
$$
	\Sigma_r(S,R)=	\sum_{\substack{\xx\in\CS(S,R)\\ \deltabad(\xx)=r}}\frac{\sigma_{\infty}(\xx)}{|\xx|}\sum_{\substack{q_2\leqslant B^{\eta/100}\\ q_2\mid (2r)^\infty}}\frac{S_{q_2}(\xx)}{q_2^4}\sum_{\substack{q_1\in I_{q_2}\cap \ZZ\\ \gcd(q_1,2r)=1}}\left(\frac{\varDelta(\xx)}{q_1}\right)\frac{\phi^*(q_1)}{q_1}.
$$

The idea is to  now bring the  $\xx$-sum to the inside, in order to exploit  cancellation from the Jacobi symbol. 
To do so, we must  first exchange the $q_2$-sum and the $\xx$-sum, by sorting the $\xx$-sum into residue classes modulo $q_2$.
This  leads to the expression
\begin{equation}\label{eq:sumr}
\Sigma_r(S,R)=
\sum_{\substack{q_2\leqslant B^{\eta/100}\\ q_2\mid (2r)^\infty}}\frac{1}{q_2^4}
\sum_{\substack{\cc\in(\BZ/q_2\BZ)^2\\
\gcd(q_2,\cc)=1}}S_{q_2}(\cc)
\sum_{\substack{q_1\in I_{q_2}\cap \ZZ\\ \gcd(q_1,2r)=1}} \frac{\phi^*(q_1)}{q_1} \cdot U(q_1,q_2;\cc),
\end{equation}
where
$$
U(q_1,q_2;\cc)=
\sum_{\substack{\xx\in\CS(S,R)\\ \deltabad(\xx)=r\\ \xx\equiv \cc\bmod{q_2}}}\frac{\sigma_{\infty}(\xx)}{|\xx|}\left(\frac{\varDelta(\xx)}{q_1}\right).
$$
To handle the condition $\deltabad(\xx)=r$, we note that it is equivalent to the pair of conditions 
$
\gcd\left(r,\varDelta(\xx)/r\right)=1$ and $\mu^2\left(\varDelta(\xx)/r\right)=1$.
These can both be detected using the  M\"obius function, leading to 
$$
U(q_1,q_2;\cc)
=\sum_{d_1\mid r}\mu(d_1)\sum_{d_2}\mu(d_2) 
\sum_{\substack{\xx\in\CS(S,R)\\  r[d_1,d_2^2]\mid \varDelta(\xx)\\ 
\xx\equiv \cc \bmod{q_2}}}\frac{\sigma_{\infty}(\xx)}{|\xx|}\left(\frac{\varDelta(\xx)}{q_1}\right).
$$
Note $\gcd(q_1,d_1)=1$, since 
 $\gcd(q_1,r)=1$. Hence we have   $\gcd(q_1,rd_1d_2)=1$.

Clearly $\deltabad(\xx)\geqslant d_2^2$ and so we must have $d_2\leq \Theta^{1/2}\leq (SR)^{\theta/2}$.
We  therefore have 
$$
U(q_1,q_2;\cc)
=\sum_{d_1\mid r}\mu(d_1)\sum_{\substack{d_2\leq (SR)^{\eta/20}\\ \gcd(d_2,q_1)=1}}\mu(d_2) 
\sum_{\substack{\xx\in\CS(S,R)\\  r[d_1,d_2^2]\mid \varDelta(\xx)\\ 
\xx\equiv \cc \bmod{q_2}}}\frac{\sigma_{\infty}(\xx)}{|\xx|}\left(\frac{\varDelta(\xx)}{q_1}\right),
$$
where we recall that $\CS(S,R)$ is defined in \eqref{eq:SSR}.
Since $\x$ is primitive in the inner sum, 
it follows from 
\eqref{le:discform} that $\gcd(L_i(\x),L_j(\x))\mid \CD$ for $i\neq j$, where $\CD$ is defined in 
\eqref{eq:badprimes} and 
satisfies $\CD=O(1)$. 
We write 
$
r[d_1,d_2^2]=DE,
$
where $D$ only contains primes $p\nmid \CD$, while $p\mid E \Rightarrow p\mid \CD$.
We further break the $\x$-sum into congruences modulo $E$, finding that 
$$
\sum_{\substack{\xx\in\CS(S,R)\\  r[d_1,d_2^2]\mid \varDelta(\xx)\\ 
\xx\equiv \cc \bmod{q_2}}}\frac{\sigma_{\infty}(\xx)}{|\xx|}\left(\frac{\varDelta(\xx)}{q_1}\right)
=
\sum_{\substack{\mathbf{s}\bmod{E}\\
\gcd(\mathbf{s},E)=1\\
E\mid \Delta(\mathbf{s})
}} 
\sum_{\substack{\xx\in\CS(S,R)\\  D\mid \varDelta(\xx)\\ 
\xx\equiv \cc \bmod{q_2}\\ 
\xx\equiv \mathbf{s} \bmod{E}
}}\frac{\sigma_{\infty}(\xx)}{|\xx|}\left(\frac{\varDelta(\xx)}{q_1}\right).
$$
We claim that 
\begin{equation}\label{eq:tea}
\#\left\{ \mathbf{s}\in (\ZZ/E\ZZ)^2: \gcd(\mathbf{s},E)=1, ~E\mid \Delta(\mathbf{s})\right\} =O_\ve(E^{1+\ve}),
\end{equation}
for any $\ve>0$. By the Chinese remainder theorem it suffices to study the case where $E=p^e$ is a prime power. If $p^\lambda\mid \gcd(L_i(\mathbf{s}),L_j(\mathbf{s}))$ for $i\neq j$, then $p^\lambda \mid \mathcal{D}$. Thus the number of solutions modulo $p^\lambda$ is clearly $O(p^\lambda)$. The claimed bound \eqref{eq:tea} easily follows.

Since $\gcd(D,\CD)=1$, there is a bijection between  
$D\mid \Delta(\x)$ and vectors $(D_1,\dots,D_4)\in \NN^4$ 
with pairwise coprime coordinates, such  that $D_i\mid L_i(\x)$, for $1\leq i\leq 4$.
Thus 
$$
\sum_{\substack{\xx\in\CS(S,R)\\  D\mid \varDelta(\xx)\\ 
\xx\equiv \cc \bmod{q_2}\\ 
\xx\equiv \mathbf{s} \bmod{E}
}}\frac{\sigma_{\infty}(\xx)}{|\xx|}\left(\frac{\varDelta(\xx)}{q_1}\right)=
\sum_{\substack{D=D_1\cdots D_4\\ 
\gcd(D_i,D_j)=1
}}
\sum_{\substack{\xx\in\CS(S,R)\\  D_i\mid L_i(\xx)\\ 
\xx\equiv \cc \bmod{q_2}\\ 
\xx\equiv \mathbf{s} \bmod{E}
}}\frac{\sigma_{\infty}(\xx)}{|\xx|}\left(\frac{\varDelta(\xx)}{q_1}\right).
$$
We use the M\"obius function to remove the coprimality condition on $\x$, and we observe that 
$\sigma_\infty (k\x)=k^{-1}\sigma_\infty (\x)$ for any $k>0$. 
Thus  
\begin{equation}\label{eq:thread}
\begin{split}
U(q_1,q_2;\cc)
=~&\sum_{d_1\mid r}\mu(d_1)
\sum_{\substack{
d_2\leq (SR)^{\eta/20}\\ \gcd(d_2,q_1)=1}}\mu(d_2) 
\hspace{-0.3cm}
\sum_{r[d_1,d_2^2]=DE}
\sum_{\substack{\mathbf{s}\bmod{E}\\
\gcd(\mathbf{s},E)=1\\
E\mid \Delta(\mathbf{s})
}} 
\sum_{\substack{\DD\in \NN^4\\
D_1\cdots D_4=D\\
\gcd(D_i,D_j)=1}}
\sum_{\substack{k\ll R\\
\gcd(k,q_2)=1}
} \frac{\mu(k)}{k^2} 
U_{\DD',k},
\end{split}\end{equation}
where
$$
U_{\DD',k}=
\sum_{\substack{\xx\in \ZZ^2\\ 
|\x|\sim S', ~|L_1(\x)|\sim R'\\
|L_i(\x)|>|L_1(\x)| \text{ for $i\geq 2$}\\
D_i'\mid L_i(\x)\\
k\xx\equiv \cc \bmod{q_2}\\
k\xx\equiv \mathbf{s} \bmod{E}
}}\frac{\sigma_{\infty}(\xx)}{|\xx|}\left(\frac{\varDelta(\xx)}{q_1}\right),
$$
with 
$$
D_i'=\frac{D_i}{\gcd(D_i,k)}  \text{ for $1\leq i\leq 4$},  \quad S'=\frac{S}{k},
\quad R'=\frac{R}{k}.
$$
In particular, we clearly have $\gcd(D_i',D_j')=1$ for $i\neq j$ and, moreover,  
 $k $ is coprime to $q_2E$, since 
 $\cc$ is coprime to $q_2$
and $\mathbf{s}$ is coprime to $E$.

We now focus our attention 
on the sum $U_{\DD',k}$.
Suppose that   
$
L_1(x_1,x_2)=a_1x_1+b_1x_2,
$ 
for coprime integers $a_1,b_1$. Then there exists 
$\mathbf{M}\in \mathrm{SL}_2(\ZZ)$ with first row  equal to $(a_1,b_1)$. Making the change of variables 
$\y=\mathbf{M}\x$, we let $J_i(\yy)=L_i(\mathbf{M}^{-1}\yy)$, for $1\leq i\leq 4$,  and
$\varDelta^\prime(\yy)=J_1(\yy)\cdots J_4(\yy)$. 
Under this transformation, 
there exists $\mathbf{c}'\in \ZZ^2$ such that 
$$
U_{\DD',k}
=
\sum_{\substack{\yy\in\ZZ^2\cap \mathcal{R}\\
D_i'\mid J_i(\y)\\ 
\y\equiv \mathbf{c}'\bmod{[q_2,E]}\\
}}\frac{\sigma_{\infty}'(\yy)}{|\mathbf{M}^{-1}\yy|}\left(\frac{\varDelta'(\yy)}{q_1}\right),
$$
where  $\sigma_{\infty}^\prime(\yy)=\sigma_{\infty}(\mathbf{M}^{-1}\yy)$ and 
$$
\mathcal{R} =\left\{\y\in \RR^2:  
|y_1|\sim R', ~|\mathbf{M}^{-1}\y|\sim S',
 \text{ and $|J_i(\yy)|>|y_1|$ for $i\geqslant 2$}\right\}.
$$
Note that once $y_1$ is fixed, there exists an interval $K_{y_1}$ of length $O( S')$, such that 
$\y\in \mathcal{R}$ if and only if $y_2\in K_{y_1}$.
Hence
\begin{equation}\label{eq:ship}
U_{\DD',k}
=
\sum_{\substack{|y_1|\sim R'\\
y_1\equiv c_1' \bmod{[q_2,E]}}}
V(y_1),
\end{equation}
where
$$
V(y_1)=
\sum_{\substack{y_2\in K_{y_1}\cap \ZZ\\
D_i'\mid J_i(\y)\\ 
y_2\equiv c_2' \bmod{[q_2,E]}}}
\frac{\sigma_{\infty}'(\yy)}{|\mathbf{M}^{-1}\yy|}\left(\frac{\varDelta'(\yy)}{q_1}\right).
$$

We now seek to apply Lemma \ref{lem:PV-prep} to estimate $V(y_1)$. For this we recall that $\gcd(q_1,q_2D'E)=1$,
where $D'=D_1'\cdots D_4'$.
There exists a unique  factorisation $q_1=ut^2$, where $u$ is the largest square-free divisor of $q_1$.
We then deduce from Lemma \ref{lem:PV-prep} that 
$$
\sum_{\substack{y_2\in I\cap \ZZ	\\
D_i'\mid J_i(\y)\\ 
y_2\equiv c_2' \bmod{[q_2,E]}}}
\left(\frac{\varDelta'(\yy)}{q_1}\right)\ll_\ve
 \left(
		\frac{\vol(I)}{u^{\frac{1}{2}}[[q_2,E], D']} 
+				u^{\frac{1}{2}} \log(q_2D'E)	\right)
				u^{\ve}\gcd(y_1,u D'),
$$
for any $\ve>0$ and any  interval $I\subset \RR$. 
Note that 
$
[[q_2,E], D']\geq [E,D']=D'E,
$
since $D'$ and $E$ are coprime.

Armed with this bound,  it now follows from   Lemma \ref{le:derivative} and partial summation that 
\begin{align*}
V(y_1)
&\ll \sup_{I\subset K_{y_1}}\left|\sum_{\substack{y_2\in I\cap \ZZ\\  
D_i' \mid J_i(\y)\\
y_2\equiv c_2'\bmod{[q_2,E]}}}
\left(\frac{\varDelta^\prime(\yy)}{q_1}\right)\right| \cdot \sup_{y_2\in K_{y_1}} W(y_2),
\\ 
\end{align*}
where
\begin{align*}
W(y_2)
&=\vol(K_{y_1})\cdot \left|\frac{\partial }{\partial y_2}\frac{\sigma_{\infty}^\prime(\yy)}{|\mathbf{M}^{-1}\yy|}\right| +
\frac{\sigma_{\infty}'(\yy)}{|\mathbf{M}^{-1}\yy|}\ll \frac{1}{R'S'}.
\end{align*}
Hence
\begin{align*}
V(y_1)&\ll_\ve
 \left(
		\frac{S'}{u^{\frac{1}{2}}{D'E} }
+				u^{\frac{1}{2}} \log(q_2D'E)	\right)\frac{
				u^{\frac{\ve}{2}}\gcd(y_1,u D')}{R'S'}.
\end{align*}

On returning to \eqref{eq:ship} and summing over $y_1$, 
we obtain 
\begin{align*}
U_{\DD',k}
&\ll_\ve
\frac{
				u^{\frac{\ve}{2}}}{R'S'}
\left(
		\frac{S'}{u^{\frac{1}{2}}D'E} 
+				u^{\frac{1}{2}} \log(q_2D'E)	\right)
\tau(uD') \left(\frac{R'}{[q_2,E]}+1\right)\\
&\ll_\ve
\frac{
				(q_1q_2DE)^{\ve}}{R'S'}
\left(
		\frac{S'}{u^{\frac{1}{2}}D'E} 
+				u^{\frac{1}{2}}	\right)
\left(\frac{R'}{[q_2,E]}+1\right).
\end{align*}
But $$
D'E= \frac{DE}{\gcd(D_1,k)\cdots \gcd(D_4,k)}\gg \frac{DE}{k}\geq \frac{d_2^2}{k},
$$ 
since $DE=r[d_1,d_2^2]\geq d_2^2$.
Moreover, $R'/[q_2,E]\leq R'=R/k$ and $S'=S/k$. It therefore follows that 
\begin{align*}
U_{\DD',k}
&\ll_\ve 
kB^{2\ve}
\left(
		\frac{1}{u^{\frac{1}{2}}d_2^2} 
		+				\frac{u^{\frac{1}{2}}}{S}	\right),
\end{align*}
on noting that  $q_1q_2DE\leq B^2.$ 
We now insert this into \eqref{eq:thread} and apply \eqref{eq:tea}. 
Observe  that there are 
$O_\ve(B^\ve)$ choices for $D_1,\dots,D_4$, for fixed $D$, and that the sum over $k$  contributes $O(\log B)$.
Hence we find that
\begin{align*}
U(q_1,q_2;\cc)
&\ll_\ve B^{3\ve}\log B\sum_{d_1\mid r}
\sum_{\substack{
d_2\leq (SR)^{\eta/20}}}
\sum_{r[d_1,d_2^2]=DE} E^{1+\ve}
\left(
		\frac{1}{u^{\frac{1}{2}}d_2^2} 
		+				\frac{u^{\frac{1}{2}}}{S}	\right),
\end{align*}
where we recall that 
$\gcd(D,\CD)=1$ and 
$E$ is only divisible by primes dividing $\CD$.  In particular, the factorisation of 
$r[d_1,d_2^2]$ as $DE$ is uniquely determined. 
We factorise $d_2=d_2'd_2''$, where $\gcd(d_2',\CD)=1$ and $p\mid d_2'' \Rightarrow p\mid \CD$.
There are clearly $O_\ve(B^\ve)$ choices for $d_2''$. Moreover, we now have  
$E\leq rd_1(d_2'')^2\leq (rd_2'')^2$ and so we may sum over $d_2'$ and $d_2''$ to get
\begin{align*}
U(q_1,q_2;\cc)
&\ll_\ve B^{4\ve}\log B
r^{2+\ve}
\left(
		\frac{1}{u^{\frac{1}{2}}} 
		+				\frac{u^{\frac{1}{2}}(SR)^{\eta/20}}{S}	\right)
\ll_\ve r^{2+\ve} B^{5\ve}
\left(
		\frac{1}{u^{\frac{1}{2}}} 
		+				\frac{u^{\frac{1}{2}}(SR)^{\eta/20}}{S}	\right).		
\end{align*}

It remains to substitute this bound into \eqref{eq:sumr}. Recalling that 
$q_1=ut^2$, where $u$ is the largest square-free divisor of $q_1$, we observe  that 
$$
\sum_{Q_1<q_1\leq Q_2} \frac{1}{q_1u^{\frac{1}{2}}}\leq \sum_{t\leq \sqrt{Q_2}}
\frac{1}{t^2}
\sum_{u>Q_1/t^2} \frac{1}{u^{\frac{3}{2}}}\ll Q_1^{-\frac{1}{2}}\log Q_2,
$$
for any $Q_1\leq Q_2$. Similarly 
$$
\sum_{Q_1<q_1\leq Q_2} \frac{u^{\frac{1}{2}}}{q_1}\leq 
\sum_{q_1\leq Q_2} q_1^{-\frac{1}{2}}\ll Q_2^{\frac{1}{2}}.
$$
Recalling the definitions \eqref{eq:N1new}  and \eqref{eq:int-q2} of $N_1$ and $I_{q_2}$, respectively, it  follows that 
\begin{align*}
\sum_{\substack{q_1\in I_{q_2}\cap \ZZ\\ \gcd(q_1,2r)=1}} \frac{\phi^*(q_1)}{q_1} 
\left(
		\frac{1}{u^{\frac{1}{2}}} 
+				\frac{u^{\frac{1}{2}} (SR)^{\eta/20}}{S}	\right)
&\ll 
\frac{q_2^{\frac{1}{2}}\log B}{B^{\eta/20}}
+
\frac{N_1^{\frac{1}{2}}(SR)^{\eta/20}}{S}
\ll 
\frac{q_2^{\frac{1}{2}}\log B}{B^{\eta/20}}
+
\frac{B^{7\eta/100}}{S^{1-4\theta}},
\end{align*}
since $R\ll S$.
Appealing to part (i) of Lemma \ref{le:Sqxx} to estimate $S_{q_2}(\cc)$ we deduce from \eqref{eq:sumr} that 
\begin{align*}
\Sigma_r(S,R)
&\ll_\ve  r^{2+\ve}B^{2\ve}\left(
\frac{1}{B^{\eta/20}}
+
\frac{B^{7\eta/100}}{S^{1-4\theta}}\right)
\sum_{q_2\leqslant B^{\eta/100}}q_2^{2}\\
&\ll_\ve  r^{2+\ve} B^{2\ve}\left(
\frac{1}{B^{\eta/50}}
+
\frac{B^{\eta/10}}{S^{1-4\theta}}\right).
\end{align*}
This bound is valid for any choice of  $\theta>\frac{3}{16}$. Taking $\theta=\frac{1}{5}$ and recalling that $S\gg B^{2\eta}$, it therefore follows that
$$
\Sigma_r(S,R)\ll_\ve  r^{2+\ve} B^{2\ve}\left(
\frac{1}{B^{\eta/50}}
+
\frac{B^{9\eta/100}}{S^{1/5}}\right)\ll r^{2+\ve}
B^{-\eta/50+2\ve}.
$$
Summing over $r\leqslant \Theta\leq B^{\eta/1000}$ in \eqref{eq:yellow-bach} and taking $\ve$ sufficiently small, we finally conclude the proof of Lemma \ref{prop:lace}.
\end{proof}

\subsection{Proof of Proposition \ref{prop:lid}: final step}

We now have everything in place to analyse the asymptotic behaviour of 
$M(B)$, as defined in \eqref{eq:M1B}. Combining Lemmas \ref{le:smalldeltabad} and \ref{le:finiteqsum} with   Proposition \ref{prop:fksum}, we deduce that 
$$
M(B)=	\sum_{\substack{\xx\in\BZprim^2\\ B^{2\eta}\leqslant |\xx|\leqslant B^\frac{1}{4}\\ 
\deltabad(\xx)\leq B^{\eta/1000}
}}\frac{\sigma_{\infty}(\xx)\mathfrak{S}(\xx;B^{\eta/10})}{|\xx|} +O_\eta(1).
$$
The proof of Lemma \ref{le:smalldeltabad} applies in the same way to show 
that there is an  overall contribution $O_\eta(1)$ from $\xx$ such that 
$\deltabad(\xx)> B^{\eta/1000}$. 
Let 
$$
	c_q(a)=\sum_{\substack{x\bmod q\\ \gcd(x,q)=1}} \operatorname{e}_q(ax)
$$
be  the Ramanujan sum, for  $a,q\in\BN$.
Then, on opening up $\mathfrak{S}(\xx;B^{\eta/10})$ and rearranging the sums, we obtain
$$
M(B)=	
\sum_{q\leq B^{\eta/10}} q^{-4} \sum_{\mathbf{b}\in (\ZZ/q\ZZ)^4}
\sum_{\substack{\xx\in\BZprim^2\\ B^{2\eta}\leqslant |\xx|\leqslant B^\frac{1}{4}
}}c_q\left(\sum_{i=1}^4 L_i(\x)b_i^2\right)\frac{\sigma_{\infty}(\xx)}{|\xx|} +O_\eta(1).
$$
We break the $\xx$-sum into residue classes modulo $q$, leading to 
\begin{equation}\label{eq:yellow-rope}
	M(B)=\sum_{q\leqslant B^{\eta/10}}q^{-4} \sum_{\mathbf{b}\in (\ZZ/q\ZZ)^4}\sum_{\substack{\cc\in (\BZ/q\BZ)^2\\ \gcd(q,\cc)=1}}c_q\left(\sum_{i=1}^{4}L_i(\cc)b_i^2\right)
	U_q(\cc) +O_\eta(1),
\end{equation}
where
$$
U_q(\cc)=	\sum_{\substack{\xx\in\BZprim^2\\ \xx\equiv \cc\bmod q\\ B^{2\eta}\leqslant |\xx|\leqslant B^\frac{1}{4}}}\frac{\sigma_{\infty}(\xx)}{|\xx|}.
$$
The following result is concerned with the asymptotic evaluation of this sum. 

\begin{lemma}\label{lem:712}
We have 
$$
U_q(\cc)=
\frac{1}{q^2\zeta(2)}\prod_{p\mid q}\left(1-\frac{1}{p^2}\right)^{-1}
\int_{\{\mathbf{t}\in \RR^2: B^{2\eta}\leq |\mathbf{t}| \leq B^{\frac{1}{4}}\}} 
\frac{\sigma_\infty(\mathbf{t})}{|\mathbf{t}|} \mathrm{d} \mathbf{t}
 +O(B^{-\frac{6\eta}{7}}\log B),
$$
for any $\cc\in (\ZZ/q\ZZ)^2$ and $q\leq B^{\eta/10}$.
\end{lemma}

\begin{proof}
It will be convenient to define $m(\x)=\min_{1\leq i\leq 4}|L_i(\x)|$ in the proof of this result. Then, in view of Lemma \ref{le:smallvalueLi} and \eqref{eq:brown-bag}, we have the estimate
\begin{equation}\label{eq:grey-lead}
\sigma_\infty(\x)\ll |\x|^{-\frac{3}{4}}m(\x)^{-\frac{1}{4}}.
\end{equation}
Since there are no primitive vectors $\x\in \ZZ^2$ with $|x_1|=|x_2|$ and $|\x|>B^{2\eta}$, we may write
$$
U_q(\cc)=	
\sum_{\substack{\xx\in\BZprim^2,~ |x_1|<|x_2|\\ \xx\equiv \cc\bmod q\\ B^{2\eta}\leqslant |x_2|\leqslant B^\frac{1}{4}}}\frac{\sigma_{\infty}(\xx)}{|x_2|}+
\sum_{\substack{\xx\in\BZprim^2,~ |x_2|<|x_1|\\ \xx\equiv \cc\bmod q\\ B^{2\eta}\leqslant |x_1|\leqslant B^\frac{1}{4}}}\frac{\sigma_{\infty}(\xx)}{|x_1|}=U_q^{(1)}(\cc)+U_q^{(2)}(\cc),
$$
say. We focus our efforts on $U_q^{(1)}(\cc)$, the treatment of the remaining sum being identical.

We begin by handling the overall contribution to $U_q^{(1)}(\cc)$ from $\x$ such that $m(\x)\leq \delta |x_2|$, for a parameter $\delta$ that will be selected in due course, but which will tend to $0$ as $B\to \infty$.  In particular $m(\x)$ cannot be proportional to $x_2$ in this case.
Given $x_2\in \ZZ$, there are at most $O(L)$ values of $x_1\in \ZZ$ such that $m(\x)\leq L$, for any 
$L\leq \delta |x_2|$.
 Thus \eqref{eq:grey-lead} implies that 
$$
\sum_{\substack{\xx\in\BZprim^2, ~|x_1|<|x_2|\\
m(\x)\leq \delta |x_2|\\
B^{2\eta}\leqslant |x_2|\leqslant B^\frac{1}{4}}}\frac{\sigma_{\infty}(\xx)}{|x_2|}\ll   
\delta^{\frac{3}{4}}\log B.
$$
Since $\sigma_{\infty}(k\xx)=k^{-1}\sigma_{\infty}(\xx)$, for any $k>0$, we apply  M\"obius inversion to deal with  the coprimality of $\xx$, giving
$$
U_q^{(1)}(\cc)=\sum_{\substack{k\leq B^{\frac{1}{4}}\\ \gcd(k,q)=1}}
\frac{\mu(k)}{k^2}
	\sum_{\substack{\xx\in \ZZ^2,~|x_1|<|x_2|\\ \xx\equiv \bar{k}\cc\bmod q\\ B^{2\eta}\leqslant k|x_2|\leqslant B^\frac{1}{4}\\
	m(\xx)>\delta |x_2|
	}}\frac{\sigma_{\infty}(\xx)}{|x_2|} +O\left(\delta^{\frac{3}{4}}\log B\right).
$$
where $\bar k$ is the multiplicative inverse of $k$ modulo $q$.
It follows from  \eqref{eq:grey-lead} that the $\x$-sum is $O(\log B)$. Hence, the overall contribution to the main term from $k>B^{2\eta}$ is easily seen to be 
$O(B^{-2\eta}\log B)$. Hence
$$
U_q^{(1)}(\cc)=\sum_{\substack{k\leq B^{2 \eta}\\ \gcd(k,q)=1}}
\frac{\mu(k)}{k^2}
\hspace{-0.3cm}
\sum_{\substack{x_2\in \ZZ\\ x_2\equiv \bar{k} c_2\bmod{q}\\
B^{2\eta}/k\leqslant |x_2|\leqslant B^\frac{1}{4}/k}} \frac{1}{|x_2|}
	\sum_{\substack{x_1\in \ZZ\cap K_{x_2}\\ x_1\equiv \bar{k}c_1\bmod q}}\sigma_{\infty}(\xx)
	+O((B^{-2 \eta}+\delta^{\frac{3}{4}})\log B),
$$
where $K_{x_2}$ is the interval of $t\in \RR$ such that 
$|t|<|x_2|$ and $ m(t,x_2)>\delta|x_2|$.

Appealing to partial summation, together with \eqref{eq:grey-lead} and part (i) of Lemma~\ref{le:derivative},  it easily follows that 
$$
	\sum_{\substack{x_1\in \ZZ\cap K_{x_2}\\ x_1\equiv \bar{k}c_1\bmod q}}\sigma_{\infty}(\xx)
	=\frac{1}{q} \int_{K_{x_2}}
	\sigma_\infty(t_1,x_2)\mathrm{d} t_1  +O\left(E_1(x_2) \right),
$$
where
\begin{align*}
E_1(x_2)&= 
\sup_{t\in K_{x_2}}|\Delta(t,x_2)|^{-\frac{1}{4}}+
\int_{K_{x_2}} \left( |\Delta(t,x_2)|^{-\frac{1}{3}}m(t,x_2)^{-\frac{2}{3}}\right)\mathrm{d} t
\ll 
\frac{1}{|x_2| \delta}.
\end{align*}
For given $t_1$, let $I_{t_1}$ be  the interval of $t\in \RR$ cut out by the conditions  $m(t_1,t)>\delta|t|$,  $B^{2\eta}/k\leqslant |t|\leqslant B^\frac{1}{4}/k$ and $|t|> |t_1|$. We therefore obtain
\begin{equation}\label{eq:gold-bottle}
U_q^{(1)}(\cc)=\frac{1}{q}\sum_{\substack{k\leq B^{2\eta}\\ \gcd(k,q)=1}}
\frac{\mu(k)}{k^2} \int_{-B^{\frac{1}{4}}/k}^{B^{\frac{1}{4}}/k} S(t_1) \mathrm{d} t_1
	+O((B^{-2 \eta}+\delta^{\frac{3}{4}})\log B),
\end{equation}
where
$$
S(t_1)=
\sum_{\substack{x_2\in \ZZ \cap I_{t_1}\\ x_2\equiv \bar{k} c_2\bmod{q}}}
 \frac{\sigma_\infty(t_1,x_2)}{|x_2|}.
$$

We now once more use  partial summation, equipped with  \eqref{eq:grey-lead} and part (ii) of Lemma~\ref{le:derivative}. This leads to the conclusion that 
$$
S(t_1)
	=\frac{1}{q} \int_{I_{t_1}}\frac{\sigma_\infty(t_1,t_2)}{|t_2|}\mathrm{d} t_2  +O\left(E_2(t_1) \right),
$$
where
\begin{align*}
E_2(t_1)&= 
\sup_{t\in I_{t_1}}\frac{|\Delta(t_1,t)|^{-\frac{1}{4}}}{|t|}
 +\int_{I_{t_1}}
\left( \frac{|\Delta(t_1,t)|^{-\frac{1}{3}}m(t_1,t)^{-\frac{2}{3}}}{|t|} 
+\frac{|\Delta(t_1,t)|^{-\frac{1}{4}}}{|t|^2}
\right)\mathrm{d} t\\
&\ll 
\frac{1}{\max(B^{2\eta}/k, |t_1|)^2\delta}.
\end{align*}
Moreover, 
it is easily confirmed that 
$$
\int_{I_{t_1}}\frac{\sigma_\infty(t_1,t_2)}{|t_2|}\mathrm{d} t_2 =
\int_{J_{t_1}}\frac{\sigma_\infty(t_1,t_2)}{|t_2|}\mathrm{d} t_2 +O\left(
\frac{\delta^{\frac{3}{4}}}{
\max(B^{2\eta}/k, |t_1|)}
\right),
$$
where $J_{t_1}$ is defined as for $I_{t_1}$, but with the constraint 
$m(t_1,t_2)>\delta |t_2|$ removed. Hence it follows that 
$$
S(t_1)
	=\frac{1}{q} \int_{J_{t_1}}\frac{\sigma_\infty(t_1,t_2)}{|t_2|}\mathrm{d} t_2  +O\left(
\frac{1}{\max(B^{2\eta}/k, |t_1|)^2\delta}+
\frac{\delta^{\frac{3}{4}}}{
\max(B^{2\eta}/k, |t_1|)}\right).
$$
The contribution from this error term to \eqref{eq:gold-bottle} is 
\begin{align*}
&\ll \frac{1}{q} \sum_{k\leq B^{2 \eta}} \frac{1}{k^2} \left(\frac{k}{B^{2\eta}\delta}+
\delta^{\frac{3}{4}} \log B
\right)
\ll  \left(\frac{1}{B^{2\eta}\delta}+
\delta^{\frac{3}{4}} \right)\log B.
\end{align*}

An obvious  change of variables shows that 
$$
\int_{-B^{\frac{1}{4}}/k}^{B^{\frac{1}{4}}/k}
\int_{J_{t_1}}\frac{\sigma_\infty(t_1,t_2)}{|t_2|}\mathrm{d} t_2 
\mathrm{d} t_1 
= 
\int_{\{\substack{\mathbf{t}\in \RR^2: B^{2\eta}\leq |t_2| \leq B^{\frac{1}{4}}}, ~|t_2|>|t_1|\}} 
\frac{\sigma_\infty(\mathbf{t})}{|t_2|} \mathrm{d} \mathbf{t}.
$$
Hence, on returning to \eqref{eq:gold-bottle} and extending the $k$-sum to infinity, we readily obtain
\begin{align*}
U_q^{(1)}(\cc)=~&\frac{1}{q^2}\sum_{\substack{k=1\\ \gcd(k,q)=1}}^\infty
\frac{\mu(k)}{k^2} 
\int_{\{\substack{\mathbf{t}\in \RR^2: B^{2\eta}\leq |t_2| \leq B^{\frac{1}{4}}}, ~|t_2|>|t_1|\}} 
\hspace{-0.3cm}
\frac{\sigma_\infty(\mathbf{t})}{|t_2|} \mathrm{d} \mathbf{t}
	+O\left((B^{-2 \eta}\delta^{-1}+\delta^{\frac{3}{4}})\log B\right).
\end{align*}
Clearly
$$
\sum_{\substack{k=1\\ \gcd(k,q)=1}}^\infty
\frac{\mu(k)}{k^2} =\frac{1}{\zeta(2)}\prod_{p\mid q}\left(1-\frac{1}{p^2}\right)^{-1}.
$$
The statement of the lemma is now a consequence of combing this with the analogous estimate for $U_q^{(2)}(\cc)$, which follows by symmetry, and taking $\delta=B^{-\frac{8\eta}{7}}$.
\end{proof}

Before returning to our expression \eqref{eq:yellow-rope} for $M(B)$, we proceed by analysing 
the term
$$
\int_{\{\mathbf{t}\in \RR^2: B^{2\eta}\leq |\mathbf{t}| \leq B^{\frac{1}{4}}\}} 
\frac{\sigma_\infty(\mathbf{t})}{|\mathbf{t}|} \mathrm{d} \mathbf{t}=
\int_{\{\mathbf{t}\in \RR^2: 1\leq |\mathbf{t}| \leq B^{\frac{1}{4}}\}} 
\frac{\sigma_\infty(\mathbf{t})}{|\mathbf{t}|} \mathrm{d} \mathbf{t}+O(\eta \log B).
$$
Arguing as in the proof of \cite[Lemma 6.4]{duke}, it easily follows that 
$$
\int_{\{\mathbf{t}\in \RR^2: 1\leq |\mathbf{t}| \leq B^{\frac{1}{4}}\}} 
\frac{\sigma_\infty(\mathbf{t})}{|\mathbf{t}|} \mathrm{d} \mathbf{t}=\frac{1}{4}\tau_\infty \log B,
$$
where $\tau_\infty$ is defined in \eqref{eq:tau-inf}. Note that, as readily follows from 
\eqref{eq:brown-bag}, we have 
\begin{equation}\label{eq:tau-inf-upper}
\tau_\infty=O(1),
\end{equation}
for an implied constant that depends on $L_1,\dots,L_4$.

In summary, 
it follows from combining the previous calculation with \eqref{eq:yellow-rope} and Lemma \ref{lem:712} that
\begin{equation}\label{eq:train-red}
M(B)=
\frac{1}{\zeta(2)}\left(\sum_{q\leq B^{\eta/10}} \frac{A_q}{q^6}\prod_{p\mid q}\left(1-\frac{1}{p^2}\right)^{-1}
\right)
\left(\frac{\tau_\infty}{4} +O(\eta)\right) \log B
+O_\eta(1),
\end{equation}
where $A_q$ is given by \eqref{eq:defAq}. 
The final remaining task is to show that the sum over $q$ can be extended to infinity, with acceptable error.
Since $A_q$ is multiplicative in $q$, it will suffice to study it when $q$ is prime power, as in the following result.

\begin{lemma}\label{le:Phi}
	For any $r\in\BN$ and any  prime $p$, we have
	$$A_{p^r}=\phi(p^r)p^{2r}\left(\varrho(p^r)-p^2\varrho(p^{r-1})\right),
	$$
	where $\rho$ is defined in \eqref{eq:rho-q}.
\end{lemma}
\begin{proof}
We begin by observing that 
\begin{align*}
\sum_{\substack{\cc\in (\BZ/p^r\BZ)^2}}c_{p^r}\left(\sum_{i=1}^{4}L_i(\cc)b_i^2\right)
&=
\sum_{\substack{a\bmod{p^r}\\ \gcd(a,p)=1}} 
\sum_{\substack{\cc\in (\BZ/p^r\BZ)^2}}
\operatorname{e}_{p^r}\left(a(c_1Q_1(\bb)-c_2Q_2(\bb))\right)\\
&=
\begin{cases}
\phi(p^r)p^{2r} &\text{if $p^r\mid (Q_1(\bb),Q_2(\bb))$,}\\
0 &\text{ otherwise}.
\end{cases}
\end{align*}
On noting that $A_{p^r}$ can be written as the difference of sums
$$
 \sum_{\mathbf{b}\in  (\ZZ/p^r\ZZ)^4}\left(\sum_{\substack{\cc\in (\BZ/p^r\BZ)^2}}c_{p^r}\left(\sum_{i=1}^{4}L_i(\cc)b_i^2\right)- p
 \sum_{\substack{\cc\in (\BZ/p^{r-1}\BZ)^2}}c_{p^{r-1}}\left(\sum_{i=1}^{4}L_i(\cc)b_i^2\right)\right),
 $$
 the 
 lemma  readily follows.
\end{proof}

\begin{corollary}\label{cor:cor}
Let $\ve>0$ and let  $q=q_0q_1$, where $\gcd(q_0,\Delta)=1$ and $q_1\mid \Delta^\infty$. Then 
$
A_q\ll_\ve q_0^{\frac{9}{2}+\ve} q_1^{5+\ve}.
$
\end{corollary}

\begin{proof}
We have $A_q=A_{q_0}A_{q_1}$. 
Now it follows from part (i) of Lemma \ref{le:Sqxx}
and \eqref{eq:tea} that 
$$
A_{q_1}=\sum_{\substack{\mathbf{c}\in (\ZZ/q_1\ZZ)^2\\
\gcd(q_1,\mathbf{c})=1
}} S_{q_1}(\cc)\ll q_1^3
\sum_{\substack{\mathbf{c}\in (\ZZ/q_1\ZZ)^2\\
\gcd(q_1,\mathbf{c})=1
}}\gcd(q_1,\Delta(\cc))^{\frac{1}{2}}\ll_\ve q_1^3\sum_{d\mid q_1} d^{\frac{1}{2}} \cdot \left(\frac{q_1}{d}\right)^2 \cdot d^{1+\ve}.
$$
Thus $A_{q_1}=O_\ve(q_1^{5+\ve})$ on taking the trivial estimate for the divisor function. 

Turning to $A_{q_0}$, we study $A_{p^r}$ for $p\nmid \Delta$.
It follows from Lemma \ref{le:Phi} that 
$$
|A_{p^r}|\leq p^{3r}\left|\varrho(p^r)-p^2\varrho(p^{r-1})\right|,
$$
Extracting common divisors between $\y$ and $p^r$ it is easily checked that 
$$
\rho(p^r)=\sum_{0\leq k< \frac{r}{2}} p^{2k} \rho^*(p^{r-2k}) +p^{2(r-\lceil \frac{r}{2}\rceil)},
$$
in the notation of \eqref{eq:rho-q*}. (This follows from \cite[Eq.~(2.4)]{BM13}, for example.)
Since $p\nmid \Delta$, it follows from part (i) of 
 Lemma \ref{lem : loc2'} that 
$$
\rho(p^r)=p^{2r}\left(1+O(p^{-\frac{1}{2}})\right)\sum_{0\leq k< \frac{r}{2}} p^{-2k}  +p^{2(r-\lceil \frac{r}{2}\rceil)}.
$$
Similarly,
$$
p^2\rho(p^{r-1})=p^{2r}\left(1+O(p^{-\frac{1}{2}})\right)\sum_{0\leq k< \frac{r-1}{2}} p^{-2k}  +p^{2(r-1-\lceil \frac{r-1}{2}\rceil)}.
$$
Combining these, it easily follows that 
$
\varrho(p^r)-p^2\varrho(p^{r-1})\ll p^{r+1}\leq p^{\frac{3r}{2}},
$
if  $p\nmid \Delta$. The statement of the lemma follows.
\end{proof}

Taking $\ve=\frac{1}{4}$ in Corollary \ref{cor:cor}, it follows that 
\begin{align*}
\sum_{q> B^{\eta/10}} \frac{A_q}{q^6}\prod_{p\mid q}\left(1-\frac{1}{p^2}\right)^{-1}
\ll
\sum_{p\mid q_1\Rightarrow p\mid \Delta} q_1^{-\frac{3}{4}}
\sum_{q_0> B^{\eta/10}/q_1} q_0^{-\frac{5}{4}}
&\ll
B^{-\eta/40}\sum_{p\mid q_1\Rightarrow p\mid \Delta} q_1^{-\frac{1}{2}}\\
&\ll B^{-\eta/40}.
\end{align*}
In particular, this  implies that 
\begin{equation}\label{eq:SS2}
\mathfrak{S}_2=O(1),
\end{equation}
in the notation of \eqref{eq:def-SS2}.
Hence, returning to
\eqref{eq:train-red}, it now follows that 
\begin{align*}
M(B)
&=
\frac{\mathfrak{S}_2}{\zeta(2)}
\left(\frac{\tau_\infty}{4} +O(\eta)\right) \log B
+O_\eta(1),
\end{align*}
which easily leads to the statement of  Proposition \ref{prop:lid}.

\section{Comparison of the leading constants}\label{sec:finale}

In this section we complete the proof of Theorem \ref{t:main}. 
On recalling \eqref{eq:L} and \eqref{eq:cl1-3}, we see that 
$$
N_V(\Omega,B)=\frac{1}{4}\left(\#\CL_1(B)+\#\CL_2(B)+\#\CL_3(B)\right)
$$
 in \eqref{eq:step1}. We begin by analysing the main term in 
 Proposition \ref{pro:L1}.
 
 \begin{lemma}\label{lem:RHB}
Let  $Y_2>Y_1\geq 1$. Then 
$$
\int_{\substack{\y\in \RR^4\\ 
Y_1
\leq |\y|< Y_2}}
 \frac{\mathrm{d}\y}{|\y|^2\max(|Q_1(\y)|,|Q_2(\y)|)} = \tau_\infty\log \left(\frac{Y_2}{Y_1}\right),
$$
where $\tau_\infty$ is given by \eqref{eq:tau-inf}.
\end{lemma}

This result will be established  at the end of this section. 
Taking it on faith for the moment, and arguing as 
\cite[\S~6.3]{duke}, it now 
follows from the union of  Propositions \ref{pro:L1}--\ref{pro:L3}  that 
$$
N_V(\Omega,B)\sim c B\log B,
$$
as $B\to \infty$, with 
\begin{equation}\label{eq:c}
 c=\frac{\tau_\infty}{4}\left(\frac{ \mathfrak{S}_1}{4}+ \frac{ \mathfrak{S}_2}{4\zeta(2)^2}\right).
\end{equation}
The following result confirms that this agrees with Peyre's  constant 
 \cite{Peyre}, as required to complete the proof of Theorem \ref{t:main}.

\begin{proposition}\label{pro:peyre}
We have $c=c_V$, where $c_V$ is the constant  predicted by Peyre.
\end{proposition}

The constant $c_V$ has been calculated by Elsenhans \cite{Elsenhans}, but we shall give more details here. 
Let $V\subset \PP^1\times \PP^3$ be the smooth threefold
\eqref{eq:xQ}, which we view as the blow-up of $\PP^3$ along the genus $1$ curve $Z$.
The Picard group $\Pic (V)$ is generated by the hyperplane
classes
 $H_1=\pi_1^*\mathcal{O}_{\PP^1}(1)$ and  $H_2=\pi_2^*\mathcal{O}_{\PP^3}(1)$.
On the other hand,  
we saw in \S \ref{s:intro} that 
 the effective cone  of divisors $\Eff_V$ is generated by 
 $H_1$
  and the exceptional divisor $E=-H_1+2H_2$. 
  Finally, the anticanonical divisor is $-K_V= 4H_2-H_1$.
The  constant $c_V$ predicted by Peyre \cite{Peyre} then takes the shape
$$
c_V=\alpha(V) \omega_\infty(V(\RR)) \prod_p \left(1-\frac{1}{p}\right)^2\omega_p(V(\QQ_p)),
$$
where 
\begin{equation}\label{eq:cone}
\alpha(V)=\rank \left(\Pic (V)\right) \cdot \vol\left(
\x\in \Eff_V^\vee: \langle \x, -K_V\rangle\leq 1
\right) 
\end{equation}
and, for each place $v$, the measure  $\omega_v$ is the local Tamagawa measure defined by Peyre \cite{Peyre}. 
The dual of the effective cone is 
$\Eff_V^\vee = \{(t_1,t_2)\in \RR^2 : t_1\ge 0, 2t_2 - t_1\ge 0\}$,
and so the volume in \eqref{eq:cone} is
\[
\vol\{(t_1,t_2)\in \RR^2 : t_1\ge 0, ~2t_2 - t_1\ge 0, ~4t_2 - t_1 \le 1\}.
\]
This is the volume of the triangle with vertices $(0,0)$, $(0,\frac{1}{4})$, and $(1,\frac{1}{2})$, so
$\alpha(V) =2\cdot  \frac{1}{8}=\frac{1}{4}$.
The quantities $\omega_\infty(V(\RR))$ and $\omega_p(V(\QQ_p))$ have been calculated by Schindler \cite[\S~3]{Schindler}. It follows from \cite[Lemma~3.2]{Schindler} that 
$
\omega_\infty(V(\RR))=\frac{1}{2} \tau_\infty,
$
where  $\tau_\infty$ is given by \eqref{eq:tau-inf}, and from \cite[Lemma~3.1]{Schindler} that 
$$
\omega_p(V(\QQ_p))=\left(1-\frac{1}{p}\right)^{-2}\left(1-\frac{1}{p}\right)\left(1-\frac{1}{p^2}\right)\tau_p,
$$
where
$$\tau_p=\lim_{t\to\infty}p^{-5t}\#\left\{(\xx,\yy)\in(\BZ/p^t\BZ)^6: L_1(\xx)y_1^2+\cdots + L_4(\xx)y_4^2\equiv 0\bmod p^t\right\}.
$$
In this way, we deduce that 
\begin{equation}\label{eq:cVP}
c_V=\frac{1}{8}\tau_\infty \prod_{p}(1-p^{-1})(1-p^{-2})\tau_p.
\end{equation}

At first glance, it is not perhaps clear that the Euler product converges in \eqref{eq:cVP}.
However,  Elsenhans gives an explicit formula for $\tau_p$ when $p\nmid \Delta$.
	Let $E$ be the elliptic curve cut out by the equation $y^2=\prod_{i=1}^4 L_i(x_1,x_2)$ in 
	$\PP(2,1,1)$.  Let $T_p(E)=p+1-\#C(\BF_p)$ be the Frobenian trace of $E$. Then it follows from  \cite[\S~3.1]{Elsenhans} that 
$$
\tau_p=\left(1+\frac{1}{p}\right)^{-1}\left(1
+\frac{2}{p}
-\frac{T_p(E)-2}{p^2}+\frac{1}{p^3}\right).
$$	
The Hasse--Weil bound gives $|T_p(E)|\leqslant 2\sqrt{p}$. 
Thus 
$$
\prod_{p\nmid \Delta}
(1-p^{-1})(1-p^{-2})\tau_p=
\prod_{p\nmid \Delta}
\left(1-\frac{1}{p}\right)^2
\left(1
+\frac{2}{p} +O\left(\frac{1}{p^{\frac{3}{2}}}\right)
\right),
$$
which is clearly  convergent.

For any prime $p$, we may write $\tau_p=\lim_{t\to \infty} p^{-5t}n(p^t)$, where
$$
n(p^t)=\#\left\{(\xx,\yy)\in(\BZ/p^t\BZ)^6: L_1(\xx)y_1^2+\cdots + L_4(\xx)y_4^2\equiv 0\bmod p^t\right\}.
$$
The following result provides a convenient formula for this quantity.

\begin{lemma}\label{lem:npt}
For any prime power $p^t$, we  have 
$$
p^{-5t} n(p^t)=1+\left(1-\frac{1}{p}\right)\sum_{j=1}^{t}\frac{\varrho(p^j)}{p^{3j}},
$$
where $\rho(p^t)$ is defined in \eqref{eq:rho-q}. In particular,
$$
\tau_p=1+\left(1-\frac{1}{p}\right)\sum_{j=1}^{\infty}\frac{\varrho(p^j)}{p^{3j}}.
$$
\end{lemma}

\begin{proof}
On  recalling the definition  \eqref{eq:rho-q} of $\rho$, we may write
\begin{align*}
	n(p^t)&=\#\{(\xx,\yy)\in(\BZ/p^t\BZ)^6:x_1Q_1(\yy)\equiv x_2Q_2(\yy)\bmod p^t\}\\ 
	&=\varrho(p^t)p^{2t}+\sum_{j=0}^{t-1}\sum_{\substack{\yy\bmod p^t\\ p^j\|\gcd(Q_1(\yy),Q_2(\yy))}}
	\hspace{-0.4cm}\#\{\xx\bmod p^t:x_1Q_1(\yy)\equiv x_2Q_2(\yy)\bmod p^t\}.
\end{align*}
For each $0\leqslant j\leqslant t-1$ and for each $\yy$ in the sum, any $\xx\bmod p^t$ to be counted must satisfy
$x_1p^{-j}{Q_1(\yy)} \equiv x_2p^{-j}{Q_2(\yy)}\bmod p^{t-j}.$
Since $p\nmid p^{-j}\gcd(Q_1(\yy), Q_2(\yy))$, the number of such $\xx\bmod p^{t-j}$ is $p^{t-j}$, giving 
$p^{t-j}\cdot p^{2j}=p^{t+j}$ values of  $\xx\bmod p^t$. 
Moreover, 
\begin{align*}
\sum_{\substack{\yy\bmod p^t\\ p^j\|\gcd(Q_1(\yy),Q_2(\yy))}} 1&=
\sum_{\substack{\yy\bmod p^t\\ p^j\mid \gcd(Q_1(\yy),Q_2(\yy))}}1-
\sum_{\substack{\yy\bmod p^t\\ p^{j+1}\mid\gcd(Q_1(\yy),Q_2(\yy))}}1\\
&=
p^{4(t-j)}\varrho(p^j)-p^{4(t-j-1)}\varrho(p^{j+1}).
\end{align*}
It therefore follows that 
\begin{align*}
	n(p^t)&=\varrho(p^t)p^{2t}+\sum_{j=0}^{t-1}p^{t+j}\left(p^{4(t-j)}\varrho(p^j)-p^{4(t-j-1)}\varrho(p^{j+1})\right)\\&=\varrho(p^t)p^{2t}+\sum_{j=0}^{t-1}p^{5t-3j}\left(\varrho(p^j)-p^{-4}\varrho(p^{j+1})\right),
\end{align*}
whence
\begin{align*}
	\frac{n(p^t)}{p^{5t}}&=\frac{\varrho(p^t)}{p^{3t}}+\sum_{j=0}^{t-1}\left(\frac{\varrho(p^j)}{p^{3j}}-\frac{1}{p}\frac{\varrho(p^{3(j+1)})}{p^{3(j+1)}}\right) =1+\left(1-\frac{1}{p}\right)\sum_{j=1}^{t}\frac{\varrho(p^j)}{p^{3j}}.
\end{align*}
as claimed.
\end{proof}

We now turn to the Euler product $\mathfrak{S}_1$
defined in 
\eqref{eq:series}, writing $ \mathfrak{S}_1=\prod_{p}\lambda_p$, say. 
The following result 
 confirms that the local factor
$\lambda_p$ matches the corresponding local factor in  
Peyre's constant 
\eqref{eq:cVP}.

\begin{lemma}\label{lem:8.3}
For any prime $p$, 
we have 
$
\lambda_p=
(1-p^{-1})(1-p^{-2})\tau_p.
$
\end{lemma}

\begin{proof}
Let $p$ be a prime.
We have
$$
\lambda_p=
\left(1-\frac{1}{p}\right)
\left(1-\frac{1}{p^4}+\left(1-\frac{1}{p}\right)^2\sum_{a=1}^{\infty}\frac{\#V_{p^{a}}^{\times}}{p^{2a}}
\right),
$$
where  $V_{p^a}^{\times}$ is given by
\eqref{eq:def-Vd}.  We observe that 
$$
\varrho^*(p^{b})=
\begin{cases}
\varrho(p)-1								&\text{ if $b=1$}\\
\varrho(p^{b})-p^{4}\varrho(p^{a-2}) 		&\text{ otherwise}.
\end{cases}
$$
Hence it follows from  \eqref{eq:stone}  that
\[
\begin{split}
\sum_{a=1}^{\infty}\frac{\#V_{p^{a}}^{\times}}{p^{2a}}&=\left(1-\frac{1}{p}\right)^{-1}\sum_{a=1}^{\infty}\frac{\varrho^{*} (p^{a})}{p^{3a}}\\
&=\left(1-\frac{1}{p}\right)^{-1}\left(\left(1-\frac{1}{p^{2}}\right)\sum_{a=1}^{\infty}\frac{\varrho (p^{a})}{p^{3a}}-\frac{1}{p^{2}}\left(1+\frac{1}{p}\right)\right)\\
&=\left(1+\frac{1}{p}\right)\left(\sum_{a=1}^{\infty}\frac{\varrho (p^{a})}{p^{3a}}-\frac{1}{p(p-1)}\right).
\end{split}
\]
Thus 
$\lambda_p=(1-\frac{1}{p})(1-\frac{1}{p^{2}})\lambda_{p}'$,
where
\[
\begin{split}
\lambda_{p}'=1+\frac{1}{p^{2}}+\left(1-\frac{1}{p}\right)
\sum_{a=1}^{\infty}\frac{\varrho(p^{a})}{p^{3a}}
-\frac{1-\frac{1}{p}}{p(p-1)}
&=1+\left(1-\frac{1}{p}\right)\sum_{a=1}^{\infty}\frac{\varrho(p^{a})}{p^{3a}}.
\end{split}
\]
Lemma \ref{lem:npt} confirms that the right hand side is 
$\tau_p$.
\end{proof}

It remains to examine the second term in \eqref{eq:c}. 
We recall that 
$$
\mathfrak{S}_2=\sum_{q=1}^\infty \frac{A_q}{q^6}\prod_{p\mid q}\left(1-\frac{1}{p^{2}}\right)^{-1}
$$
where   $A_q$  is defined in \eqref{eq:defAq}. 
Since $A_q$ is a multiplicative function of $q$, we can  represent the $q$-sum as an Euler product, finding that 
$$
\frac{\mathfrak{S}_2}{\zeta(2)^2}=
\prod_p 
\left(1-\frac{1}{p^2}\right)^2
\left(1+\left(1-\frac{1}{p^{2}}\right)^{-1}\sum_{r\geq 1} \frac{A_{p^r}}{p^{6r}}\right).
$$
We may now record the following result.

\begin{lemma}\label{lem:8.4}
For any prime $p$, 
we have 
$$
\left(1-\frac{1}{p^2}\right)^2\left(1+\left(1-\frac{1}{p^{2}}\right)^{-1}\sum_{r\geq 1} \frac{A_{p^r}}{p^{6r}}\right)=
(1-p^{-1})(1-p^{-2})\tau_p.
$$
\end{lemma}
\begin{proof}
We need to prove that 
$$
\left(1-\frac{1}{p^2}\right)\left(1+\left(1-\frac{1}{p^{2}}\right)^{-1}\sum_{r\geq 1} \frac{A_{p^r}}{p^{6r}}\right)=
(1-p^{-1})\tau_p.
$$
But 
Lemma \ref{le:Phi} implies that the left hand side is 
\begin{align*}
1-\frac{1}{p^2}+\sum_{r\geq 1} \frac{A_{p^r}}{p^{6r}}
&=1-\frac{1}{p^2}+ \left(1-\frac{1}{p}\right)^2\sum_{r=1}^\infty \frac{\rho(p^r)}{p^{3r}}-\frac{1}{p}\left(1-\frac{1}{p}\right)\\
&=\left(1-\frac{1}{p}\right) \left(1+\left(1-\frac{1}{p}\right)\sum_{r=1}^\infty \frac{\rho(p^r)}{p^{3r}}
\right).
\end{align*}
The desired equality now follows from  Lemma \ref{lem:npt}.
\end{proof}

Combining Lemmas 
\ref{lem:8.3} and 
\ref{lem:8.4} in \eqref{eq:c}, we therefore conclude that $c=c_V$, as claimed in Proposition \ref{pro:peyre}, subject to the verification of Lemma \ref{lem:RHB}.

\begin{proof}[Proof of Lemma \ref{lem:RHB}]
Let $\bfy\in [-1,1]^4$ and define
$$
\rho_\infty(\y)=
\int_{-\infty}^\infty\int_{[-1,1]^{2}}\e\left(\theta(x_{1}Q_{1}(\bfy)-x_{2}Q_{2}(\bfy))\right)\mathrm{d}\bfx\mathrm{d} \theta.
$$
Note that   $\max(|Q_{1}(\bfy)|,|Q_{2}(\bfy)|)>0$, since $Z(\RR)= \emptyset$. Hence
\[
\begin{split}
\int_{[-1,1]^{2}}\e\left(\theta(x_{1}Q_{1}(\bfy)-x_{2}Q_{2}(\bfy))\right)\mathrm{d}\bfx&=
\prod_{i=1,2} \int_{-1}^{1}\e\left(\theta x Q_{i}(\bfy)\right)\mathrm{d}x 
=
\prod_{i=1,2}\frac{\sin(2\pi \theta |Q_{i}(\bfy)|)}{\pi\theta |Q_{i}(\bfy)|}.
\end{split}
\]
Hence it follows from \cite[\S~3.741]{g-r} that
\[
\begin{split}
\rho_\infty(\y)=
\frac{1}{\pi^{2}|Q_{1}(\bfy)||Q_{2}(\bfy)|}\int_{-\infty}^{\infty}\frac{\sin(2\pi \theta |Q_{1}(\bfy)|)\sin(2\pi \theta |Q_{2}(\bfy)|)}{\theta^{2}}\mathrm d\theta 
&=\frac{2\pi^{2}\min_{i=1,2}|Q_{i}(\bfy)|}{\pi^{2}|Q_{1}(\bfy)||Q_{2}(\bfy)|}\\&=\frac{2}{\max_{i=1,2}|Q_{i}(\bfy)|}.
\end{split}
\]
Let  $Y_2>Y_1\geq 1$.  We may now conclude that 
$$
\int_{\substack{\y\in \RR^4\\ 
Y_1
\leq |\y|< Y_2}}
 \frac{\mathrm{d}\y}{|\y|^2\max(|Q_1(\y)|,|Q_2(\y)|)}=
 \frac{1}{2}
 \int_{\substack{\y\in \RR^4\\ 
Y_1
\leq |\y|< Y_2}}
 \frac{\rho_\infty(\y)}{|\y|^2
 }\mathrm{d}\y.
 $$
 Let us first consider the contribution from $\y$ for which $|\y|=|y_4|$. 
 Writing 
$t_i=y_i/|y_4|$ for $1\leq i\leq 3$, we obtain the contribution
$$
 \int_{
Y_1}^{Y_2}
\frac{\mathrm{d} y_4}{y_4} \int_{[-1,1]^3} 
\rho_\infty(t_1,t_2,t_3,1)\mathrm{d}\mathbf{t}=
\log\left(\frac{Y_2}{Y_1}\right)  \int_{[-1,1]^3} 
\rho_\infty(t_1,t_2,t_3,1)\mathrm{d}\mathbf{t},
$$
where $\mathbf{t}=(t_1,t_2,t_3)$.
On adding in the remaining three 	contributions, and observing that 
$$
\int_{[-1,1]^4}\rho_\infty (\y )\mathrm{d}\y= \int_{[-1,1]^3} 
\rho_\infty(t_1,t_2,t_3,1)\mathrm{d}\mathbf{t}
+\cdots+
\int_{[-1,1]^3} 
\rho_\infty(1,t_2,t_3,t_4)\mathrm{d}\mathbf{t},
$$
the statement of the lemma easily follows.
\end{proof}


\begin{thebibliography}{99}


\bibitem{BT} V. Batyrev and Yu. Tschinkel, Manin's conjecture for
   toric varieties.
\emph{J.\  Alg.\ Geom.} {\bf 7} (1998),  15--53.

	\bibitem{blomer}
V. Blomer, J. Br\"udern, U. Derenthal and G. Gagliardi,
The Manin--Peyre conjecture for smooth spherical Fano varieties of semisimple rank one.
{\em Preprint}, 2020. (arXiv:2004.09357)

	\bibitem{blomer2}
V. Blomer, J. Br\"udern, U. Derenthal and G. Gagliardi,
The Manin--Peyre conjecture for smooth spherical Fano threefolds
{\em Preprint}, 2022. (arXiv:2203.14841)



	
	\bibitem{BrowningM} T.D. Browning,  {Density of integer solutions to diagonal quadratic forms.} {\em Monatshefte Math.} {\bf 152} (2007), 13--38.
	
	
	\bibitem{DA} T.D. Browning and  D.R. Heath-Brown,
	Counting rational points on quadric surfaces. {\em Discrete Analysis} 2018:15, 29 pp.

	\bibitem{duke} T.D. Browning and  D.R. Heath-Brown, Density of rational points on a quadric bundle in $\BP^3\times\BP^3$. {\em Duke Math.\ J.} {\bf 169} (2020), 3099--3165.
	
	 \bibitem{BM13} T.D. Browning and R. Munshi,
Rational points on singular intersections of quadrics.
{\em Compos.\ Math.} {\bf 149}  (2013), 1457--1494.



\bibitem{burgess2}
D.A. Burgess, On character sums and $L$-series. II. {\em  Proc. London Math.
Soc.} {\bf 13} (1963), 524--536.


\bibitem{ct}
A. Chambert-Loir and Yu. Tschinkel,
On the distribution of points of bounded height 
on equivariant compactifications of vector groups.
{\em  Invent. Math.} {\bf 148} (2002), 421--452.
		
	\bibitem{Elsenhans} A.-S. Elsenhans, {Rational points on some Fano quadratic bundles.} {\em Exp.\ Math.} {\bf 20} (2011), 373--379.
	
	\bibitem{fmt}
{J. Franke, Y.I. Manin and Y. Tschinkel}, {Rational
points of bounded height on Fano varieties}. {\em Invent.\ Math.} {\bf95} (1989), 421--435.


\bibitem{g-r}
I.S. Gradshteyn and I.M. Ryzhik,
{\em Table of integrals, series and products.}
7th ed., Academic Press, 2007.

	
	\bibitem{H-B} D.R. Heath-Brown, Diophantine approximation with square-free numbers. {\em  Math.\ Z.} {\bf 187} (1984), 335--344.


 \bibitem{HB02} D.R. Heath-Brown,
The density of rational points on curves and surfaces.
{\em Annals of Math.} {\bf 155}, (2013), 553--598.

\bibitem{HM}
Z. Huang and P. Montero,
 Fano threefolds as equivariant compactifications of the vector group.
{\em  Michigan Math.\ J.} {\bf  69} (2020), 341--368. 

\bibitem{isk} V.A. Iskovskikh and Yu.G. Prokhorov, \emph{Fano varieties.} In Algebraic geometry
		V: Fano varieties. Springer, Berlin, 1999.


%
\bibitem{brian2}
B. Lehmann and S. Tanimoto, On exceptional sets in Manin’s conjecture. {\em Res.\ Math.\ Sci.} {\bf 6} (2019),  article~12.

\bibitem{manin}
Yu.I. Manin, 
Notes on the arithmetic of Fano threefolds.
{\em Compositio Math.} {\bf 85} (1993), 37--55.
 
 \bibitem{MM}
 S. Mori and S. Mukai, Classification of Fano 3-folds with $B_2\geq 2$. {\em Manuscripta Math.} {\bf 36} (1981/82), 147--162.
 
 		\bibitem{Ottem}  J.C. Ottem, Birational geometry of hypersurfaces in products of projective spaces. 
		{\em Math. Z.} {\bf 280} (2015), 135--148. 

 
 	\bibitem{Peyre} E. Peyre,  Hauteurs et mesures de Tamagawa sur les variétés de Fano. {\em Duke Math.\ J.} {\bf 79} (1995), 101--218.


\bibitem{Schindler}  D. Schindler, Manin’s conjecture for certain biprojective hypersurfaces. {\em J. reine angew. Math.} \textbf{714} (2016), 209–250.


\bibitem{Sc68} W.M. Schmidt,
Asymptotic formulae for point lattices of bounded determinant and subspaces of bounded height.
{\em Duke Math.\ J.} {\bf 35} (1968), 327--339.

\bibitem{Ser08} {J.-P.~Serre}, {\em Topics in Galois theory}.
  Second edition.	Research Notes in
  Mathematics~{\bf  1}, A K Peters, Ltd., Wellesley, MA, 2008. 

\bibitem{skoro}
A. Skorobogatov, {\em Torsors and rational points}.  Cambridge Tracts in Mathematics {\bf 144},  
CUP, 2010.

\bibitem{sho}
S. Tanimoto, On upper bounds of Manin type.
{\em Algebra \& Number Theory}  {\bf 14} (2020), 731--762.

 \bibitem{Te15} G. Tenenbaum,
\emph{Introduction to analytic and probabilistic number theory}.
Grad. Studies  Math.\ {\bf 163}, Springer, 2015.

\bibitem{weil}
A. Weil, On some exponential sums. {\em Proc. Nat.\ Acad.\ Sci.\ USA} {\bf 34} (1948), 204--207.

 \end{thebibliography}
\end{document}